\documentclass[11pt,reqno]{amsart}
\usepackage{amssymb,mathrsfs,color}
\usepackage{bbm}

\usepackage{enumitem}
\usepackage{enumerate}
\usepackage{graphicx}
\usepackage{xcolor} 
\usepackage[letterpaper, margin=1in]{geometry}
\usepackage{amsfonts,amssymb,amsmath}
\usepackage{slashed}

\usepackage[titletoc,title]{appendix}
\usepackage{wrapfig}

\usepackage{cancel}

\usepackage{cite}

\usepackage{nth}

\usepackage{marginnote}

\usepackage{pgfplots}

\usepackage{verbatim}

\usepackage{dsfont}

\definecolor{green}{rgb}{0,0.8,0} 

\definecolor{deepgreen}{cmyk}{1,0,1,0.5}

\newcommand{\EQ}[1]{\begin{equation}\begin{split} #1 \end{split}\end{equation}}

\newcommand{\Del}[1]{}

\numberwithin{equation}{section}

\newtheorem{theorem}{Theorem}[section]
\newtheorem{corollary}[theorem]{Corollary}
\newtheorem{lemma}[theorem]{Lemma}
\newtheorem{proposition}[theorem]{Proposition}
\newtheorem{remark}[theorem]{Remark}
\newtheorem{definition}[theorem]{Definition}

\newcommand{\jap}[1]{\langle #1\rangle}

\renewcommand{\Re}{\mathrm{Re}}
\renewcommand{\Im}{\mathrm{Im}}

\renewcommand{\hbar}{{\underline h}}

\newcommand{\bbC}{\mathbb C}

\newcommand{\bbR}{\mathbb R}

\newcommand{\bbZ}{\mathbb Z}

\newcommand{\calE}{\mathcal E}
\newcommand{\calF}{\mathcal F}

\newcommand{\calI}{\mathcal I}

\newcommand{\calL}{\mathcal L}

\newcommand{\calO}{\mathcal O}

\newcommand{\calS}{\mathcal S}

\newcommand{\wtilg}{{\widetilde{g}}}

\newcommand{\wtilD}{{\widetilde{D}}}

\newcommand{\wtilY}{{\widetilde{Y}}}

\newcommand{\px}{\partial_x}

\newcommand{\pt}{\partial_t}
\newcommand{\ps}{\partial_s}

\newcommand{\jt}{\jap{t}}
\newcommand{\js}{\jap{s}}
\newcommand{\jx}{\jap{x}}
\newcommand{\jxi}{\jap{\xi}}

\newcommand{\tD}{\widetilde{D}}
\newcommand{\jtD}{\jap{\widetilde{D}}}
\newcommand{\hf}{\frac{1}{2}}
\newcommand{\thf}{\frac{3}{2}}

\newcommand{\wtcalF}{\widetilde{\calF}}

\newcommand{\bR}{{\mathbb R}}

\newcommand{\ud}{\mathrm{d}}

\DeclareMathOperator{\sech}{sech}

\def\R{\mathbb{R}}
\def\eps{\varepsilon}
\def\nn{\nonumber}
\def\ol{\overline}
\def\vphi{\varphi}
\def\les{\lesssim}
\def\one{\mathbbm{1}}

\def\les{\lesssim} 
\def\calL{\mathcal{L}}
\def\calS{\mathcal{S}}
\def\tD{\widetilde{D}}
\def\Id{\mathrm{Id}}

\begin{document}

\title[On modified scattering for 1D quadratic KG equations with non-generic potentials]{On modified scattering for 1D quadratic Klein-Gordon equations with non-generic potentials}

\author[H. Lindblad]{Hans Lindblad}
\address{Department of Mathematics \\ Johns Hopkins University \\ Baltimore, MD 21218, USA}
\email{lindblad@math.jhu.edu}

\author[J. L\"uhrmann]{Jonas L\"uhrmann}
\address{Department of Mathematics \\ Texas A\&M University \\ College Station, TX 77843, USA}
\email{luhrmann@math.tamu.edu}

\author[W. Schlag]{Wilhelm Schlag}
\address{Department of Mathematics \\ Yale University \\ New Haven, CT 06511, USA}
\email{wilhelm.schlag@yale.edu}

\author[A. Soffer]{Avy Soffer}
\address{Mathematics Department, Rutgers University, New Brunswick, NJ 08903, USA}
\email{soffer@math.rutgers.edu}

\thanks{
H. Lindblad was partially supported by NSF grant DMS-1500925 and by Simons Foundation Collaboration grant 638955. 
J. L\"uhrmann was partially supported by NSF grant DMS-1954707. 
A. Soffer was partially supported by NSF grant DMS-1600749 and by NSFC11671163. 
W. Schlag was partially supported by NSF grant DMS-1902691.
}

\begin{abstract}
 We consider the asymptotic behavior of small global-in-time solutions to a 1D Klein-Gordon equation with a spatially localized, variable coefficient quadratic nonlinearity and a non-generic linear potential. 
 The purpose of this work is to continue the investigation of the occurrence of a novel modified scattering behavior of the solutions that involves a logarithmic slow-down of the decay rate along certain rays.
 This phenomenon is ultimately caused by the threshold resonance of the linear Klein-Gordon operator.
 It was previously uncovered for the special case of the zero potential in~\cite{LLS2}.
 The Klein-Gordon model considered in this paper is motivated by the asymptotic stability problem for kink solutions arising in classical scalar field theories on the real line. 
\end{abstract}

\maketitle 

\tableofcontents

\section{Introduction}

We study the long-time behavior of small global-in-time solutions to the Cauchy problem for the following $(1+1)$-dimensional Klein-Gordon equation
\begin{equation} \label{equ:intro_nlkg}
 \left\{ \begin{aligned}
          (\partial_t^2 - \partial_x^2 + m^2 + V(x)) u &= P_c \bigl( \alpha(\cdot) u^2 \bigr)  \text{ on } \bbR^{1+1}, \\
          (u, \partial_t u)|_{t=0} &= (P_c u_0, P_c u_1),
         \end{aligned} \right.
\end{equation}
where the potential $V(x)$ and the variable coefficient $\alpha(x)$ are sufficiently smooth and decaying, where $m > 0$ is the mass parameter, and where the real-valued initial data $(u_0, u_1)$ are small in weighted Sobolev spaces. As a core assumption in this paper, we suppose that the Schr\"odinger operator $H = -\px^2 + V(x)$ exhibits a zero energy resonance, i.e., a non-trivial bounded solution of $H \varphi = 0$ that approaches $1$ as $x \to \infty$ and a non-zero constant as $x \to -\infty$, see Definition~\ref{def:reson}. In other words, we assume that the potential $V(x)$ is non-generic. The projection onto the continuous spectral subspace of $L^2(\bbR)$ relative to $H$ is denoted by $P_c$.

The goal of this work is to continue the investigation of the occurrence of a novel modified scattering behavior of small solutions to~\eqref{equ:intro_nlkg} that features a logarithmic slow-down of the free decay rate along certain rays. This phenomenon was recently discovered in~\cite{LLS2} in the special case $V(x) = 0$ and is ultimately caused by the threshold resonance of the linear operator $-\px^2 + m^2 + V(x)$. In this regard it is worth to record that a peculiar feature of the Laplacian in one space dimension -- in contrast to higher odd space dimensions -- is that it possesses a zero energy resonance, namely the constant function $1$.
We also refer to the beginning of Section~\ref{sec:spectral_scattering_theory} for precise definitions of some of the spectral theory terminology used in this introduction.

\subsection{Motivation}

Our interest in the model~\eqref{equ:intro_nlkg} stems from the asymptotic stability problem for kink solutions arising in classical scalar field theory models on the real line. 
Kinks are special soliton solutions to scalar field equations 
\begin{equation} \label{equ:EL_equation}
 (\pt^2 - \px^2) \phi = - W'(\phi) \text{ on } \bbR^{1+1},
\end{equation}
where $W \colon \bbR \to [0, \infty)$ is a sufficiently regular scalar potential that features a double-well, i.e., there exist (at least) two consecutive (global) minima $\phi_-, \phi_+ \in \bbR$ of $W$ with $\phi_- < \phi_+$, $W(\phi_\pm) = W'(\phi_\pm) = 0$, and $W''(\phi_\pm) > 0$.
Trivial solutions to~\eqref{equ:EL_equation} are given by the constant functions $\phi(t, x) = \phi_\pm$ for all $t \in \bbR$. Correspondingly, $\phi_-$ and $\phi_+$ are referred to as vacuum solutions. 
A static solution $\psi(x)$ to~\eqref{equ:EL_equation} that connects the two consecutive vacuua $\phi_-$ and $\phi_+$ is called a kink and satisfies
\begin{equation} \label{equ:static_kink_equation}
 \left\{ \begin{aligned}
  &\px^2 \psi = W'(\psi) \text{ on } \bbR, \\
  &\lim_{x \to \pm \infty} \psi(x) = \phi_\pm.
 \end{aligned} \right.
\end{equation}
Solutions to~\eqref{equ:static_kink_equation} are unique up to spatial translations. Moreover, the Lorentz invariance of~\eqref{equ:EL_equation} gives rise to moving kinks upon applying a Lorentz boost. 

Kinks are simple one-dimensional examples of topological solitons, see e.g.~\cite{MantSut04, Vachaspati06, Lamb80, DauxPey10}.
A fundamental question related to the dynamics of kinks is their asymptotic stability under small perturbations. A perturbative approach to this problem generally consists in decomposing the perturbed solution into the sum of a modulated kink, possibly discrete modes, and a dispersive remainder term. One then studies the long-time dynamics of the associated sytem of ODEs and PDEs. One of the key steps in that analysis is to conclude that the dispersive remainder term decays to zero in a suitable sense.
For concreteness, we now take a closer look at what this part of the problem entails for two prime examples of classical scalar field models on the real line, namely the $\phi^4$ model with
\begin{align*}
 W_{\phi^4}(\phi) &:= \frac14 (1-\phi^2)^2, \qquad \psi_{\phi^4}(x) = \tanh( {\textstyle \frac{x}{\sqrt{2}} } ), 
\end{align*}
and the sine-Gordon model\footnote{The sine-Gordon model is completely integrable and the study of its dynamics is therefore amenable to inverse scattering techniques, see, e.g., the recent work~\cite{CLL20}.} with
\begin{align*}
 W_{sG}(\phi) &:= 1 - \cos(\phi), \qquad \psi_{sG}(x) = 4 \arctan( e^x ).
\end{align*}
To simplify matters, we do not take into account any modulational aspects. 
For perturbations of the static kink $\psi_{\phi^4}(x)$ in the $\phi^4$ model, the remainder term $u(t,x) = \phi(t,x) - \psi_{\phi^4}(x)$ satisfies 
\begin{equation} \label{equ:phi4_perturbation}
 \bigl( \partial_t^2 - \partial_x^2 + 2 - 3 \sech^2( {\textstyle \frac{x}{\sqrt{2}} }) \bigr) u = - 3 \tanh( {\textstyle \frac{x}{\sqrt{2}} } ) u^2 - u^3,
\end{equation}
while for perturbations of the static kink $\psi_{sG}(x)$ in the sine-Gordon model, the remainder term $u(t,x) = \psi(t,x) - \psi_{sG}(x)$ is a solution to 
\begin{equation} \label{equ:sineGordon_perturbation}
 \begin{aligned}
  ( \partial_t^2 - \partial_x^2 + 1 - 2 \sech^2(x) ) u = - \sech(x) \tanh(x) u^2 + \bigl( {\textstyle \frac{1}{6}} - {\textstyle \frac{1}{3}} \sech^2(x) \bigr) u^3 + \bigl\{\text{higher order}\bigr\}.   
 \end{aligned}
\end{equation}
The study of the decay and the asymptotics of small solutions to 1D Klein-Gordon equations such as~\eqref{equ:phi4_perturbation} and~\eqref{equ:sineGordon_perturbation} encompasses several difficulties:
Due to the slow dispersive decay of Klein-Gordon waves in one space dimension, the quadratic and cubic nonlinearities cause long-range effects. In particular, subtle resonance phenomena can occur in the interactions in the (variable coefficient) quadratic nonlinearities. Moreover, the linearized operators may exhibit threshold resonances and may have internal modes, i.e., positive gap eigenvalues below the continuous spectrum. The latter are in fact an obstruction to decay at the linear level.
We note that the linearized operators for the $\phi^4$ model and the sine-Gordon model both exhibit threshold resonances, and that the linearized operator for the $\phi^4$ model additionally features an internal mode. 

Delicate resonance phenomena in the quadratic nonlinearities in 1D Klein-Gordon models such as~\eqref{equ:phi4_perturbation} and~\eqref{equ:sineGordon_perturbation} may lead to novel types of modified scattering behaviors of the solutions that are deeply related to the presence of a threshold resonance in the linearized operator. The purpose of this work is to uncover a precise picture of such behavior for the simplified Klein-Gordon model~\eqref{equ:intro_nlkg}, building on the recent analysis of the flat case $V(x) = 0$ in~\cite{LLS2}.

\subsection{Previous results}

The study of the asymptotic stability of kinks and of the asymptotics of solutions to nonlinear Klein-Gordon equations is a fascinating and vast subject that cannot be reviewed in its entirety here. 
In this subsection we give an overview of previous works that are closely related to the contents of this paper.

\medskip 

We begin with a brief review of orbital and asymptotic stability results for kinks. The orbital stability of kinks for general scalar field models was studied in the classical work of Henry-Perez-Wreszinski~\cite{HPW82}. 
In~\cite{KK11_1, KK11_2} Komech-Kopylova proved the asymptotic stability of kinks with respect to a weighted energy norm for a class of scalar field models with a certain flatness assumption on the potential near the wells and under suitable spectral assumptions (no resonances, presence of an internal mode). 
Kowalczyk-Martel-Mu\~{n}oz \cite{KMM17} established the asymptotic stability of the kink of the $\phi^4$ model locally in the energy space under odd finite energy perturbations. In this regard, Delort-Masmoudi~\cite{DelMas20} very recently obtained long-time dispersive estimates for odd weighted perturbations of the kink of the $\phi^4$ model up to times $T \sim \varepsilon^{-4+c}$, for arbitrary $c>0$, where $\varepsilon$ is the size of the initial data in a weighted Sobolev space. 
A sufficient condition for the asymptotic stability locally in the energy space of (moving) kinks in general $(1+1)$-scalar field models under arbitrary small finite energy perturbations has been introduced by Kowalczyk-Martel-Mu\~{n}oz-Van den Bosch~\cite{KMMV20}. 
Interestingly, the asymptotic stability properties of the kink of the sine-Gordon model hinge delicately on the topology with respect to which the perturbations are measured. The existence of special periodic solutions called wobbling kinks are an obstruction to asymptotic stability in the energy space, see for instance Alejo-Mu\~{n}oz-Palacios~\cite{AMP20} for a discussion. However, the sine-Gordon kink is asymptotically stable under sufficiently strongly weighted perturbations, as has recently been shown by Chen-Liu-Lu~\cite{CLL20} by relying on the complete integrability of the model and using the nonlinear steepest descent method. 
We also refer to the survey~\cite{KMM17_1} and to references therein. 

\medskip 

Next, we give a survey of results on the dispersive decay and the asymptotics of small solutions to one-dimensional Klein-Gordon equations with an eye towards Klein-Gordon models that are related to the asymptotic stability problem for kinks.
We note that the investigation of the long-time behavior of small solutions to Klein-Gordon equations with constant coefficient nonlinearities (in higher space dimensions) originates in the pioneering works of Klainerman~\cite{Kl80, Kl85} and Shatah~\cite{Sh85}.

Due to the slow decay of Klein-Gordon waves in one space dimension, quadratic and cubic nonlinearities exhibit long-range effects. Specifically, Delort~\cite{Del01, Del06}\footnote{We point out that the results of~\cite{Del01, Del06} pertain to more general quasilinear nonlinearities. With an eye towards the asymptotic stability problem for kinks, here we emphasize the applicability of~\cite{Del01, Del06} to the displayed Klein-Gordon model~\eqref{equ:intro_previous_constant_qu_cu}.} established modified scattering of small global solutions to the one-dimensional Klein-Gordon equation
\begin{equation} \label{equ:intro_previous_constant_qu_cu}
 (\pt^2 - \px^2 + 1) u = \alpha_0 u^2 + \beta_0 u^3 \text{ on } \bbR^{1+1}
\end{equation}
with $\alpha_0, \beta_0 \in \bbR$ in the sense that the solutions are shown to decay in $L^\infty_x$ at the rate $t^{-\hf}$ of free Klein-Gordon waves, but that their asymptotics feature logarithmic phase corrections with respect to the free flow. 
An alternative physical space approach was later developed by the first and fourth authors~\cite{LS05_1, LS05_2} in the cubic case, providing a detailed asymptotic expansion of the solution for large times. Subsequently, Hayashi-Naumkin~\cite{HN08, HN12} removed the compact support assumptions about the initial data required in~\cite{Del01, LS05_1, LS05_2}, see also Stingo~\cite{Stingo18} and the work of Candy and the first author~\cite{CL18}.

The study of the asymptotics of small solutions to one-dimensional Klein-Gordon equations with variable coefficient nonlinearities was initiated by the first and fourth authors~\cite{LS15} and by Sterbenz~\cite{Sterb16} for the model
\begin{equation} \label{equ:intro_previous_constant_qu_cu_plus_var_cu}
 (\pt^2 - \px^2 + 1) u = \alpha_0 u^2 + \beta_0 u^3 + \beta(x) u^3 \text{ on } \bbR^{1+1},
\end{equation}
where $\alpha_0, \beta_0 \in \bbR$ and where $\beta(x)$ is a spatially localized, variable coefficient. 
Surprisingly, the addition of a variable coefficient cubic nonlinearity in~\eqref{equ:intro_previous_constant_qu_cu_plus_var_cu} leads to non-trivial difficulties of dealing with the long-range nature of the (non-localized) constant coefficient quadratic and cubic nonlinearities. The latter typically requires to combine energy estimates for weighted vector fields with an ODE argument and normal form methods. In the case of the Klein-Gordon equation, the Lorentz boost $Z = t \px + x \pt$ is the only weighted vector field that commutes with the linear flow. However, differentiation of the variable coefficient by a Lorentz boost produces a strongly divergent factor of~$t$, which seems to place corresponding slow energy growth estimates out of reach.
In~\cite{LS15, Sterb16} the main idea to overcome this issue is the introduction of a variable coefficient cubic normal form. 
More recently, three of the authors~\cite{LLS1} obtained an improvement of~\cite{LS15, Sterb16} using local decay estimates for the Klein-Gordon propagator to overcome difficulties caused by the variable coefficient nonlinearity. 

In~\cite{LLS2} three of the authors recently considered the quadratic Klein-Gordon equation 
\begin{equation} \label{equ:intro_previous_var_qu}
 (\pt^2 - \px^2 + 1) u = \alpha(x) u^2  \text{ on } \bbR^{1+1}
\end{equation}
with a spatially localized coefficient $\alpha(x)$ and uncovered a novel modified scattering behavior of small solutions that involves a logarithmic slow-down of the free decay rate along certain rays. This discovery provided the impetus for the present work. We note that the occurrence of a logarithmic-type slow-down of the decay rate due to the presence of a space-time resonance was pointed out by Bernicot-Germain~\cite{BerGerm13} in a simpler setting of proving bilinear dispersive estimates for quadratic interactions of 1D free dispersive waves. See also~\cite{DIP17, DIPP17} for higher-dimensional instances, where the optimal pointwise decay cannot be propagated by the nonlinear flow (but where the obtained decay rate is not asserted to be sharp).
We emphasize that~\cite[Theorem 1.1]{LLS2} and Theorem~\ref{thm:thm1} of the present work uncover a sharp picture of the asymptotics for one-dimensional nonlinear Klein-Gordon models, where a logarithmic slow-down of the free decay rate occurs. In particular, the origin of the logarithmic loss is precisely identified to stem from the contribution of an explicit resonant source term that is deeply related to the threshold resonance of the Klein-Gordon operator.
Moreover, under the non-resonance assumption $\widehat{\alpha}(\pm \sqrt{3}) = 0$, \cite[Theorem 1.6]{LLS2} establishes that small solutions~to
\begin{equation} \label{equ:intro_previous_var_qu_const_cub_non_resonant}
 (\pt^2 - \px^2 + 1) u = \alpha(x) u^2 + \beta_0 u^3 + \beta(x) u^3 \text{ on } \bbR^{1+1}
\end{equation}
decay in $L^\infty_x$ at the free rate $t^{-\hf}$ and that their asymptotics feature logarithmic phase corrections (caused by the constant coefficient cubic nonlinearity $\beta_0 u^3$). 

Recently, Germain-Pusateri~\cite{GP20} studied the following general one-dimensional quadratic Klein-Gordon equation with a linear potential
\begin{equation} \label{equ:intro_previous_with_pot_ax}
 (\pt^2 - \px^2 + 1 + V(x)) u = a(x) u^2 \text{ on } \bbR^{1+1},
\end{equation}
where $a(x)$ is a smooth coefficient satisfying $a(x) \to \ell_{\pm \infty}$ as $x \to \pm \infty$ for arbitrary fixed $\ell_{\pm \infty} \in \bbR$ (and is thus not necessarily localized) and where $H = - \px^2 + V(x)$ has no bound states. Under the key assumption that the distorted Fourier transform of the solution $\tilde{u}(t,0) = 0$ vanishes at zero frequency at all times $t \in \bbR$, \cite[Theorem 1.1]{GP20} establishes that small solutions to~\eqref{equ:intro_previous_with_pot_ax} decay in $L^\infty_x$ at the free rate $t^{-\hf}$ and that their asymptotics feature logarithmic phase corrections (caused by the ``non-zero limits'' $\ell_{\pm\infty}$ of the coefficient $a(x)$).
We note that $\tilde{u}(t,0) = 0$ holds automatically for generic potentials, while in the case of non-generic potentials this condition only holds for solutions that are ``orthogonal'' to the zero energy resonance of $H$ (in the sense of an $L^1_x$--$L^\infty_x$ pairing). The latter can for instance be enforced by imposing suitable parity conditions. As an application, \cite[Corollary 1.4]{GP20} yields the full asymptotic stability of kinks with respect to \emph{odd} perturbations for the double sine-Gordon problem in an appropriate range of the deformation parameter.

For closely related results on modified scattering for nonlinear Schr\"odinger equations, we refer to~\cite{HN98, LS06, KatPus11, IT15, DZ03, GHW15, Del16, Naum16, MurphPus17, GermPusRou18, Leg18, Leg19, MasMurphSeg19, ChenPus19, PusSof20} and references therein. 

\medskip 

Finally, we anticipate that local decay estimates for the perturbed Klein-Gordon propagator $e^{it\sqrt{m^2 + H}} P_c$ play a major role in the proof of the main result in this paper. Such local decay estimates for much larger classes of unitary operators originate in the works of Rauch~\cite{Rauch78}, Jensen-Kato~\cite{KJ79}, and Jensen~\cite{Jensen80, Jensen84}, see also~\cite{HSS99, JSS91, KS06, Gold07, KomKop10, Kop11, Ger08, GLS16, LarS15, EgoKopMarTes16} as well as the survey~\cite{Schl07} and references therein.

\subsection{Main result}

We are now in the position to state the main result of this paper on the long-time behavior of small solutions to~\eqref{equ:intro_nlkg}. Without loss of generality we set the mass parameter $m=1$. We write $\jtD = \sqrt{1+H}$ on the positive spectrum of $H = -\px^2 + V(x)$ and we denote by~$\widetilde{\calF}$ the distorted Fourier transform associated with $H$. We refer to the beginning of Section~\ref{sec:spectral_scattering_theory} for a brief review of some basics of the spectral and scattering theory for Schr\"odinger operators $H$. 

Given a solution $u(t)$ to~\eqref{equ:intro_nlkg}, we introduce the new variable 
\begin{equation*}
 v(t) := \frac{1}{2} \bigl( u(t) - i \jtD^{-1} \pt u(t) \bigr)
\end{equation*}
that satisfies the first-oder Klein-Gordon equation 
\begin{equation*}
 (\pt - i \jtD) v = \frac{1}{2i} \jtD^{-1} P_c \bigl( \alpha(\cdot) (v+\bar{v})^2 \bigr) \text{ on } \bbR^{1+1}
\end{equation*}
with initial datum $v(0) = \frac12 (P_c u_0 - i \jtD^{-1} P_c u_1)$.
It suffices to derive decay estimates and asymptotics for the variable~$v(t)$ since we have that
\begin{equation} \label{equ:u_equ_v_plus_vbar}
 u(t) = v(t) + \bar{v}(t).
\end{equation}
We will occasionally use~\eqref{equ:u_equ_v_plus_vbar} as a convenient short-hand notation. The following theorem contains the main result of this paper.

\begin{theorem} \label{thm:thm1}
 Assume that the real-valued potential $V \in L^\infty(\R) \cap C^3(\R)$ satisfies $\jap{x}^9 V^{(\ell)}(x) \in L^1(\R)$ for all $0\le\ell\le3$,  
 and that $H = -\px^2 + V(x)$ exhibits a zero energy resonance $\varphi(x)$, cf.~Definition~\ref{def:reson}. Suppose that $\|\jx^{15} \alpha(x)\|_{H^3_x} < \infty$.
 Then there exists an absolute constant $0 < \varepsilon_0 \ll 1$ such that for any initial condition $v_0$ satisfying 
 \begin{equation*}
  \varepsilon := \|\jx^5 v_0\|_{H^2_x} \leq \varepsilon_0,   
 \end{equation*}
 there exists a global-in-time solution $v \in C(\bbR; H^2_x)$ to
 \begin{equation} \label{equ:thm1_nlkg}
  (\pt - i \jtD) v = \frac{1}{2i} \jtD^{-1} P_c \bigl( \alpha(\cdot) (v+\bar{v})^2 \bigr) \text{ on } \bbR^{1+1}
 \end{equation}
 with initial datum $v(0) = P_c v_0$. Moreover, the solution $v(t)$ exhibits the following asymptotic behavior as $t \to \infty$:
 \begin{itemize}[leftmargin=*]
  \item (Resonant Case) Suppose 
  \begin{equation*}
   \wtcalF\bigl[ \alpha \varphi^2 ](\sqrt{3}) \neq 0 \quad \text{or} \quad \wtcalF\bigl[ \alpha \varphi^2 \bigr](-\sqrt{3}) \neq 0.
  \end{equation*}
  Then it holds 
  \begin{equation} \label{equ:thm1_Linfty_decay_resonant}
   \|v(t)\|_{L^\infty_x} \lesssim \frac{\log(1+\jt)}{\jt^\hf} \varepsilon.
  \end{equation}
  In addition, $v(t)$ admits a decomposition
  \begin{equation*}
   v(t) = v_{free}(t) + v_{mod}(t), \qquad t \geq 1,
  \end{equation*}
  with the following properties:
  \begin{itemize}
  \item[(i)] The component $v_{free}(t)$ satisfies
  \begin{equation} \label{equ:thm1_Linfty_decay_vfree}
   \|v_{free}(t)\|_{L^\infty_x} \lesssim \frac{\varepsilon}{\jap{t}^{\hf}}, \qquad t \geq 1.
  \end{equation}
  Moreover, $v_{free}(t)$ scatters to a free Klein-Gordon wave in $H^2_x$ in the sense that there exists $v_\infty \in H^2_x$ such that 
  \begin{equation} \label{equ:thm1_vfree_scattering}
   \bigl\| v_{free}(t) - e^{it\jtD} v_\infty \bigr\|_{H^2_x} \lesssim \frac{\varepsilon^2}{\jt^\hf}, \qquad t \geq 1.
  \end{equation}
  \item[(ii)] There exists a small amplitude $a_0 \in \bbC$, $|a_0| \lesssim \varepsilon$, such that the component $v_{mod}(t)$ is given by 
  \begin{equation} \label{equ:thm1_def_vmod}
   v_{mod}(t) := c_0^2 \frac{a_0^2}{2} \int_1^t e^{i(t-s)\jtD} \jtD^{-1} P_c \bigl( \alpha \varphi^2 \bigr) \frac{e^{2is}}{s} \, \ud s,
  \end{equation}
  where the real constant $c_0$ only depends on the scattering matrix $S(0)$ of the potential $V(x)$ at zero energy, cf.~\eqref{eq:SU2}, and is explicitly given by
  \begin{equation} \label{equ:thm_def_c0}
   c_0 = \frac{1}{(2\pi)^{\frac32}} \frac{T(0)^2}{1+R_-(0)}, \quad \text{where} \quad T(0) \neq 0.
  \end{equation}
  For arbitrary $0 < \delta \ll 1$, there exists a constant $C_\delta \geq 1$ such that we have uniformly
  \begin{equation} \label{equ:thm1_resonant_decay_off_rays}
   \bigl| v_{mod}(t,x) \bigr| \leq C_\delta \frac{\varepsilon^2}{\jap{t}^{\frac{1}{2}}} \quad \text{whenever} \quad |x| < \Bigl( \frac{\sqrt{3}}{2} - \delta \Bigr) t \quad \text{or} \quad |x| > \Bigl( \frac{\sqrt{3}}{2} + \delta \Bigr) t,
  \end{equation}
  and along the rays $x = \pm \frac{\sqrt{3}}{2} t$ the asymptotics of $v_{mod}(t)$ are given by
  \begin{equation} \label{equ:thm1_resonant_asymptotics_along_special_rays}
   v_{mod}\Bigl( t, \pm \frac{\sqrt{3}}{2} t \Bigr) = c_0^2 \frac{a_0^2}{\sqrt{8}} e^{i\frac{\pi}{4}} e^{i \frac{t}{2}} \widetilde{\calF}[\alpha \varphi^2](\mp \sqrt{3}) \frac{\log(t)}{t^{\frac{1}{2}}} + \calO_{L^\infty_t} \Bigl( \frac{\varepsilon^2}{t^\hf} \Bigr), \quad t \gg 1. 
  \end{equation}  
  In particular, when $a_0 \neq 0$ the decay estimate~\eqref{equ:thm1_Linfty_decay_resonant} is sharp. 
  \end{itemize}
  \item (Non-Resonant Case) Suppose 
  \begin{equation*}
   \wtcalF\bigl[ \alpha \varphi^2 ](\sqrt{3}) = 0 \quad \text{and} \quad \wtcalF\bigl[ \alpha \varphi^2 \bigr](-\sqrt{3}) = 0.
  \end{equation*}
  Then it holds 
  \begin{equation} \label{equ:thm1_Linfty_decay_nonresonant}
   \|v(t)\|_{L^\infty_x} \lesssim \frac{\varepsilon}{\jt^\hf}.
  \end{equation}
  Moreover, $v(t)$ scatters to a free Klein-Gordon wave in $H^2_x$ in the sense that there exists $v_\infty \in H^2_x$ such that 
  \begin{equation} \label{equ:thm1_nonresonant_scattering}
   \bigl\| v(t) - e^{it\jtD} v_\infty \bigr\|_{H^2_x} \lesssim \frac{\varepsilon^2}{\jt^\hf}, \qquad t \geq 1.
  \end{equation}
 \end{itemize}
\end{theorem}

\noindent We proceed with several remarks on Theorem~\ref{thm:thm1}:
\begin{itemize}
 \setlength\itemsep{1em}
 \item[(i)] The amplitude $a_0 \in \bbC$ in the statement of Theorem~\ref{thm:thm1} is explicitly given by 
 \begin{equation} \label{equ:formula_a0}
  \begin{aligned}
  a_0 = \langle \varphi, v_0 \rangle &+ \frac{1}{2} \langle \varphi, \alpha(\cdot) v_0^2 \rangle - \langle \varphi, \alpha(\cdot) |v_0|^2 \rangle - \frac{1}{6} \langle \varphi, \alpha(\cdot) \overline{v}_0^2 \rangle \\
  &+ \int_0^\infty e^{is} \langle \varphi, \alpha(\cdot) \partial_s \bigl( e^{-is} v(s) \bigr) \bigl( e^{-is} v(s) \bigr) \rangle \, \ud s \\
  &- \int_0^\infty e^{-is} \langle \varphi, \alpha(\cdot) \partial_s \bigl( (e^{-is} v(s)) (e^{is} \bar{v}(s) ) \bigr) \rangle \, \ud s \\
  &- \frac{1}{3} \int_0^\infty e^{-3is} \langle \varphi, \alpha(\cdot) \partial_s \bigl( e^{is} \bar{v}(s) \bigr) \bigl( e^{is} \bar{v}(s) \bigr) \rangle \, \ud s,
  \end{aligned}
 \end{equation}
 where we use the notation $\langle f, g \rangle := \int_\bbR \overline{f(x)} g(x) \, \ud x$. In particular, we have $a_0 \neq 0$ when $\langle \varphi, v_0 \rangle \neq 0$.
 
 \item[(ii)] We did not optimize the decay and regularity assumptions on the initial data and on the potential. The proof of Theorem~\ref{thm:thm1} given below can be improved to some extent to sharpen these assumptions.  
 
 \item[(iii)] Under certain conditions, the nonlinear solution $v(t)$ to~\eqref{equ:thm1_nlkg} does not exhibit modified scattering in the sense that it just scatters to a free Klein-Gordon wave. On the one hand, this occurs in the non-resonant case $\widetilde{\calF}[\alpha \varphi^2](\pm \sqrt{3}) = 0$ for arbitrary (sufficiently small) initial data.
 On the other hand, this may also occur in the resonant case for initial data satisfying certain parity conditions. From the explicit formula~\eqref{equ:formula_a0} for the coefficient $a_0$ it is evident that $a_0 = 0$ (and thus $v_{mod}(t) \equiv 0$) if $\varphi(x)$ is even and $v(t)$ as well as $\alpha(x)$ are odd, or if $\varphi(x)$ is odd and $v(t)$ as well as $\alpha(x)$ are even. Of course, a parity condition on the solution $v(t)$ to~\eqref{equ:thm1_nlkg} in turn imposes corresponding parity conditions on the potential $V(x)$ and the eigenfunctions of $H = -\px^2 + V(x)$. 
 
 \item[(iv)] The Klein-Gordon model~\eqref{equ:intro_nlkg} considered in this paper is a simplified model for nonlinear Klein-Gordon equations with non-generic potentials such as~\eqref{equ:phi4_perturbation} and~\eqref{equ:sineGordon_perturbation} that govern the dynamics of the dispersive remainder term in a perturbative approach to the study of the asymptotic stability of kinks. 
 An important next step is to additionally allow for non-localized cubic nonlinearities $\beta_0 u^3$ on the right-hand side of~\eqref{equ:thm1_nlkg} and ultimately for non-localized variable coefficient quadratic nonlinearities in the presence of general non-generic potentials (say without bound states) \emph{without making any parity assumptions on the initial data}.
 In fact, even just for the case of a pure non-localized cubic nonlinearity, it remains a very interesting open problem to establish decay estimates and asymptotics for small solutions to the one-dimensional Klein-Gordon equation 
 \begin{equation} \label{equ:constant_cubic_kg_intro}
  (\pt^2 - \px^2 + V(x) + 1) u = \beta_0 u^3
 \end{equation}
 with a general non-generic potential $V(x)$ (say without bound states) and \emph{without making any parity assumptions on the initial data}. 
 In the case of generic potentials and in the case of non-generic potentials \emph{under suitable parity assumptions on the initial data} (or under a related assumption about the vanishing of the distorted Fourier transform of the solution at zero frequency), two approaches have emerged over the last years to prove modified scattering for small solutions to \eqref{equ:constant_cubic_kg_intro}: one based on the distorted Fourier transform, see \cite{Naum16, Naum18, GermPusRou18, ChenPus19, GP20, GHW15}, and one using the wave operator, see \cite{Del16, DelMas20}.
 Both approaches appear to crucially rely on genericity assumptions or parity assumptions to by-pass the effects of the threshold resonances of the linear operator.
 
 The proof of Theorem~\ref{thm:thm1} builds on the spatial localization of the variable coefficient~$\alpha(x)$ on the right-hand side of~\eqref{equ:thm1_nlkg} in conjunction with the use of refined local decay estimates. Correspondingly, it is not straightforward to extend Theorem~\ref{thm:thm1} to the above mentioned more general settings involving non-localized quadratic or cubic nonlinearities. Non-localized low power nonlinearities are to some extent incompatible with the use of weighted norms as in the proof of Theorem~\ref{thm:thm1}.
 
 After completion of this work, in the context of proving the asymptotic stability of the sine-Gordon kink under odd perturbations, two of the authors~\cite{LS1} introduced an approach to study modified scattering problems for Klein-Gordon equations with non-generic P\"oschl-Teller potentials by exploiting specific super-symmetric factorization properties of the corresponding linear Klein-Gordon operators. 
 
 We note that the linearized Klein-Gordon equation around a (static) kink solution to the scalar field equation~\eqref{equ:EL_equation} features a spatially localized variable coefficient $\alpha(x)$ for the quadratic nonlinearity as in~\eqref{equ:intro_nlkg} if and only if the scalar potential $W$ in~\eqref{equ:EL_equation} satisfies $W^{(3)}(\phi_\pm) = 0$. For example, this is the case for the sine-Gordon model, but not for the $\phi^4$ model.

 \item[(v)] We expect that in the presence of a non-generic potential $V(x)$, a slow-down of the decay rate as uncovered in Theorem~\ref{thm:thm1} should occur more generally for coefficients $\alpha(x)$ that may also assume non-zero limits $\alpha(x) \to \alpha_{\pm \infty} \neq 0$ as $x \to \pm \infty$.
 However, we would like to emphasize again that it is by far not straightforward to extend Theorem~\ref{thm:thm1} to this more general setting.
 The proof of Theorem~\ref{thm:thm1} crucially exploits the spatial localization of the coefficient $\alpha(x)$ in conjunction with the use of refined local decay estimates for the perturbed Klein-Gordon propagator.
 
 The extension of Theorem~\ref{thm:thm1} for arbitrary small initial data to non-localized coefficients $\alpha(x)$ likely requires to further advance normal form techniques in the presence of a non-generic potential. 
 The difficulty of this step is deeply related to a loss of regularity of the distorted Fourier transform of the profile $g(t) := e^{-it\jtD} v(t)$ of the solution to~\eqref{equ:thm1_nlkg} caused by the quadratic nonlinearity. 
 This also manifests itself prominently in the difficulty to derive slowly growing energy estimates for a Lorentz boost $Z = t \px + x \pt$ of the nonlinear solution to~\eqref{equ:intro_nlkg} in the flat case $V(x) = 0$. Indeed, when the Lorentz boost falls onto the variable coefficient of the quadratic nonlinearity, it produces a strongly divergent factor of~$t$ that is hard to sufficiently compensate for.
 
 For generic potentials as well as for non-generic potentials in the special case of solutions that are ``orthogonal'' to the threshold resonance, these difficulties have very recently been overcome in the remarkable work of Germain-Pusateri~\cite{GP20}. 
 We note that under the assumptions of~\cite[Theorem 1.1]{GP20}, the nonlinear solutions to the model~\eqref{equ:intro_previous_with_pot_ax} decay in $L^\infty_x$ at the usual free decay rate $t^{-\frac12}$ and their asymptotic behavior features logarithmic phase corrections ``caused by the non-zero limits'' $\ell_{\pm\infty}$ of the coefficient $a(x)$ in~\eqref{equ:intro_previous_with_pot_ax} at spatial infinity.
 
 For a related discussion, we refer to Remark (6) following Theorem~1.1 in~\cite{GP20} and to the remarks at the end of Subsection~2.3 in~\cite{GP20}.
  
 \item[(vi)] The explicit expression~\eqref{equ:thm1_def_vmod} for $v_{mod}(t)$ indicates that we would have $v_{mod}(t) \equiv 0$ for generic potentials $V(x)$, because their transmission coefficient vanishes at zero energy $T(0) = 0$, whence $c_0 = 0$.
 
 \item[(vii)] The proof of Theorem~\ref{thm:thm1} easily generalizes to arbitrary mass parameters $m \neq 0$ in the Klein-Gordon model~\eqref{equ:intro_nlkg}. Then the resonant case corresponds to the condition 
 \begin{equation*}
  \wtcalF\bigl[ \alpha \varphi^2 ](\sqrt{3m^2}) \neq 0 \quad \text{or} \quad \wtcalF\bigl[ \alpha \varphi^2 \bigr](-\sqrt{3m^2}) \neq 0,
 \end{equation*}
 and a logarithmic slow-down of the decay rate also occurs along the rays $x = \pm \frac{\sqrt{3}}{2} t$.

 \item[(viii)] 
 In the resonant case when $a_0 \neq 0$ the distorted Fourier transform of the profile $g(t) := e^{-it\jtD} v(t)$ of the solution $v(t)$ to~\eqref{equ:thm1_nlkg} diverges logarithmically at frequencies $\xi = \pm \sqrt{3}$, see Remark~\ref{rem:profile_log_divergence}. Specifically, one has that
 \begin{equation*}
  \wtilg(t, \pm \sqrt{3}) = c_0^2 \frac{a_0^2}{4} \widetilde{\calF}\bigl[\alpha \varphi^2\bigr](\pm \sqrt{3}) \log(t) + \calO(\varepsilon), \quad t \gg 1,
 \end{equation*}
 which indicates that the free $L^\infty_x$ decay rate $t^{-\frac12}$ cannot be expected for the solution $v(t)$.
 
 \item[(ix)] In the resonant case, the derivation of the asymptotics of $v_{mod}(t)$ in the proof of Theorem~\ref{thm:thm1} along the special rays $x = \pm \frac{\sqrt{3}}{2} t$  also applies to nearby rays $x = \lambda t$ with $|\lambda - (\pm \frac{\sqrt{3}}{2} )| \ll 1$. One finds that uniformly for all $|\lambda - (\pm \frac{\sqrt{3}}{2}) | \ll 1$,
 \begin{equation*}
  \biggl| v_{mod}(t, \lambda t) - c_0^2 \frac{a_0^2}{2} \frac{e^{i\frac{\pi}{4}} e^{i (1-\lambda^2)^{\hf} t}}{(1-\lambda^2)^{\frac14}} \widetilde{\calF}[\alpha \varphi^2]\Bigl(-\frac{\lambda}{(1-\lambda^2)^\hf}\Bigr) \frac{A(t, \lambda)}{t^\hf} \biggr| \leq C \frac{\varepsilon^2}{t^\hf} , \quad t \gg 1,
 \end{equation*}
 where the amplitude correction $A(t, \lambda)$ is of the form
 \begin{equation*}
  A(t, \lambda) := \int_1^{t^\hf} \frac{e^{i s (2-(1-\lambda^2)^{-\hf})}}{s} \, \ud s. 
 \end{equation*}
 Clearly, along the special rays $\lambda = \pm \frac{\sqrt{3}}{2}$, this yields the asymptotics~\eqref{equ:thm1_resonant_asymptotics_along_special_rays} featuring a logarithmic slow-down of the decay rate, while we obtain uniformly for all nearby rays $\lambda \neq \pm \frac{\sqrt{3}}{2}$ that
 \begin{equation*}
  |A(t, \lambda)| \lesssim \frac{1}{2 - (1-\lambda^2)^{-\hf}}, \quad t \gg 1.
 \end{equation*}

 \item[(x)] It appears that the novel type of modified scattering behavior uncovered in Theorem~\ref{thm:thm1} as well as in~\cite[Theorem 1.1]{LLS2} is reminiscent of a new phenomenon observed in the remarkable recent work of Delort-Masmoudi~\cite{DelMas20} on long-time dispersive estimates for odd perturbations of the (odd) kink $\psi_{\phi^4}(x) = \tanh(\frac{x}{\sqrt{2}})$ in the $\phi^4$ model. We recall from~\eqref{equ:phi4_perturbation} that the corresponding remainder term $u(t,x) = \phi(t,x) - \psi_{\phi^4}(x)$ satisfies the equation 
 \begin{equation} \label{equ:phi4_perturbation_recalled}
 \bigl( \partial_t^2 - \partial_x^2 + 2 - 3 \sech^2( {\textstyle \frac{x}{\sqrt{2}} }) \bigr) u = - 3 \tanh( {\textstyle \frac{x}{\sqrt{2}} } ) u^2 - u^3.
\end{equation}
 The linear operator $-\px^2 + 2 - 3 \sech^2(\frac{x}{\sqrt{2}})$ exhibits an even threshold resonance 
 \begin{equation} \label{equ:phi4_threshold_resonance}
 \varphi(x) = 1 - \frac32 \sech^2\Bigl( \frac{x}{\sqrt{2}} \Bigr), 
 \end{equation}
 and possesses an odd internal mode with eigenvalue $\mu^2$, $\mu = \sqrt{\frac{3}{2}}$, given by
 \begin{equation} \label{equ:comment_phi4_def_Y}
  Y(x) = 2^{-\frac{3}{4}} 3^{\frac12} \tanh\Bigl( \frac{x}{\sqrt{2}} \Bigr) \sech\Bigl( \frac{x}{\sqrt{2}} \Bigr), \quad \langle Y, Y \rangle = 1.
 \end{equation}
 Note that for odd perturbations one can disregard the even zero eigenfunction of the linear operator stemming from the translation invariance of the model. 
 To study the long-time behavior of odd solutions to~\eqref{equ:phi4_perturbation_recalled} one therefore enacts a spectral decomposition
 \begin{equation} \label{equ:comment_phi4_decomp}
  u(t,x) = z(t) Y(x) + w(t,x), \quad \langle Y, w(t) \rangle = 0,
 \end{equation}
 where $z(t) = \langle Y, u(t) \rangle$ is the projection of $u(t)$ onto the internal mode $Y(x)$. 
 The presence of the internal mode is a major difficulty in the study of the asymptotic dynamics of $u(t,x)$.
 In fact, at the linear level it would be an obstruction to decay. However, for the nonlinear Klein-Gordon equation~\eqref{equ:phi4_perturbation_recalled}, a coupling of the oscillations of the internal mode to the continuous spectrum occurs through the so-called nonlinear Fermi Golden Rule, see Sigal~\cite{Sigal93} and Soffer-Weinstein~\cite{SofWein99} for pioneering works in this direction. This mechanism was exploited by Kowalczyk-Martel-Mu\~{n}oz~\cite{KMM17} to establish the decay of $w(t)$ in a local energy sense and the decay of $z(t)$ in an integrated sense. Delort-Masmoudi~\cite{DelMas20} recently obtained explicit decay rates for $z(t)$ and for $w(t)$ in $L^\infty_x$ for times up to $T \sim \varepsilon^{-4+c}$ for arbitrary $c > 0$, where $\varepsilon$ is the size of the initial data measured in a weighted Sobolev space. 
 
 It appears that the limitation to times $\calO(\varepsilon^{-4})$ in~\cite{DelMas20} stems from a possible slow-down of the decay rate of $w(t,x)$ along the special rays $\frac{x}{t} = \pm \sqrt{\frac{2}{3}}$. The latter is caused by a resonant source term in the nonlinear Klein-Gordon equation for $w(t)$ whose contribution can be thought of to have the following schematic Duhamel form
 \begin{equation} \label{equ:schematic_phi4_source_term}
  \int_1^t e^{i(t-s)\sqrt{-\px^2 + 2 - 3 \sech^2(\frac{x}{\sqrt{2}})}} P_c \bigl( \alpha Y^2 \bigr) \frac{e^{2i\mu s}}{s} \, \ud s
 \end{equation}
 with $\alpha(x) = \tanh (\frac{x}{\sqrt{2}})$ and $Y(x)$ defined in~\eqref{equ:comment_phi4_def_Y}.
 It arises from the quadratic contribution of the long-time behavior of the projection $z(t) Y(x)$ to the internal mode in the nonlinear Klein-Gordon equation for~$w(t)$, see the spectral decomposition~\eqref{equ:comment_phi4_decomp} above. 
 Interestingly, the structure of the source term~\eqref{equ:schematic_phi4_source_term} is reminiscent of the source term~\eqref{equ:thm1_def_vmod} defining $v_{mod}(t)$ in the statement of Theorem~\ref{thm:thm1}. By the same mechanism described in Subsection~\ref{subsec:proof_ideas} below on the ideas of the proof of Theorem~\ref{thm:thm1}, the source term~\eqref{equ:schematic_phi4_source_term} is resonant at the distorted frequencies~$\xi_\mu$ satisfying $\sqrt{2 + \xi_\mu^2} = 2 \mu$, i.e., $\xi_\mu = \pm 2$, if $\widetilde{\calF}[\alpha Y^2](\pm \xi_\mu) \neq 0$. Correspondingly, one can expect a slow-down of the decay rate of $w(t,x)$ along the associated rays
 \begin{equation*}
  \frac{x}{t} = - \frac{\xi_\mu}{\sqrt{2+\xi_\mu^2}} = \mp \sqrt{\frac{2}{3}}.
 \end{equation*}
 In the context of the $\phi^4$ model, the resonance condition $\widetilde{\calF}[\alpha Y^2](\pm \xi_\mu) \neq 0$ is referred to as the nonlinear Fermi Golden Rule and it is in fact key for the projection $z(t) Y(x)$ of $u(t)$ to the internal mode to decay at all as $t \to \infty$. 
 
 We stress that the displayed form~\eqref{equ:schematic_phi4_source_term} of the contribution of the resonant source term is very schematic, and just serves here to highlight the intriguing resemblance of the source term~\eqref{equ:thm1_def_vmod} defining $v_{mod}(t)$ in the statement of Theorem~\ref{thm:thm1} and the source term~\eqref{equ:schematic_phi4_source_term} appearing in the analysis of perturbations of the $\phi^4$ kink. 
 While the source term~\eqref{equ:thm1_def_vmod} is ultimately caused by a threshold resonance, the source term~\eqref{equ:schematic_phi4_source_term} is caused by the internal mode of the $\phi^4$ model.
 Finally, we note that a possible slow-down effect of the decay rate of $w(t)$ in~\eqref{equ:comment_phi4_decomp} due to the threshold resonance~\eqref{equ:phi4_threshold_resonance}, similar to the result in Theorem~\ref{thm:thm1}, is not expected for odd perturbations of the $\phi^4$ kink since these are ``orthogonal'' to the even threshold resonance~\eqref{equ:phi4_threshold_resonance}.
\end{itemize}

\begin{remark}
 A natural question is whether the non-resonance condition $\widetilde{\calF}[\alpha \varphi^2](\pm \sqrt{3}) = 0$ happens to hold in concrete applications to asymptotic stability problems for kink solutions. It turns out that the sine-Gordon model features this miraculous vanishing property! Recall from~\eqref{equ:sineGordon_perturbation} that the equation for a perturbation of the static sine-Gordon kink involves the variable quadratic coefficient
 \begin{equation*}
  \alpha(x) = \sech(x) \tanh(x)
 \end{equation*}
 and the Schr\"odinger operator 
 \begin{equation*}
  H = -\px^2 - 2 \sech^2(x).
 \end{equation*}
 The latter belongs to the family of P\"oschl-Teller potentials, see for instance \cite[Problem 39]{Fluegge}, and admits a zero energy resonance that is explicitly given by
 \begin{equation*}
  \varphi(x) = \tanh(x).
 \end{equation*}
 It turns out that the distorted Fourier transform with respect to $H$ of $\alpha \varphi^2$ satisfies
 \begin{equation} \label{equ:sineGordon_vanishing}
  \widetilde{\calF}\bigl[ \alpha \varphi^2 \bigr](\pm \sqrt{3}) = 0.
 \end{equation}
 The authors are not aware of a reference in the literature for this observation\footnote{This observation has previously been made by Jacob Sterbenz (unpublished note).}. Below we provide a simple proof of~\eqref{equ:sineGordon_vanishing} using contour integration. 
\end{remark}
\begin{proof}[Proof of \eqref{equ:sineGordon_vanishing}]
By direct computation one can verify that the Jost solutions of the Schr\"odinger operator $H = -\px^2 - 2 \sech^2(x)$ are explicitly given by
\begin{align*}
 f_+(x,\xi) &= \frac{i \xi - \tanh(x)}{i \xi - 1} e^{ix \xi}, \\
 f_-(x, \xi) &= \frac{-i\xi - \tanh(x)}{-i \xi + 1} e^{-ix\xi}.
\end{align*}
The distorted Fourier basis associated with $H$ therefore takes the form
\begin{equation*}
 e(x,\xi) := \frac{1}{\sqrt{2 \pi}} \left\{ \begin{aligned}
                         &T(\xi) \frac{i \xi - \tanh(x)}{i \xi - 1} e^{ix \xi} &\text{for } \xi \geq 0, \\
                         &T(-\xi) \frac{i \xi - \tanh(x)}{i \xi + 1} e^{ix\xi} &\text{for } \xi < 0,
                        \end{aligned} \right. 
\end{equation*}
where $T(\xi)$ denotes the transmission coefficient associated with $H$.
Thus, in order to evaluate the distorted Fourier transform of $\alpha \varphi^2$ at frequencies $\xi = \pm \sqrt{3}$, 
\begin{equation*}
 \widetilde{\calF}\bigl[\alpha \varphi^2\bigr](\pm \sqrt{3}) = \int_{\bbR} \overline{e(x, \pm \sqrt{3})} \, \alpha(x) \varphi(x)^2 \, \ud x, 
\end{equation*}
it suffices to evaluate the integrals
\begin{align*}
  \calI_\pm &:= \int_\bbR e^{\pm i \sqrt{3} x} \bigl( \pm i \sqrt{3} - \tanh(x) \bigr) \frac{\sinh^3(x)}{\cosh^4(x)} \, \ud x.
\end{align*}
 To this end we observe that the function 
 \begin{equation*}
  F_\pm(z) := e^{\pm i \sqrt{3} z} \bigl( \pm i \sqrt{3} - \tanh(z) \bigr) \frac{\sinh^3(z)}{\cosh^4(z)}, \quad z \in \bbC,
 \end{equation*}
 is meromorphic on $\bbC$ with poles at $z_k = i \frac{\pi}{2} (2k+1)$, $k \in \bbZ$. It is easy to see that the integral $\calI_\pm$ can be obtained from the contour integral of $F_\pm$ along the rectangle with vertices at $\pm \ell \pi$, $\pm \ell \pi + i \ell \pi$ as $\ell \to \infty$. By the residue theorem, it follows that
 \begin{equation*}
  \calI_\pm = 2\pi i \sum_{k=0}^\infty \mathrm{Res}_{z=z_k}(F_\pm). 
 \end{equation*}
 Using that $\cosh(z_k + w) = i (-1)^k \sinh(w)$ and that $\sinh(z_k + w) = i (-1)^k \cosh(w)$, we find that
 \begin{equation*}
  F_\pm(z_k + w)= (-1)^{k-1} i e^{\mp \sqrt{3} \frac{\pi}{2} (2k+1)} \biggl( \pm i \sqrt{3} \, e^{\pm i \sqrt{3} w} \frac{\cosh^3(w)}{\sinh^4(w)} - e^{\pm i \sqrt{3}w} \frac{\cosh^4(w)}{\sinh^5(w)} \biggr).
 \end{equation*}
 Then we compute 
 \begin{align*}
  \mathrm{Res}_{w=0} \biggl( e^{\pm i \sqrt{3} w} \frac{\cosh^3(w)}{\sinh^4(w)} \biggr) &= \pm \frac{i}{\sqrt{3}}, \\
  \mathrm{Res}_{w=0} \biggl( e^{\pm i \sqrt{3} w} \frac{\cosh^4(w)}{\sinh^5(w)} \biggr) &= -1.
 \end{align*}
 Correspondingly, we obtain for all $k \in \bbZ$ that
 \begin{align*}
  \mathrm{Res}_{z=z_k}(F_\pm) = (-1)^{k-1} i e^{\mp \sqrt{3} \frac{\pi}{2} (2k+1)} \biggl( \pm i \sqrt{3} \Bigl( \pm \frac{i}{\sqrt{3}} \Bigr) - (-1) \biggr) = 0,
 \end{align*}
 whence $\calI_\pm = 0$, which implies the asserted vanishing property $\widetilde{\calF}\bigl[\alpha \varphi^2\bigr](\pm \sqrt{3}) = 0$.
\end{proof}

\subsection{Proof ideas} \label{subsec:proof_ideas}

The analysis of the asymptotic behavior of small global solutions to the one-dimensional quadratic Klein-Gordon equation
\begin{equation} \label{equ:proof_ideas_nlkg}
 (\pt - i \jtD) v = \frac{1}{2i} \jtD^{-1} P_c \bigl( \alpha(\cdot) (v + \bar{v})^2 \bigr) \text{ on } \bbR^{1+1}
\end{equation}
under the assumptions of Theorem~\ref{thm:thm1} begins with the observation that due to the spatial localization of the coefficient $\alpha(x)$, the nature of the quadratic nonlinearity $\alpha(x) (v + \bar{v})^2$ is entirely determined by the \emph{local decay} of the nonlinear solution $v(t)$. 

It is therefore instructive to first study the interactions in the quadratic nonlinearity $\alpha(x) (v(t) + \bar{v}(t))^2$ when $v(t)$ is replaced by a linear Klein-Gordon wave $v_{lin}(t) = e^{it \jtD} P_c v_0$. Since~$H = -\px^2 + V(x)$ is assumed to exhibit a zero energy resonance $\varphi(x)$, the local decay of $e^{it\jtD} P_c v_0$ (measured in weighted spaces) is only of order $t^{-\hf}$. Importantly, this slow local decay solely stems from a contribution of the zero energy resonance $\varphi(x)$ in the sense that upon subtracting a suitable projection onto $\varphi(x)$, the bulk of the linear Klein-Gordon wave $e^{it\jtD} P_c v_0$ exhibits faster local decay. More specifically, one of the key local decay estimates for the Klein-Gordon evolution on the line, which we establish in Subsection~\ref{subsec:decay_estimates}, reads
\begin{equation} \label{equ:proof_ideas_local_decay_est_subtr_off}
 \Bigl\| \jx^{-\sigma} \Bigl( e^{it\jtD} P_c v_0 - c_0 \frac{e^{i\frac{\pi}{4}} e^{it}}{t^\hf} \langle \varphi, v_0 \rangle \varphi \Bigr) \Bigr\|_{L^2_x} \lesssim \frac{1}{t^\thf} \| \jx^{\sigma} v_0 \|_{L^2_x}, \quad t \geq 1,
\end{equation}
where $\sigma > \frac92$ and the real constant $c_0$ defined in~\eqref{equ:thm_def_c0} only depends on the scattering matrix $S(0)$ of the potential $V(x)$ at zero energy.
The local decay estimate~\eqref{equ:proof_ideas_local_decay_est_subtr_off} suggests that the leading order behavior of $\alpha(x) ( v_{lin}(t) + \bar{v}_{lin}(t) )^2$ should be of the schematic form
\begin{equation*}
 c_0^2 \alpha(x) \varphi(x)^2 \frac{1}{t} \Bigl( e^{i \frac{\pi}{2}} e^{2it} (\langle \varphi, v_0 \rangle )^2 + 2 |\langle \varphi, v_0 \rangle|^2 + e^{-i\frac{\pi}{2}} e^{-2it} ( \overline{\langle \varphi, v_0 \rangle} )^2 \Bigr) + \calO_{L^\infty_t} \Bigl( \frac{1}{t^2} \Bigr), \quad t \geq 1.
\end{equation*}
Correspondingly, we can expect the asymptotic behavior of a solution $v_{inh}(t)$ to 
\begin{equation} \label{equ:proof_ideas_inhomogeneous_model}
 (\pt - i \jtD) v_{inh} = \frac{1}{2i} \jtD^{-1} P_c \bigl( \alpha(\cdot) (v_{lin} + \bar{v}_{lin})^2 \bigr) \text{ on } \bbR^{1+1}
\end{equation}
to be determined by the contributions of three source terms given in Duhamel form by 
\begin{equation} \label{equ:proof_ideas_source_terms}
 \frac{c_0^2}{2i} \int_1^t e^{i(t-s) \jtD} \jtD^{-1} P_c (\alpha \varphi^2) \frac{1}{s} \Bigl( e^{i \frac{\pi}{2}} e^{2is} (\langle \varphi, v_0 \rangle )^2 + 2 |\langle \varphi, v_0 \rangle|^2 + e^{-i\frac{\pi}{2}} e^{-2is} ( \overline{\langle \varphi, v_0 \rangle} )^2 \Bigr) \, \ud s.
\end{equation}
Due to the non-integrable time decay $s^{-1}$ of these source terms, the overall time oscillations in the integrand ultimately determine the long-time behavior of $v_{inh}(t)$. This becomes particularly transparent on the distorted Fourier side, where the overall oscillations in time $s$ for the three source terms in the parentheses in the integrand of~\eqref{equ:proof_ideas_source_terms} are respectively given by $e^{is(2-\jxi)}$, $e^{-is\jxi}$, and $e^{-is(2+\jxi)}$, where $\jxi = (1+\xi^2)^\hf$. While the last two have good oscillatory behavior at all frequencies, the phase of $e^{is(2-\jxi)}$ vanishes when $2 - \jxi = 0$, i.e. at frequencies $\xi = \pm \sqrt{3}$. 
In the non-resonant case this is offset by the vanishing of $\widetilde{\calF}[ \alpha \varphi^2 ](\pm \sqrt{3}) = 0$ at these specific frequencies. However, in the resonant case, where $\widetilde{\calF}[ \alpha \varphi^2 ](\sqrt{3}) \neq 0$ or $\widetilde{\calF}[ \alpha \varphi^2 ](-\sqrt{3}) \neq 0$, these observations indicate that the long-time behavior of $v_{inh}(t)$ decomposes into the contribution of a resonant source term of the form
\begin{equation*}
 \frac{c_0^2}{2i} \int_1^t e^{i(t-s) \jtD} P_c (\alpha \varphi^2) \frac{e^{2is}}{s} \, \ud s,
\end{equation*}
and a bulk term that can be expected to asymptotically behave like a free Klein-Gordon wave. 

It turns out that the study of the asymptotic behavior of the nonlinear solution $v(t)$ to~\eqref{equ:proof_ideas_nlkg} can effectively be reduced to the above heuristics. The key step to achieve this reduction is to identify the precise leading order behavior of the variable coefficient quadratic nonlinearity $\alpha(x) (v(t) + \bar{v}(t))^2$. To this end we introduce the function 
\begin{equation} 
 w(t) := c_0 \frac{e^{i\frac{\pi}{4}} e^{it}}{t^\hf} \langle \varphi, v_0 \rangle \varphi + \frac{1}{2i} \int_0^{t-1} c_0 \frac{e^{i\frac{\pi}{4}} e^{i(t-s)}}{(t-s)^{\hf}} \bigl\langle \varphi, \alpha(\cdot) \bigl( v(s) + \bar{v}(s) \bigr)^2 \bigr\rangle \varphi \, \ud s, \quad t \geq 1,
\end{equation}
that can perhaps be thought of as a ``projection'' of the nonlinear solution $v(t)$ to the zero energy resonance $\varphi(x)$. Note that we may write $w(t,x) = a(t) \varphi(x)$ with the time-dependent coefficient 
\begin{equation} 
 a(t) := c_0 \frac{e^{i\frac{\pi}{4}} e^{it}}{t^\hf} \langle \varphi, v_0 \rangle + \frac{1}{2i} \int_0^{t-1} c_0 \frac{e^{i\frac{\pi}{4}} e^{i(t-s)}}{(t-s)^{\hf}} \bigl\langle \varphi, \alpha(\cdot) \bigl( v(s) + \bar{v}(s) \bigr)^2 \bigr\rangle \, \ud s, \quad t \geq 1.
\end{equation}
In Proposition~\ref{prop:local_decay_bounds} we establish the following local decay bounds for the nonlinear solution $v(t)$ to~\eqref{equ:proof_ideas_nlkg} via a bootstrap argument
\begin{equation*} 
 \begin{aligned}
   \sup_{t\in\bbR} \, \Bigl\{ \jt^{-(0+)} \| v(t) \|_{H^2_x} &+ \jt^\hf \|\jx^{-\sigma} v(t)\|_{L^2_x} + \jt \|\jx^{-\sigma} (1-\chi_0(H)) v(t)\|_{L^2_x} \\
   &\quad \quad + \jt \|\jx^{-\sigma} \sqrt{H} v(t)\|_{L^2_x} + \jt \|\jx^{-\sigma} \pt (e^{-it} v(t))\|_{L^2_x} \Bigr\} \lesssim \varepsilon.
 \end{aligned}
\end{equation*}
The proof of Proposition~\ref{prop:local_decay_bounds} crucially exploits the spatial localization of the coefficient $\alpha(x)$ in conjunction with several local decay estimates for the Klein-Gordon propagator $e^{it\jtD} P_c$ summarized in Corollary~\ref{cor:main}, in particular~\eqref{equ:proof_ideas_local_decay_est_subtr_off}. 
While $v(t)$ has the slow local decay $\|\jx^{-\sigma} v(t)\|_{L^2_x} \lesssim \varepsilon t^{-\hf}$, we then conclude in Corollary~\ref{cor:local_decay_improved_low_energy_v_minus_w} that the difference $v(t) - w(t)$ enjoys the faster local decay 
\begin{equation*}
 \| \jx^{-\sigma} ( v(t) - w(t) ) \|_{L^2_x} \lesssim \frac{\varepsilon}{t}, \quad t \geq 1.
\end{equation*}
The local decay bounds on $v(t)$, in particular the faster local decay of the time derivative of the ``phase-filtered'' component $e^{-it} v(t)$ given by $\|\jx^{-\sigma} \pt ( e^{-it} v(t) ) \|_{L^2_x} \lesssim \varepsilon t^{-1}$ enables us in Corollary~\ref{cor:asymptotics_a} to extract the asymptotics of the time-dependent coefficient $a(t)$ given by 
\begin{equation*}
 a(t) = c_0 \frac{e^{i \frac{\pi}{4}} e^{it}}{t^\hf} a_0 + \calO_{L^\infty_t}\Bigl( \frac{1}{t} \Bigr), \quad t \geq 1,
\end{equation*}
with $a_0$ defined in~\eqref{equ:formula_a0}. This suggests that the leading order behavior of the quadratic nonlinearity $\alpha(x) (v(t) + \bar{v}(t))^2$ is of the form 
\begin{equation*}
 \alpha(x) \varphi(x)^2 ( a(t) + \bar{a}(t) )^2 = c_0^2 \alpha(x) \varphi(x)^2 \frac{1}{t} \bigl( e^{i \frac{\pi}{2}} e^{2it} a_0^2 + 2 |a_0|^2 + e^{-i\frac{\pi}{2}} e^{-2it} \bar{a}_0^2 \bigr) + \calO_{L^\infty_t}\Bigl( \frac{1}{t^{\frac32}} \Bigr).
\end{equation*}
Analogously to the preceding discussion of the simplified equation~\eqref{equ:proof_ideas_inhomogeneous_model}, we are reduced to analyzing the asymptotic behavior of the contribution of the possibly resonant source term 
\begin{equation*}
 v_{mod}(t) := c_0^2 \frac{a_0^2}{2} \int_1^t e^{i(t-s)\jtD} \jtD^{-1} P_c \bigl( \alpha \varphi^2 \bigr) \frac{e^{2is}}{s} \, \ud s,
\end{equation*}
and of the bulk term $v_{free}(t) := v(t) - v_{mod}(t)$. 

The derivation of the asymptotic behavior of $v_{mod}(t)$ and $v_{free}(t)$ asserted in Theorem~\ref{thm:thm1} is carried out in Section~\ref{sec:proof_theorem}. It combines the local decay bounds for $v(t)$ established in Proposition~\ref{prop:local_decay_bounds} and Corollary~\ref{cor:local_decay_improved_low_energy_v_minus_w} with pointwise linear estimates and asymptotics for the propagator $e^{it\jtD} P_c$ established in Lemma~\ref{lem:pw decay} and in Lemma~\ref{lem:asymptotics_KG}. In particular, in the resonant case a careful stationary phase analysis of the asymptotics of $v_{mod}(t,x)$ reveals the logarithmic slow-down~\eqref{equ:thm1_resonant_asymptotics_along_special_rays} of the decay rate of $v_{mod}(t)$ along the rays $\frac{x}{t} = \mp \frac{\sqrt{3}}{2}$ that are associated with the resonant frequencies $\xi = \pm \sqrt{3}$ for which the phase of $e^{i s (2-\jxi)}$ vanishes.

This concludes a sketch of some of the main ideas entering the proof of Theorem~\ref{thm:thm1}.

\subsection{Notation and conventions}

For non-negative $X$, $Y$ we write $X \lesssim Y$ or $X = \calO(Y)$ if $X \leq C Y$ for some constant $C > 0$. 
We employ the notation $X \lesssim_\nu Y$ to indicate that the implicit constant depends on a parameter $\nu$ and we write $X \ll Y$ if the implicit constant should be considered as small. 
Further, we use the japanese bracket notation $\langle x \rangle = (1+x^2)^\hf$, $\langle t \rangle = (1+t^2)^\hf$, and $\langle \xi \rangle = (1+\xi^2)^\hf$.
For a real number $b \in \bR$ we denote by $b+$, respectively by $b-$, a number that is larger, respectively smaller, than $b$, but that can be taken arbitrarily close to $b$.

Throughout we denote by $\chi_0(\xi)$ a smooth cutoff to $|\xi| \lesssim 1$, equal to $1$ near $\xi = 0$. Moreover, we denote by $\chi(\xi)$ a smooth bump function with support near $|\xi| \simeq 1$.

We denote the inner $L^2_x$ product by $\langle f, g \rangle := \int_{\bbR} \overline{f(x)} g(x) \, \ud x$,
and we denote the ``projection'' onto the resonance $\varphi$ by
\begin{equation} \label{eq:Pphi}
 (\varphi \otimes \varphi) g := \langle \varphi, g \rangle \varphi.
\end{equation}
We use the notation $\widetilde{f}(\xi) = \widetilde{\calF}[f](\xi)$ for the distorted Fourier transform associated with $H = -\px^2 + V$.
Finally, we work with the following definition for the Sobolev spaces $H^k_x(\bbR)$, $k = 1, 2$, given by
\begin{align*}
 \|g\|_{H^k_x} &:= \sum_{j=0}^k \| \px^j g\|_{L^2_x}.
\end{align*}

\section{Spectral and Scattering Theory} \label{sec:spectral_scattering_theory}

This section is devoted to the study of the linear flow generated by the Klein-Gordon equation with a potential. In other words,  we investigate the linear PDE
$(\partial_t - i\jap{\tilde D})v=P_c f$ with datum $v(0)=P_c v_0$ where $v_0$ lies in suitable weighted Sobolev spaces. Recall that $\jap{\tilde D}$ is the nonnegative operator with the property $\jap{\tilde D}^2P_c=(1+H)P_c$ and $H=-\partial_x^2+V$ with
$V$ and finitely many of its derivatives decaying at a sufficiently rapid polynomial rate. Moreover, $P_c$ is the projection onto the continuous spectrum of~$H$. Throughout, we will focus on the case where $H$ exhibits a $0$ energy resonance which is commonly referred to as the {\em non-generic case}. A $0$ energy resonance simply means that there is a globally bounded nontrivial solution of $Hf=0$. Or, equivalently, that the bounded solution as $x\to-\infty$, which is unique up to a nonzero constant, is linearly dependent with its cousin which remains bounded as $x\to+\infty$. All of this is equivalent with the Laurent expansion of the resolvent $(H- z^2)^{-1}$ around $z=0$ in the upper half plane $\Im \, z>0$, starts  with a $z^{-1}$ power. The Laurent expansion needs to be understood in the  weighted $L^2(\bR)$ sense, and the coefficient of $z^{-1}$ is a rank-$1$ operator given by~\eqref{eq:Pphi}. In the generic case, there is no singular power in this expansion. 

The easier generic $H$ is essentially a special case of our analysis and statements relevant to it can be obtained by carrying out straightforward modifications.  
An important technical device in our estimates is the {\em distorted Fourier transform}.  This refers to the map $f\to \tilde f(\xi):=\int f(x) \overline{e(x,\xi)}\, \ud x$, and its inverse $f(x) = \int \tilde f(\xi) {e(x,\xi)}\, \ud \xi$ which holds for all $f\in L^1\cap L^2(\bR)$ which are perpendicular to all eigenfunctions of~$H$. Here $H e(\cdot,\xi)=\xi^2 e(\cdot,\xi)$ suitably normalized so that Plancherel holds with spectral measure $\ud\xi$, i.e., $\|f\|_2=\|\tilde f\|_2$. In the non-generic case the distorted Fourier basis is discontinuous at $\xi=0$. We therefore do not use it for small energies but rather directly work with the resolvent (Green function). 

\subsection{Spectral theory and distorted Fourier transform}

This subsection recalls the Jost solutions, and the standard Volterra perturbation theory needed to construct them. 

\begin{definition} \label{def:H}
Fix two positive integers $N_0$ and $M_0$, both exceeding $2$. 
We consider $H = -\px^2 + V$ on the domain $C^2_{comp}(\R)\subset L^2(\R)$ with real-valued continuous $V\in L^\infty(\R)\cap C^{M_0}(\R)$, and $\jap{x}^{N_0} V^{(\ell)}(x) \in L^1(\R)$ for all $0\le\ell\le M_0$. The Friedrichs extension of $H$ is self-adjoint with domain $H^2(\R)$. 
\end{definition}

For such $V$, it is a standard fact that the spectrum of $H$ consists of $[0,\infty)$, which is essential spectrum, and finitely many negative simple eigenvalues, more precisely the number of eigenvalues must be less than or equal to $1 + \int_\bbR |x| |V(x)| \, \ud x$, see~\cite[p.~149]{DeiTru}. 
Moreover, the spectrum on $[0,\infty)$ is absolutely continuous, which follows from the usual explicit representation of the Green function, i.e., the kernel of the resolvent $(H- z^2)^{-1}$ as $\Im \, z \to 0+$, see Lemma~\ref{lem:stone}. As already mentioned, $0$ energy occupies a special role here and the resolvent may or may not be singular around $z=0$. The latter is {\em generic}, whereas the former is {\em non-generic}.  
It is worth mentioning that $0$ cannot be an eigenvalue under our assumptions on~$V$, since the solutions $f$ of $Hf=0$ can only approach constants but not decay as $x\to\pm\infty$. It can only be a resonance.

We now begin the technical work by recalling basic notions of scattering theory on the line. See~\cite{DeiTru} for much sharper statements. Throughout, constants of the form $C(V)$ depend on $V$  only via the norms $\|\jap{x}^{N_0} V^{(\ell)}(x) \|_{L^1}$ for $0\le \ell\le M_0$. Constants may also depend on  the resonance function $\varphi$, see Definition~\ref{def:reson} below. The latter is only relevant for estimates involving $0$ energy. 

\begin{lemma}
\label{lem:fm}
Let $V(x)$ be as in Definition~\ref{def:H}. 
There exist unique solutions $f_{\pm}(x,\xi)$ for every $\xi\in\R$  of $$H f_{\pm}(\cdot,\xi) = \xi^2 f_{\pm}(\cdot,\xi)$$ satisfying $f_{\pm}(x,\xi)\sim e^{\pm i x\xi}$ as $x\to\pm\infty$. 
They are of the form $f_{\pm}(x,\xi) = e^{\pm i x\xi} m_{\pm}(x,\xi)$ where $m_{\pm}(x,\xi)\sim 1$ as $x\to\pm\infty$,  
and one has the bounds $|\partial_\xi^\ell\partial_x^k m_\pm(x,\xi)|\le C$ for all $0\le k\le M_0$, $0\le \ell\le N_0-1$  uniformly in $\pm x\ge0, \xi\in\R$.  
\end{lemma}
\begin{proof}
We solve the ODEs
\[
m_{\pm}''(x,\xi) \pm 2i\xi m'(x,\xi) = V(x) m_{\pm}(x,\xi)
\]
by means of the Volterra equation
\EQ{\label{eq:m+}
m_{+}(x,\xi) &= 1 + \int_x^\infty \int_0^{y-x} e^{2i\xi t}\, dt\; V(y) m_+(y,\xi)\, dy \\
& = 1 + \int_x^\infty \frac{e^{2i\xi(y-x)}-1}{2i\xi} V(y) m_+(y,\xi)\, dy
}
By iteration one finds that  for all $x\ge0$ and uniformly in $\xi\in\R$
\EQ{ \label{eq:m+-1}
|m_{+}(x,\xi) -1|\le e^{\gamma(x)}-1,\quad  \gamma(x):=\int_x^\infty y|V(y)|\, dy
}
An analogous bound holds for $|m_{-}(x,\xi) -1|$ on $x\le0$. Next, differentiating \eqref{eq:m+} in~$\xi$ yields
\EQ{\label{eq:m+pl}
\partial_\xi m_{+}(x,\xi) &= 2i \int_x^\infty \int_0^{y-x} e^{2i\xi t}\, t\, dt\; V(y) m_+(y,\xi)\, dy \\
& \quad  + \int_x^\infty \frac{e^{2i\xi(y-x)}-1}{2i\xi} V(y) \partial_\xi m_+(y,\xi)\, dy
}
whence, with $|m_{+}(x,\xi)|\le M$ for all $x\ge0$ and $\xi\in\R$, 
\EQ{\label{eq:m+plbd}
|\partial_\xi m_{+}(x,\xi) | & \le  \int_x^\infty (y-x)^2 |V(y)| M\, dy + \int_x^\infty (y-x)|V(y)| | \partial_\xi m_+(y,\xi)|\, dy \\
& \le \eta(x) e^{\gamma(x)}, \qquad \eta(x):= \int_x^\infty y^2 |V(y)|\, dy
}
where the last line follows by iteration. 
Similarly one checks that 
\[
|\partial_\xi^2  m_{+}(x,\xi) |  \le C \int_x^\infty (1+y^3) |V(y)|\, dy
\]
for all $x\ge0$, $\xi\in\R$, and $C=C(V)$.  The higher $\xi$ derivatives are handled analogously. Note that in particular $f_{\pm}(x,\xi)$ are continuous in $(x,\xi)\in\R^2$. For the derivatives in~$x$ we compute 
\EQ{\label{eq:m+ x}
\partial_x m_{+}(x,\xi) &=  \int_x^\infty \partial_x \frac{e^{2i\xi(y-x)}-1}{2i\xi} V(y) m_+(y,\xi)\, dy \\
& = -  \int_x^\infty \partial_y \frac{e^{2i\xi(y-x)}-1}{2i\xi} V(y) m_+(y,\xi)\, dy \\
& =  \int_x^\infty  \frac{e^{2i\xi(y-x)}-1}{2i\xi} (V'(y) m_+(y,\xi)+V(y) \partial_y m_+(y,\xi))\, dy
}
which implies the uniform boundedness of $\partial_x m_{+}(x,\xi)$ in $x\ge0$, $\xi\in\R$ from $y V'(y)\in L^1$. The higher $x$-derivatives follow by repeating this procedure. For the mixed derivatives, we combine the two Volterra methods. 
\end{proof}

Next, we establish symbol-type behavior for large~$\xi$. Throughout, $m'_{\pm}=\partial_x m_{\pm}$. 

\begin{lemma}
\label{lem:m symb}
With $V$ as in Definition~\ref{def:H}, and with $m_\pm$ as in Lemma~\ref{lem:fm},
\EQ{\label{eq:der mpm}
\sup_{\pm x\ge0} |\partial_\xi^j m_{\pm}(x,\xi)|\le C(V,\xi_0) |\xi|^{-1-j}
}
for all $|\xi|\ge\xi_0>0$ and $1\le j\le N_0$. Furthermore, the same bound holds for $x$ derivatives:
\EQ{\label{eq:der m'pm}
\sup_{\pm x\ge0} |\partial_\xi^j m_{\pm}'(x,\xi)|\le C(V,\xi_0) |\xi|^{-1-j}
}
for all $|\xi|\ge\xi_0>0$ and $1\le j\le N_0$. 
\end{lemma}
\begin{proof}
We freeze $\xi_0>0$ and allow constants to depend on it. 
Returning to the Volterra equation~\eqref{eq:m+}, we compute 
\EQ{\nn
\partial_\xi m_{+}(x,\xi) &= - \int_x^\infty \frac{e^{2i\xi(y-x)}-1}{2i\xi^2} V(y) m_+(y,\xi)\, dy \\
&\quad + \int_x^\infty \frac{\partial_y \, e^{2i\xi(y-x)}}{2i\xi^2} (y-x) V(y) m_+(y,\xi)\, dy  + \int_x^\infty \frac{e^{2i\xi(y-x)}-1}{2i\xi} V(y) \partial_\xi m_+(y,\xi)\, dy\\
&= - \int_x^\infty \frac{e^{2i\xi(y-x)}-1}{2i\xi^2} V(y) m_+(y,\xi)\, dy \\
&\quad - \int_x^\infty \frac{  e^{2i\xi(y-x)}}{2i\xi^2}\partial_y\big[ (y-x) V(y) m_+(y,\xi)\big]\, dy  + \int_x^\infty \frac{e^{2i\xi(y-x)}-1}{2i\xi} V(y) \partial_\xi m_+(y,\xi)\, dy
}
whence, by the bounds of Lemma~\ref{lem:fm}, and Volterra iteration, 
\[
\sup_{x\ge0} |\partial_\xi m_{+}(x,\xi)|\le C(V) \xi^{-2}
\]
Repeating this procedure yields 
\[
\sup_{x\ge0} |\partial_\xi^2 m_{+}(x,\xi)|\le C(V) |\xi|^{-3}
\]
and similarly for the third and higher derivatives and $m_-$ on $x\le0$. For the $x$-derivatives, we start from \eqref{eq:m+ x} to conclude that 
\EQ{\label{eq:m' dxi}
\partial_\xi m_{+}'(x,\xi) 
& =  -\int_x^\infty  \frac{e^{2i\xi(y-x)}-1}{2i\xi^2} (V'(y) m_+(y,\xi)+V(y) m_+'(y,\xi))\, dy \\
& \quad -   \int_x^\infty  \frac{e^{2i\xi(y-x)}-1}{2i\xi^2} \partial_y\big[(y-x)(V'(y) m_+(y,\xi)+V(y) m_+'(y,\xi))\big]\, dy \\
& \quad +   \int_x^\infty  \frac{e^{2i\xi(y-x)}-1}{2i\xi} (V'(y) \partial_\xi m_+(y,\xi)+V(y) \partial_\xi m_+'(y,\xi))\, dy
}
This yields via Volterra iteration, our assumptions on $V$, and the bound on $\partial_\xi m_+(y,\xi)$, that 
\[
\sup_{x\ge0} |\partial_\xi m_{+}'(x,\xi)|\le C\xi^{-2} 
\]
Taking higher  $\xi$ derivatives of \eqref{eq:m' dxi} concludes the proof. 
\end{proof}

The asymptotics of $f_{+}(x,\xi)$ as $x\to-\infty$, are expressed via the scattering data. In fact, 
\EQ{\label{eq:TR}
T(\xi) f_+(\cdot,\xi) & = f_-(\cdot,-\xi) + R_-(\xi) f_-(\cdot,\xi) \\
T(\xi) f_-(\cdot,\xi) & = f_+(\cdot,-\xi) + R_+(\xi) f_+ (\cdot,\xi) 
}
where $T(\xi)W(f_+(\cdot,\xi), f_-(\cdot,\xi))={-2i\xi}$ with $W = W(\xi)$ being the Wronskian. By Lemmas~\ref{lem:fm} and~\ref{lem:m symb}, $W\in C^{N_0-1}(\R)$. The scattering matrix
\EQ{\label{eq:SU2}
S(\xi) = \left[ \begin{matrix} T(\xi) & R_-(\xi) \\ R_+(\xi) & T(\xi)
\end{matrix} \right] 
}
is unitary.
We now formally introduce the class of non-generic potentials that we consider. 

\begin{definition} \label{def:reson}
 We assume that $H$ exhibits a $0$-energy resonance, i.e., that $H$ is non-generic. This means that $W(0)=0$, equivalently, $T(0)\ne0$, and $f_+(x,0)\sim c\ne0$ as $x\to-\infty$. Thus, there exists a nonzero solution of $H\varphi=0$ with $\varphi\in L^\infty(\R)$, $\varphi\ne0$,  normalized so that $\varphi(x)=f_+(x,0)$ approaches~$1$ as $x\to\infty$ and a nonzero constant as $x\to-\infty$.
\end{definition}

The following lemma collects the analytic properties of the transmission and reflection coefficients that are needed later. 

\begin{lemma} \label{lem:T}
The transmission coefficient satisfies 
$T\in C^{N_0-1}(\R\setminus \{0\})\cap C^{N_0-2}(\R)$ and 
\EQ{\label{eq:T asymp}
T(\xi) = 1 + \calO(\xi^{-1}),\qquad |\xi| \to \infty
}
where $\partial_\xi^j \calO(\xi^{-1}) = \calO(\xi^{-1-j})$ as $|\xi|\to\infty$ for all $0\le j\le N_0-1$. Furthermore, $T\ne0$ everywhere, $T$ is bounded with its first $N_0-2$ derivatives on $\R$, and its first $N_0-1$ derivatives on $|\xi|\ge\xi_0$ for any fixed $\xi_0>0$. 
The reflection coefficients satisfy $R_\pm \in C^{N_0-1}(\R\setminus \{0\})\cap C^{N_0-2}(\R)$ and 
\begin{equation} \label{eq:Rpm asymp}
 R_\pm(\xi) = \calO(\xi^{-1}), \qquad |\xi| \to \infty
\end{equation}
where $\partial_\xi^j \calO(\xi^{-1}) = \calO(\xi^{-1-j})$ as $|\xi|\to\infty$ for all $0\le j\le N_0-1$. Additionally, $R_\pm$ is bounded with its first $N_0-2$ derivatives on $\bbR$, and its first $N_0-1$ derivatives on $|\xi|\ge\xi_0$ for any fixed $\xi_0>0$. 
\end{lemma}
\begin{proof}
Writing $W(f_+(\cdot,\xi), f_-(\cdot,\xi))=:W(\xi)$, we have (with $m_\pm'(x,\xi) = \partial_x m_\pm(x,\xi)$)
\EQ{\nn 
W(\xi) &= m_+(0,\xi)(-i\xi m_-(0,\xi) + m_-'(0,\xi)) - m_-(0,\xi)(i\xi m_+(0,\xi) + m_+'(0,\xi)) \\
& = -2i\xi m_+(0,\xi) m_-(0,\xi) + m_+(0,\xi)m_-'(0,\xi) - m_-(0,\xi)m_+'(0,\xi)
}
By Lemma~\ref{lem:fm}, the final two terms are $\calO(1)$ uniformly in $\xi\in\R$, together with $N_0-1$ derivatives in~$\xi$. 
Thus, 
\EQ{\nn 
W(\xi) 
& = -2i\xi + 2i\xi(1-m_+(0,\xi) m_-(0,\xi)) + \calO(1)  \\
1-m_+(0,\xi) m_-(0,\xi) & = (1-m_+(0,\xi)) m_-(0,\xi) + 1 - m_-(0,\xi)  
}
By \eqref{eq:m+},  we have
\EQ{\nn 
 2i\xi (1-m_+(0,\xi)) & =  -\int_0^\infty \big( e^{2i\xi y}-1 \big) V(y) m_+(y,\xi)\, dy = \calO(1)
}
together with  $N_0-1$ derivatives in~$\xi$. In conclusion, 
\[
W(\xi) = -2i\xi  + \calO(1)
\]
as $\xi\to\pm\infty$ with $\calO(1)$ as above. Thus, \eqref{eq:T asymp} holds. Around $\xi=0$ we have
\[
W(\xi) = \xi \int_0^1 W'(s\xi)\, ds 
\]
whence $T(\xi) = -2i \Big( \int_0^1 W'(s\xi)\, ds \Big)^{-1}$. Since $|T(\xi)|\le 1$, we have $W'(0)\ne0$ and thus $T\in C^{N_0-2}(\R)$, and by $T(\xi)W(\xi)={-2i\xi}$, $T\ne0$ everywhere. For large $\xi$, we write 
\EQ{\nn 
T(\xi) &= \big[1+ r(\xi) \big]^{-1}, \quad T(\xi)= 1 -\frac{r(\xi)}{1+r(\xi)}   \\
r(\xi) &:= m_+(0,\xi) m_-(0,\xi) -1 +(-2i\xi)^{-1} ( m_+(0,\xi)m_-'(0,\xi) - m_-(0,\xi)m_+'(0,\xi))   
}
By the preceding, $r(\xi)=\calO(\xi^{-1})$ as $\xi\to\pm\infty$, and $r'(\xi)=\calO(\xi^{-2})$, $r''(\xi)=\calO(\xi^{-3})$, etc. which concludes the treatment of the transmission coefficient. 

Next we consider the reflection coefficient $R_-(\xi)$, the treatment of $R_+(\xi)$ being identical. Computing the Wronskian of $T(\xi) f_+(\cdot, \xi)$ and $f_-(\cdot, -\xi)$ in two different ways using~\eqref{eq:TR}, we find that
\begin{equation*}
 R_-(\xi) 2i\xi = T(\xi) W( f_+(\cdot, \xi), f_-(\cdot, -\xi) ) = T(\xi) \bigl( m_+(0,\xi) m_-'(0,-\xi) - m_+'(0,\xi) m_-(0,-\xi) \bigr).
\end{equation*}
We may then infer the asserted regularity and decay properties of $R_-(\xi)$ by proceeding as above.
\end{proof}

By the lemma, $T\in C^{N_0-2}(\R)$, $T\ne0$ everywhere,  and 
\EQ{\label{eq:TR*}
 f_+(x,\xi) & = T(\xi)^{-1} \big[f_-(x,-\xi) + R_-(\xi)  f_-(x,\xi) \big]\\
 f_-(x,\xi) & = T(\xi)^{-1} \big[ f_+(x,-\xi) + R_+(\xi) f_+ (x,\xi)  \big]
}
for all $(x,\xi)\in\R^2$. We will use \eqref{eq:TR*} for $f_+(x,\xi)$ on the half-axis $x\le0$, respectively, for $f_-(x,\xi)$ on $x\ge0$.

The Jost solutions $f_{\pm}$ give rise to the kernel of the resolvent on the real axis (approached from the upper half plane), and by Stone's formula, therefore also to the 
spectral measure on the positive half-axis. The starting point is the expression 
\EQ{\label{eq:green}
(H-(\xi^2+i0))^{-1} (x,y) = \frac{f_+(x,\xi) f_-(y,\xi)\one_{[x>y]}+ f_+(y,\xi) f_-(x,\xi)\one_{[y>x]}}{W(f_+(\cdot,\xi), f_-(\cdot,\xi))} 
}
for all $x,y\in\bR$ and $\xi\in\bR$. 

\begin{lemma}
\label{lem:stone}
The density of the spectral resolution $E(d\, \xi^2)$ of $H$ on the continuous spectrum $[0,\infty)$ has the kernel
\EQ{
\label{eq:Ekern}
\frac{E(d\, \xi^2)}{d\xi}(x,y) = \frac{|T(\xi)|^2}{2\pi } (f_+(x,\xi)f_+(y,-\xi)+f_-(x,\xi)f_-(y,-\xi)) 
}
for all $x,y\in\R$ and all $\xi\in\R$. Alternatively, 
\EQ{
\label{eq:Ekern*}
\frac{E(d\, \xi^2)}{d\xi}(x,y) & = \frac{1}{\pi }\left\{ \begin{array}{ll}
 \Re\big[ T(\xi) f_+(x,\xi)f_-(y,\xi) \big]  &\quad x>y \\
 \vspace{-4mm} \\
 \Re\big[ T(\xi) f_+(y,\xi)f_-(x,\xi) \big]  &\quad x<y
\end{array}\right.
}
and all $\xi\in\R$. 
\end{lemma}
\begin{proof}
By Stone's formula, for $x>y$, 
\EQ{\nn 
\frac{E(d\, \xi^2)}{d\xi}(x,y) &= \frac{\xi}{\pi i} \big( (H-(\xi^2+i0))^{-1} - (H-(\xi^2-i0))^{-1}\big)(x,y)\\
&= \frac{\xi}{\pi i}  ( W(\xi)^{-1} f_+(x,\xi) f_-(y,\xi) - W(-\xi)^{-1} f_+(x,-\xi) f_-(y,-\xi)) \\
&= \frac{1}{2\pi}( T(\xi) f_+(x,\xi) f_-(y,\xi) + T(-\xi) f_+(x,-\xi) f_-(y,-\xi)).
}
The final line here gives \eqref{eq:Ekern*}. On the other hand, using that $T(-\xi)=\overline{T(\xi)}$, $R_\pm(-\xi)=\overline{R_\pm(\xi)}$, and the unitarity of the scattering matrix in~\eqref{eq:SU2}, the last line here can be rewritten as 
\EQ{\nn
 \frac{E(d\, \xi^2)}{d\xi}(x,y) & = \frac{1}{2\pi } \big[ T(\xi)  f_+(x,\xi)(T(-\xi) f_+(y,-\xi) -  R_-(-\xi) f_-(y,-\xi))  \\
&\qquad\quad +  T(-\xi) (T(\xi) f_-(x,\xi) - R_+(\xi) f_+ (x,\xi))  f_-(y,-\xi) \big].
}
which is the desired expression \eqref{eq:Ekern} for $x>y$. By self-adjointness, one has 
$
\frac{E(d\, \xi^2)}{d\xi}(x,y) = \overline{\frac{E(d\, \xi^2)}{d\xi}(y,x)}
$
which concludes the proof. 
\end{proof}

The definition of the distorted Fourier transform $\wtcalF$ can now be read off from \eqref{eq:Ekern}. Indeed, we define the distorted Fourier basis as 
\EQ{\label{eq:calFdef}
e(x,\xi) &:= \frac{1}{\sqrt{2\pi}} \left \{ \begin{array}{ll} T(\xi)f_+(x,\xi) & \xi\ge0 \\
T(-\xi)f_-(x,-\xi) & \xi<0
\end{array}
\right. 
}
and define $\tilde f(\xi)=\wtcalF f(\xi) = \langle e(\cdot,\xi),f\rangle$ for all $f\in L^1\cap L^2(\bR)$. 
The reason for writing $f_-(x,-\xi) $ rather than $f_-(x,\xi) $ is due to the map $\R\to [0,\infty), \; \xi \mapsto \xi^2$ of multiplicity~$2$. We associate $f_+$ with the second cover, so to speak, of $[0,\infty)$ from $\xi>0$, and $f_-$ with the first cover by $\xi<0$. 
Then, one can read off  from \eqref{eq:Ekern} that
\[
\langle f,g\rangle = \langle \wtcalF f, \wtcalF g\rangle\quad \forall \; f,g \in L^1\cap L^2_c(\bR)
\]
where $L^2_c:=P_c L^2$. In other words, Plancherel holds, $\wtcalF:L^2_c\to L^2$ is an isometry, and $\wtcalF^*\wtcalF = \Id_{L^2_c}$. 
Explicitly, the inverse Fourier transform is given by 
\[
f(x) = \langle \overline{e(x,\cdot)}, \tilde f(\cdot)\rangle
\]
provided $\tilde f$ decays sufficiently for this inner product to exist. The reader will easily recover the standard Fourier transform for $V=0$. Note that for generic $V$, the Fourier basis vanishes at $\xi=0$, whereas in the non-generic case there is a discontinuity at zero energy.

\subsection{Sobolev and product estimates}

In this subsection we present several technical estimates, in particular a weighted Sobolev estimate and a product estimate, that will be needed in the nonlinear analysis in Sections~\ref{sec:local_decay_bounds} and~\ref{sec:proof_theorem}.

Recall that we denote by $\chi_0(\xi)$ a smooth cutoff to $|\xi| \lesssim 1$, equal to $1$ near $\xi = 0$, and that we denote by $\chi(\xi)$ a smooth bump function with support near $|\xi| \simeq 1$.
We assume throughout that $V$ is non-generic, although the following results also hold generically. 

\begin{lemma}[Kernel bounds]
\label{lem:K bds}
 Let $N \geq 1$ and $k\ge0$ be integers, and let $\nu \in \bbR$. Assume that the potential $V(x)$ satisfies $\jap{x}^{N+2}V^{(\ell)}\in L^1(\bR)$ for all $0\le\ell\le \max(k-1,1)$. Then we have for all $x,y \in \bbR$ that
 \begin{align}
  \bigl| \bigl[  \px^k \jap{\wtilD}^\nu \chi_0(H) P_c \bigr](x,y) \bigr| &\leq C(V, N, \nu, k) \sum_\pm \frac{1}{\langle x \pm y \rangle^N}, \label{equ:wtilD_chi0_kernel_bound} 
 \end{align}
 and the same holds for the kernel of $\sqrt{H} \chi_0(H) P_c$. 
 Moreover, we have for all $\lambda \geq 1$ and for all $x,y \in \bbR$ that
 \begin{align}
  \bigl| \bigl[ \jtD^\nu \chi(\wtilD/\lambda) P_c \bigr](x,y) \bigr| &\leq C(V, N, \nu) \sum_\pm \lambda^\nu \frac{\lambda}{\langle \lambda (x \pm y) \rangle^N}. \label{equ:jtD_chi_kernel_bound} \\
  \bigl| \bigl[ (\sqrt{H})^\nu \chi(\wtilD/\lambda) P_c \bigr](x,y) \bigr| &\leq C(V, N, \nu) \sum_\pm \lambda^\nu \frac{\lambda}{\langle \lambda (x \pm y) \rangle^N}. \label{equ:sqrtH_chi_kernel_bound} 
 \end{align}
\end{lemma}
\begin{proof}
We have 
\EQ{\nn
\big[ \jap{\wtilD}^\nu\chi_0(H) P_c\big](x,y) 
& = \frac{1}{2\pi }\one_{[x>y]} \int_{-\infty}^\infty   T(\xi) f_+(x,\xi)f_-(y,\xi) \,  \jap{\xi}^\nu\chi_0(\xi^2)\, \ud\xi    \\
& \quad +\frac{1}{2\pi } \one_{[x<y]} \int_{-\infty}^\infty   T(\xi) f_-(x,\xi)f_+(y,\xi)  \,  \jap{\xi}^\nu\chi_0(\xi^2)\, d\xi   \\
&=: K_1(x,y) + K_2(x,y) 
}
For the operator $\sqrt{H}\chi_0(H) P_c$, we would have a similar expression but with $\xi\chi_0(\xi^2)$ in place of $\jap{\xi}^\nu\chi_0(\xi^2)$. This does not have any significant bearing on the subsequent arguments. 
If $x\ge0$, then we write 
\EQ{\label{eq:K1xy}
K_1(x,y) &=  \frac{1}{2\pi }\one_{[x\ge0>y]} \int_{-\infty}^\infty e^{i (x-y)\xi}    m_+(x,\xi)m_-(y,\xi) \, \jap{\xi}^\nu \chi_0(\xi^2)T(\xi)\, \ud\xi  \\
&\quad + \frac{1}{2\pi }\one_{[x>y\ge0]} \int_{-\infty}^\infty e^{i (x-y)\xi}    m_+(x,\xi)m_+(y,-\xi) \, \jap{\xi}^\nu \chi_0(\xi^2)\, \ud\xi \\
&\quad +\frac{1}{2\pi }\one_{[x>y\ge0]} \int_{-\infty}^\infty e^{i (x+y)\xi}    m_+(x,\xi)m_+(y,\xi) \, \jap{\xi}^\nu R_+(\xi)\chi_0(\xi^2)\, \ud\xi 
}
Integrating by parts $N$ times gives the desired bound on $K_1(x,y)$ provided $\jap{x}^{N+2} V^{(\ell)}(x)\in L^1(\bR)$ for $\ell=0,1$, see Lemma~\ref{lem:fm}. The reason for the $N+2$ rather than $N+1$, and with the need to include $V'$ here lies with the transmission and reflection coefficients being ratios of Wronskians, see Lemma~\ref{lem:T}. 
The contributions of $x<0$ and the kernel $K_2$ are treated analogously. 
We now consider $k>0$ and without loss of generality $\nu=0$. The latter can be done since $\chi_0$ can be replaced with any bump function supported near~$0$, and we may thus absorb $\jap{\wtilD}^\nu$ into $\chi_0$.  We compute 
\EQ{\label{eq:px} 
\big[ \partial_x \chi_0(H) P_c\big](x,y) 
& = \frac{1}{2\pi }\one_{[x>y]} \int_{-\infty}^\infty   T(\xi)\partial_x f_+(x,\xi)f_-(y,\xi) \,  \jap{\xi}^\nu\chi_0(\xi^2)\, \ud\xi    \\
& \quad +\frac{1}{2\pi } \one_{[x<y]} \int_{-\infty}^\infty   T(\xi) \partial_x f_-(x,\xi)f_+(y,\xi)  \,  \jap{\xi}^\nu\chi_0(\xi^2)\, d\xi   
}
Notice that the $\pm\delta_0(x-y)$ singularities which arise by differentiating $\one_{[x>y]}$, resp.~$\one_{[x<y]}$, cancel each other. 
For $k=2$ we could differentiate once more. However, we write $\partial_x^2 = V-H$, whence 
\[
\big[\partial_x^2  \chi_0(H) P_c\big](x,y) = \big[ (V(x) -H)\chi_0(H) P_c\big](x,y)
\]
The right-hand side satisfies the bounds in \eqref{equ:wtilD_chi0_kernel_bound}, multiplied by the factor $1+\|V\|_\infty\le 1+\|V'\|_1$. The 
bound \eqref{equ:wtilD_chi0_kernel_bound} with $k=1$ requires  $\jap{x}^{N+1} V^{(\ell)} \in L^1(\bR)$ for $\ell=0,1$ by \eqref{eq:px} and Lemma~\ref{lem:fm}. The higher order derivatives in~$x$ follow by iteration, for example
\EQ{\nn 
\big[\partial_x^3  \chi_0(H) P_c\big](x,y) &= \partial_x \big[ (V(x) -H)\chi_0(H) P_c\big](x,y) \\
\big[\partial_x^4  \chi_0(H) P_c\big](x,y) &= \big[ (V(x) -H)^2\chi_0(H) P_c\big](x,y)
}
with $(V-H)^2 = V^2 - HV -VH + H^2=V^2 -2VH + V''+2V'\partial_x + H^2$. 

For \eqref{equ:jtD_chi_kernel_bound} consider the operator
\[
M_{\lambda}:= \lambda^{-\nu}  \jtD^\nu \chi(\tD/\lambda),\quad \lambda\ge1
\]
with the distorted Fourier representation 
\EQ{\label{eq:beta}
M_\lambda P_c g(x) 
& =  \frac{1}{\sqrt{2\pi}} \int_{0}^\infty   T(\xi)f_+(x,\xi) \lambda^{-\nu} \jap{\xi}^\nu \chi(\xi/\lambda) \tilde g(\xi)\, \ud\xi \\
&\quad + \frac{1}{\sqrt{2\pi}} \int_{-\infty}^0 T(-\xi) f_-(x,-\xi) \lambda^{-\nu} \jap{\xi}^\nu \chi(\xi/\lambda) \tilde g(\xi)\, \ud\xi  
}
Assuming $x>0$ and inserting the expression for the distorted Fourier transform $\tilde g$  yields the explicit kernel representation
\EQ{\label{eq:Mlam}
 &M_\lambda^+ P_c (x,y) \\
 &= \int_{0}^\infty  |T(\xi)|^2 e^{i(x-y)\xi} m_+(x,\xi) m_+(y,-\xi) \theta_+(y) \lambda^{-\nu} \jap{\xi}^\nu \chi(\xi/\lambda)  \, \frac{\ud\xi }{2\pi}  \\
 &\quad  +  \int_{0}^\infty \!\!\! T(\xi) e^{ix\xi} m_+(x,\xi)   \big(  e^{-iy\xi} m_-(y,\xi) + R_-(-\xi) e^{iy\xi} m_-(y,-\xi)\big) \theta_-(y) \lambda^{-\nu} \jap{\xi}^\nu \chi(\xi/\lambda)  \, \frac{\ud\xi}{2\pi}
}
for the contribution of the positive frequencies $\xi$ in~\eqref{eq:beta}. We leave the analogous contributions of negative $\xi$ to the reader. Combining the oscillatory integrals with phase $(x-y)\xi$ and rescaling $\xi=\lambda\eta$ leads to the estimate
\EQ{\nn
&\Big|\lambda \int_{0}^\infty  e^{i\lambda (x-y)\eta} m_+(x,\lambda\eta)  \Big\{  |T(\lambda\eta)|^2  m_+(y,-\lambda\eta) \theta_+(y) +  T(\lambda\eta) m_-(y,\lambda\eta)  \theta_-(y) \Big\} \lambda^{-\nu} \jap{\lambda \eta}^\nu \chi(\eta)  \, \frac{\ud\eta }{2\pi} \Big|\\
&\le C(V)\lambda\jap{\lambda (x-y)}^{-N}
} 
We integrated here by parts $N$ times  if $|\lambda (x-y)|\ge1$, otherwise just use the trivial bound $\lambda$. We use the symbol type bounds on $m_{\pm}$, $T,R_{\pm}$ in $\xi$ as above, uniform in the respective regimes of $x,y$, as well as that
\[ 
\lambda^{-\nu} \jap{\lambda \eta}^\nu \chi(\eta) = ( \lambda^{-2} + \eta^2)^{\frac{\nu}{2}} \chi(\eta)
\]
is bounded with all its derivatives uniformly in $\lambda\ge1$. The treatment of the phase function $(x+y)\xi$ is essentially the same, leading to the kernel bound $$C(V)\lambda\jap{\lambda (x+y)}^{-N}$$ 
This concludes the proof of \eqref{equ:jtD_chi_kernel_bound}. These bounds require that $\jap{x}^{N+1} V\in L^1$. 

The estimate \eqref{equ:sqrtH_chi_kernel_bound} follows in essentially the same way. The only difference being that in \eqref{eq:Mlam} the multiplier $\jap{\xi}$ is replaced with~$\xi$.  We leave the remaining details to the reader. 
\end{proof}

\begin{remark}
 In the nonlinear analysis in Sections~\ref{sec:local_decay_bounds} and~\ref{sec:proof_theorem}, we will occasionally use without further mentioning that the operators $\chi_0(H) P_c$ and $\px \chi_0(H) P_c$ are bounded on weighted $L^p(\bbR)$ spaces, $1 \leq p \leq \infty$, in the sense that  
 \begin{equation*}
  \bigl\| \jx^{-\sigma} \chi_0(H) P_c g \bigr\|_{L^p_x} + \bigl\| \jx^{-\sigma} \px \chi_0(H) P_c g \bigr\|_{L^p_x} \leq C(V,p,\sigma) \|\jx^{-\sigma} P_c g\|_{L^p_x}
 \end{equation*}
 for $\sigma \geq 0$. These bounds follow easily from the preceding kernel bounds~\eqref{equ:wtilD_chi0_kernel_bound} via Young's inequality.
\end{remark}

To carry out complex interpolation, it will be useful to also allow imaginary $\nu$ in Lemma~\ref{lem:K bds}. The  following lemma states the concrete estimate that we require for that purpose. 

\begin{lemma}
\label{lem:is}
Fix $\sigma\in\bR$ and assume $\jap{x}^{N+2} V^{(\ell)}\in L^1(\bR)$ for $\ell=0,1$ where $N>|\sigma|+1$ is an integer. Then 
\EQ{\label{eq:is}
\| \jx^\sigma  \jtD^{is}P_c  \jx^{-\sigma}\|_{2\to2} \le C(V,\sigma)\jap{s}^{N} 
}
for all $s\in\bR$. We also have
\[
\| \jx^{-\sigma} \jtD H^{-\frac12} (1-\chi_0(H)) P_c\|_{2\to2} \le C(V,\sigma)
\]
\end{lemma}
\begin{proof}
The limit
\EQ{\nn 
\jx^\sigma  \jtD^{is}P_c  \jx^{-\sigma} &= \lim_{L\to\infty} \jx^\sigma  \jtD^{is}\chi_0(H/L^2)P_c  \jx^{-\sigma}  \\
& =: \lim_{L\to\infty} A_L 
}
exists in the strong $L^2$  sense. 
The kernel of the operator $A_L$ is of the form
\EQ{\nn
A_L (x,y) 
& = \frac{1}{2\pi }\one_{[x>y]} \int_{-\infty}^\infty   T(\xi) f_+(x,\xi)f_-(y,\xi) \,  \jap{\xi}^{is}\chi_0(\xi^2/L^2)\, \ud\xi    \\
& \quad +\frac{1}{2\pi } \one_{[x<y]} \int_{-\infty}^\infty   T(\xi) f_-(x,\xi)f_+(y,\xi)  \,  \jap{\xi}^{is} \chi_0(\xi^2/L^2)\, d\xi   
}
and we bound it as in the previous proof, i.e., after $N$ integrations by parts we arrive at the upper bound
\EQ{\label{eq:ALkern}
\jx^\sigma |A_L(x,y)|  \jap{y}^{-\sigma}\le C\jap{s}^{N} \jx^\sigma   \jap{y}^{-\sigma}\max_{\pm} \jap{x\pm y}^{-N}
}
uniformly in $L$. Indeed, \eqref{eq:K1xy} now takes the form 
\EQ{\nn 
K_1(x,y) &=  \frac{1}{2\pi }\one_{[x\ge0>y]} \int_{-\infty}^\infty e^{i (x-y)\xi}    m_+(x,\xi)m_-(y,\xi) \, \jap{\xi}^{is} \chi_0(\xi^2/L^2)T(\xi)\, \ud\xi  \\
&\quad + \frac{1}{2\pi }\one_{[x>y\ge0]} \int_{-\infty}^\infty e^{i (x-y)\xi}    m_+(x,\xi)m_+(y,-\xi) \, \jap{\xi}^{is} \chi_0(\xi^2/L^2)\, \ud\xi \\
&\quad +\frac{1}{2\pi }\one_{[x>y\ge0]} \int_{-\infty}^\infty e^{i (x+y)\xi}    m_+(x,\xi)m_+(y,\xi) \, \jap{\xi}^{is} R_+(\xi)\chi_0(\xi^2/L^2)\, \ud\xi 
}
Integrating by parts at least twice in these expressions produces an absolutely convergent integrand, uniformly in $L$ (note $N\ge2$). 
Via Schur's test, \eqref{eq:ALkern} implies \eqref{eq:is} for $A_L$, and thus also in the limit $L\to\infty$. 

The final statement of the lemma is proved in exactly the same fashion, but with $\jap{\xi}\xi^{-1} (1-\chi_0(\xi^2))\chi_0(\xi^2/L^2)$ in place of $\jap{\xi}^{is}\chi_0(\xi^2/L^2)$. The same arguments go through with this symbol. 
\end{proof}

As an application of these bounds, we can now give a self-contained argument for the equivalence of weighted Sobolev norms defined via $H$, respectively, $H_0=-\partial_x^2$. 

\begin{lemma}[Equivalence of norms] \label{lem:equiv_norms}
 For $k = 1, 2$ and $\sigma \geq 0$, there exists a constant $C \equiv C(V, k, \sigma)$ such that
 \begin{equation}
 \label{eq:equiv norms}
  \frac{1}{C} \bigl\| \jx^\sigma \jtD^k P_c g \bigr\|_{L^2_x} \leq \| \jx^\sigma P_c g \|_{H^k_x} \leq C \bigl\| \jx^\sigma \jtD^k P_c g \bigr\|_{L^2_x}.
 \end{equation}
  The conditions on $V$ are the same as in Lemma~\ref{lem:is}. 
\end{lemma}
\begin{proof}
For $k=2$ we use $\jtD^2P_c=(1+H)P_c = (1+V -\partial_x^2)P_c$ which gives 
\EQ{\nn
 \bigl\| \jx^\sigma \jtD^2 P_c g \bigr\|_{L^2_x} &\le (1+\|V\|_\infty) \| \jx^\sigma P_c g \|_{L^2_x} + \| \jx^\sigma \partial_x^2 P_c g \|_{L^2_x}
\\
&\le C(1+\|V'\|_1)  \| \jx^\sigma P_c g \|_{H^2_x}
}
with a constant $C=C(\sigma)$. For the second inequality in \eqref{eq:equiv norms} we first note that 
\EQ{\nn
 \| \jx^\sigma P_c g \|_{H^2_x} &\les  \| \jx^\sigma \partial_x^2 P_c g \|_{L^2_x} +  \| \partial_x ( \jx^\sigma P_c g) \|_{L^2_x} +  \| \jx^\sigma  P_c g \|_{L^2_x}  \\
 &\les  \| \jx^\sigma \partial_x^2 P_c g \|_{L^2_x} + \eps \| \partial_x^2 ( \jx^\sigma P_c g) \|_{L^2_x}+ \eps^{-1} \| \jx^\sigma  P_c g \|_{L^2_x}  
}
The middle term here on the last line we move to the left-hand side for sufficiently small $\eps$, which yields
\EQ{\nn 
\| \jx^\sigma P_c g \|_{H^2_x} &\les  \| \jx^\sigma \partial_x^2 P_c g \|_{L^2_x} + \| \jx^\sigma  P_c g \|_{L^2_x}  \\
&\les \| \jx^\sigma \jtD^2 P_c g \|_{L^2_x} + \|  \jx^\sigma P_c g \|_{L^2_x}  
}
Now
\[
 \|  \jx^\sigma P_c g \|_{L^2_x}  =  \|R\, \jx^\sigma  \jtD^2 P_c g \|_{L^2_x} \les \| \jx^\sigma \jtD^2 P_c g \|_{L^2_x}
\]
since $R:= \jx^\sigma  \jtD^{-2}P_c  \jx^{-\sigma}:L^2(\bR)\to L^2(\bR)$ is bounded by Schur's test and the previous lemma. Indeed, performing a dyadic partition of unity, we can sum up the respective estimates in \eqref{equ:wtilD_chi0_kernel_bound} and \eqref{equ:jtD_chi_kernel_bound}  with $\nu=-2$, provided $N>\sigma+1$. 
This settles $k=2$ of the lemma, while $k=1$ follows by interpolation of this with $k=0$. To be more specific, we use complex interpolation which requires \eqref{eq:equiv norms} on both the vertical lines $is$ and $2+is$ with $s\in\bR$ with bounds that grow at most exponentially (say) in~$s$. In fact, 
Lemma~\ref{lem:is} allows us to extend the previous bounds from $0$, resp.~$2$, to the entire vertical lines through those points, with at most polynomial growth in $s$. 
\end{proof}

In the nonlinear analysis we will frequently use the following weighted Sobolev estimate. 

\begin{lemma}[Weighted Sobolev] \label{lem:weighted_sobolev}
 Fix $\sigma \geq 0$ and assume $\jx^{N+2} V^{(\ell)} \in L^1(\bbR)$ for $\ell = 0, 1$ where $N > \sigma + \frac12$ is an integer.
 For any $\mu > 0$ we have 
 \begin{equation} \label{equ:weighted_sobolev}
  \|\jx^{-\sigma} P_c g\|_{L^\infty_x} \leq C(V, \sigma, \mu) \bigl( \|\jx^{-\sigma} P_c g\|_{L^2_x} + \|\jx^{-\sigma} (\sqrt{H})^{\frac12 + \mu} P_c g\|_{L^2_x} \bigr).
 \end{equation}
\end{lemma}
\begin{proof}
 We begin by decomposing
 \begin{equation} \label{equ:sobolev_decomposition} 
  P_c g = \chi_0(H) P_c g + \sum_{j \geq 0} \chi(\wtilD/2^j) P_c g.
 \end{equation}
 We first treat the low-energy piece. Observe that we have uniformly for all $x,y \in \bbR$ that
 \begin{equation*}
  \jx^{-\sigma} \frac{1}{\jap{x\pm y}^N} \jap{y}^\sigma \lesssim \jx^{-\sigma} \frac{\jap{x\pm y}^\sigma + \jap{x}^\sigma}{\jap{x\pm y}^N} \lesssim \frac{1}{\jap{x\pm y}^{N-\sigma}}.
 \end{equation*}
 Hence, by the kernel bound~\eqref{equ:wtilD_chi0_kernel_bound} and by Young's inequality, we obtain the desired estimate 
 \begin{align*}
  \bigl\| \jx^{-\sigma} \chi_0(H) P_c g \bigr\|_{L^\infty_x} &\lesssim \sum_\pm \sup_{x\in\bbR} \int_{\bbR} \jx^{-\sigma} \frac{1}{\jap{x\pm y}^N} \jap{y}^\sigma \jap{y}^{-\sigma} |g(y)| \, \ud y \\
  &\lesssim \sum_\pm \sup_{x\in\bbR} \int_{\bbR} \frac{1}{\jap{x\pm y}^{N-\sigma}} \jap{y}^{-\sigma} |g(y)| \, \ud y \\
  &\lesssim \bigl\| \jap{\cdot}^{-(N-\sigma)} \bigr\|_{L^2_x} \|\jx^{-\sigma} g\|_{L^2_x} \\
  &\lesssim \|\jx^{-\sigma} g\|_{L^2_x}.
 \end{align*}
 Next we turn to the high-energy estimate. In the following we prove that for any $\lambda \geq 1$,
 \begin{equation} \label{equ:sobolev_high_energy_bound}
  \bigl\| \jx^{-\sigma} \chi(\wtilD/\lambda) P_c g \bigr\|_{L^\infty_x} \lesssim \lambda^{-\mu} \bigl\| \jx^{-\sigma} (\sqrt{H})^{\frac12 + \mu} P_c g \bigr\|_{L^2_x}.
 \end{equation}
 Then the asserted weighted Sobolev estimate~\eqref{equ:weighted_sobolev} follows from the decomposition~\eqref{equ:sobolev_decomposition} and the preceding bounds by summing over $j \geq 0$. 
 For the proof of~\eqref{equ:sobolev_high_energy_bound} we write 
 \begin{equation*}
  \jx^{-\sigma}  \chi(\wtilD/\lambda) P_c g = \jx^{-\sigma} (\sqrt{H})^{-\frac12-\mu} \chi(\wtilD/\lambda) P_c (\sqrt{H})^{\frac12+\mu} P_c g. 
 \end{equation*}
 Since we have uniformly for all $x,y \in \bbR$ and for all $\lambda \geq 1$ that
 \begin{equation*}
  \jx^{-\sigma} \frac{1}{\jap{\lambda(x \pm y)}^N} \jap{y}^\sigma \lesssim \jx^{-\sigma} \frac{\jap{x\pm y}^\sigma + \jap{x}^\sigma}{\jap{\lambda(x\pm y)}^N} \lesssim \jx^{-\sigma} \frac{\jap{\lambda(x\pm y)}^\sigma + \jap{x}^\sigma}{\jap{\lambda(x\pm y)}^N}  \lesssim \frac{1}{\jap{\lambda(x\pm y)}^{N-\sigma}},
 \end{equation*}
 the kernel bound~\eqref{equ:sqrtH_chi_kernel_bound} for $(\sqrt{H})^{-\frac12-\mu} \chi(\wtilD/\lambda) P_c$ together with Young's inequality yields the desired estimate
 \begin{align*}
  \bigl\| \jx^{-\sigma} \chi(\wtilD/\lambda) P_c g \bigr\|_{L^\infty_x} &\lesssim \sum_\pm \sup_{x\in\bbR} \int_{\bbR} \lambda^{-\frac12-\mu} \frac{\lambda}{\jap{\lambda(x\pm y)}^{N-\sigma}} \jap{y}^{-\sigma} \bigl| (\sqrt{H})^{\frac12 + \mu} P_c g(y) \bigr| \, \ud y \\
  &\lesssim \lambda^{\frac12 - \mu} \|\jap{\lambda \cdot}^{-(N-\sigma)}\|_{L^2_x} \bigl\| \jx^{-\sigma} (\sqrt{H})^{\frac12 + \mu} P_c g \bigr\|_{L^2_x} \\
  &\lesssim \lambda^{-\mu} \bigl\| \jx^{-\sigma} (\sqrt{H})^{\frac12 + \mu} P_c g \bigr\|_{L^2_x}
 \end{align*}
 and we are done. 
\end{proof}

The nonlinear analysis will require that we interchange standard derivatives with powers of $H$. 

\begin{lemma}[Weighted derivative bound] \label{lem:bound_px_by_sqrtH}
Under the same assumptions as in Lemma~\ref{lem:is} one has 
 \begin{equation}\label{eq:Wder}
  \bigl\| \jx^{-\sigma} \px  P_c g \bigr\|_{L^2_x} \lesssim \| \jx^{-\sigma} P_c g \|_{L^2_x} + \| \jx^{-\sigma} \sqrt{H} P_c g \|_{L^2_x}.
 \end{equation}
\end{lemma}
\begin{proof}
We claim that \eqref{eq:Wder} follows from the case $k=1$ of 
\begin{equation}\label{eq:Wder*}
  \bigl\| \jx^{-\sigma} \px^k  P_c g \bigr\|_{L^2_x} \lesssim \| \jx^{-\sigma} P_c g \|_{L^2_x} + \| \jx^{-\sigma} \jtD^k P_c g \|_{L^2_x}.
 \end{equation}
To see this, note that
 \EQ{\nn 
  \| \jx^{-\sigma} \jtD P_c g \|_{L^2_x}  &\leq  \| \jx^{-\sigma} \jtD \chi_0(H) P_c g \|_{L^2_x}  + \| \jx^{-\sigma} \jtD (1-\chi_0(H)) P_c g \|_{L^2_x} \\
  &\leq  \| \jx^{-\sigma} \jtD \chi_0(H) P_c \jx^{\sigma}\|_{2\to2} \|  \jx^{-\sigma} P_c g \|_{L^2_x}  \\
  &\quad + \| \jx^{-\sigma} \jtD H^{-\frac12} (1-\chi_0(H)) P_c\|_{2\to2} \| \jx^{-\sigma} \sqrt{H} P_c g \|_{L^2_x}
 }
 By Lemma~\ref{lem:K bds} and Schur's test, $\| \jx^{-\sigma} \jtD \chi_0(H) P_c \jx^{\sigma}\|_{2\to2}\leq C(V,\sigma)$, while 
 the operator norm in the last line is finite by Lemma~\ref{lem:is}. 
We now perform the following further reduction with $H_0:=-\partial_x^2$: 
\EQ{\nn
\bigl\| \jx^{-\sigma} \px  P_c g \bigr\|_{L^2_x} &\les \bigl\| \jx^{-\sigma} \px \chi_0(H_0) P_c g \bigr\|_{L^2_x}  + \bigl\| \jx^{-\sigma} \px  (1-\chi_0(H_0)) P_c g \bigr\|_{L^2_x}  \\
 &\les \bigl\| \jx^{-\sigma} \px \chi_0(H_0) P_c \jx^{\sigma}\bigr\|_{2\to2}  \|\jx^{-\sigma}  P_c g \bigr\|_{L^2_x}  + \bigl\| \jx^{-\sigma} \px  (1-\chi_0(H_0)) P_c g \bigr\|_{L^2_x}
}
The operator norm in the last line is bounded by the rapid off-diagonal decay of the kernel, and Schur's test. 
Therefore, it suffices to prove the $k=1$ case of 
\begin{equation}\label{eq:Wder**}
  \bigl\| \jx^{-\sigma} \px^k (1-\chi_0(H_0)) P_c g \bigr\|_{L^2_x} \lesssim \| \jx^{-\sigma} P_c g \|_{L^2_x} + \| \jx^{-\sigma} \jtD^k P_c g \|_{L^2_x}.
 \end{equation} 
Writing $\partial_x^2 = V-H$, the estimate \eqref{eq:Wder*} is obvious with $k=2$ since $\jtD^2=1+H$ and 
\[
 \bigl\| \jx^{-\sigma} \px^2  P_c g \bigr\|_{L^2_x} \le \|V'\|_1 \| \jx^{-\sigma} P_c g \|_{L^2_x} + \| \jx^{-\sigma} H P_c g \|_{L^2_x}
\]
As before, we can introduce the cut-off $1-\chi_0(H_0)$ on the left-hand side, hence \eqref{eq:Wder**} holds for $k=0$ and $k=2$.  Moreover, these bounds extend to the vertical lines $is$, resp.~$2+is$,  on the left-hand side only due to the fact that
$\| \jx^{-\sigma} \px^{is}  (1-\chi_0(H_0))\jx^{\sigma} \|_{2\to2} \le C(s)$ which grows polynomially as $|s|\to\infty$. 
Therefore, by complex interpolation we conclude that the desired bounds hold at $k=1$ and we are done. 
\end{proof}

We can now establish Leibniz rules as they appear in the nonlinear estimates. 

\begin{corollary}[Product estimates] \label{cor:product_est}
Fix $\sigma \geq 0$. Assume that $\jx^{N+2} V^{(\ell)} \in L^1(\bbR)$ for $\ell = 0, 1$ where $N > \sigma + 1$ is an integer. 
Then we have 
 \begin{equation} \label{equ:prod_est1}
  \begin{aligned}
   \bigl\| \jtD P_c \bigl( \alpha (P_c g) (P_c h) \bigr) \bigr\|_{L^1_x} &\lesssim \| \jx^{1+2\sigma} \alpha \|_{W^{1,\infty}_x} \Bigl( \|\jx^{-\sigma} P_c g\|_{L^2_x} +  \|\jx^{-\sigma} \sqrt{H} P_c g \|_{L^2_x} \Bigr) \times \\
   &\qquad \qquad \qquad \qquad \qquad \times \Bigl( \|\jx^{-\sigma} P_c h\|_{L^2_x} +  \|\jx^{-\sigma} \sqrt{H} P_c h\|_{L^2_x} \Bigr)
  \end{aligned}
 \end{equation}
 and 
 \begin{equation} \label{equ:prod_est2}
  \begin{aligned}
   \bigl\| \jtD P_c \bigl( \alpha (P_c g) (P_c h) \bigr) \bigr\|_{L^2_x} &\lesssim \| \jx^{2\sigma} \alpha \|_{W^{1,\infty}_x} \Bigl( \|\jx^{-\sigma} P_c g\|_{L^2_x} +  \|\jx^{-\sigma} \sqrt{H} P_c g \|_{L^2_x} \Bigr) \times \\
   &\qquad \qquad \qquad \qquad \qquad \times \Bigl( \|\jx^{-\sigma} P_c h\|_{L^2_x} +  \|\jx^{-\sigma} \sqrt{H} P_c h\|_{L^2_x} \Bigr).
  \end{aligned}
 \end{equation} 
\end{corollary}
\begin{proof}
 We give the proof of the first product estimate~\eqref{equ:prod_est1}, the proof of~\eqref{equ:prod_est2} being identical. By H\"older's inequality and by the equivalence of norms from Lemma~\ref{lem:equiv_norms}, we have that
 \begin{align*}
  \bigl\| \jtD P_c \bigl( \alpha (P_c g) (P_c h) \bigr) \bigr\|_{L^1_x} &\lesssim \|\jx^{-1}\|_{L^2_x} \bigl\| \jx \jtD P_c \bigl( \alpha (P_c g) (P_c h) \bigr) \bigr\|_{L^2_x} \lesssim \bigl\| \jx \alpha (P_c g) (P_c h) \bigr\|_{H^1_x}.
 \end{align*}
 Then by the standard product rule for the derivative and by H\"older's inequality, we obtain 
 \begin{equation} 
  \begin{aligned}
   \bigl\| \jx \alpha (P_c g) (P_c h) \bigr\|_{H^1_x} &\lesssim \| \jx^{1+2\sigma} \alpha \|_{W^{1,\infty}_x} \|\jx^{-\sigma} P_c g \|_{L^2_x} \|\jx^{-\sigma} P_c h \|_{L^\infty_x} \\
   &\quad + \| \jx^{1+2\sigma} \alpha \|_{L^\infty_x} \|\jx^{-\sigma} \px P_c g\|_{L^2_x} \|\jx^{-\sigma} P_c h\|_{L^\infty_x} \\
   &\quad + \| \jx^{1+2\sigma} \alpha \|_{L^\infty_x} \|\jx^{-\sigma} P_c g\|_{L^\infty_x} \|\jx^{-\sigma} \px P_c h\|_{L^2_x}. 
  \end{aligned}
 \end{equation}
 The product estimate~\eqref{equ:prod_est1} now follows from the weighted Sobolev inequality in Lemma~\ref{lem:weighted_sobolev} and the weighted derivative bound in Lemma~\ref{lem:bound_px_by_sqrtH}.
\end{proof}

Our final technical lemma arises in that part of the nonlinear analysis dealing with the {\em non-resonant} case of the main theorem. 

\begin{lemma} \label{lem:bound_sqrt3_vanishes}
 Fix $\sigma\in\bR$ and assume $\jap{x}^{N+2} V^{(\ell)}\in L^1(\bR)$ for $\ell=0,1$ where $N>|\sigma|+1$ is an integer.
 Let $\jx^{\sigma+3} g \in L^2_x(\bbR)$ and assume that
 $$\widetilde{\calF}[g](\pm \sqrt{3}) = 0.$$
 Then we have for $m \in \{0, \sigma\}$ and $\ell = 0, \pm 1$ that
 \begin{equation}\label{eq:m3Pcg} 
  \bigl\| \jx^m (2-\jtD)^{-1} \jtD^\ell P_c g \bigr\|_{L^2_x} \lesssim \| \jx^{m+3} g \|_{L^2_x}.
 \end{equation}
\end{lemma}
\begin{proof}
We introduce a smooth partition of unity $1=\chi_2(\xi)+\chi_3(\xi)$, where $\chi_2(\xi)$ vanishes outside a small neighborhood of the set $\{3\}$, and equals to~$1$ near $3$.  Then 
\[
 \bigl\| \jx^m (2-\jtD)^{-1} \jtD^\ell \chi_3(H) P_c \jx^{-m} g \bigr\|_{L^2_x} \le C(V,m) \| g \|_{L^2_x}
\]
by a small variant of the final statement in the proof of Lemma~\ref{lem:is}. It therefore suffices to prove \eqref{eq:m3Pcg}  for $\chi_2(H)P_c g$ in place of $P_cg$ on the left-hand side.  One has 
\EQ{ \label{eq:chi2H}
& (2-\jtD)^{-1} \jtD^\ell \chi_2(H) P_c g (x) \\
&= \frac{1}{\sqrt{2\pi}} \int_{0}^\infty   T(\xi)f_+(x,\xi)   (2-\jap{\xi})^{-1} \jap{\xi}^\ell  \chi_2(\xi^2) (\tilde g(\xi)-\tilde g(\sqrt{3}))\, \ud\xi \\
&\quad + \frac{1}{\sqrt{2\pi}} \int_{-\infty}^0 T(-\xi) f_-(x,-\xi)   (2-\jap{\xi})^{-1} \jap{\xi}^\ell  \chi_2(\xi^2)(\tilde g(\xi)-\tilde g(-\sqrt{3}))\, \ud\xi \\
&\equiv G_+(x) + G_-(x).
}
We denote by $\chi_4(\xi)$ a slight fattening of the bump function $\chi_2$ such that $\chi_4(\xi) = 1$ on the support of $\chi_2(\xi^2)$ and such that $\chi_4(\xi) = 0$ in a neighborhood of zero. Clearly, we may freely insert $\chi_4(\xi)$ in the integrands on the right-hand side of~\eqref{eq:chi2H}. We now prove the bound~\eqref{eq:m3Pcg} for $G_+(x)$, the bound for $G_-(x)$ being analogous. Using the assumption $\widetilde{g}(\sqrt{3}) = 0$, we write
\[
 \chi_4(\xi) \tilde{g}(\xi) = \chi_4(\xi) \tilde{g}(\xi) - \chi_4(\sqrt{3}) \tilde g(\sqrt{3}) = (\xi-\sqrt{3}) \int_0^1 (\chi_4 \tilde{g})'( s\xi+ (1-s)\sqrt{3}) \, ds.
\]
Then we observe that the function
\begin{equation*}
 (0, \infty) \ni \xi \mapsto (2-\jap{\xi})^{-1} \jap{\xi}^\ell  \chi_2(\xi^2) (\xi-\sqrt{3}) = - \frac{2+\jxi}{\sqrt{3}+\xi} \jxi^\ell \chi_2(\xi^2) 
\end{equation*}
is smooth and bounded on its support, with all of its derivatives bounded there as well. 
Hence, integrating by parts in $\xi$ for $N = m + 1$ times in the integrand of $G_+(x)$ and using the symbol type bounds on $m_\pm, T, R_\pm$ in $\xi$ uniformly in the respective regimes of $x$, we conclude that
\begin{equation*}
 |G_+(x)| \leq C(V, N) \jx^{-\sigma-1} \sup_{1 \leq n \leq \sigma+2} \bigl\| \partial_\xi^n (\chi_4 \widetilde{g})(\xi) \bigr\|_{L^\infty_\xi} \leq C(V, N) \jx^{-\sigma-1} \bigl\| \jx^{\sigma+2} g \bigr\|_{L^1_x}. 
\end{equation*}
This implies~\eqref{eq:m3Pcg} for $G_+(x)$, and finishes the proof of the lemma.
\end{proof}

\subsection{Decay estimates for the Klein-Gordon propagator} \label{subsec:decay_estimates}

We now establish local decay estimates as well as pointwise decay estimates for the linear Klein-Gordon flow relative to $H = -\px^2 + V(x)$, i.e., the propagator of $(\pt^2 + H + 1) u=0$.  We write $\jtD =\sqrt{1+H}$ on the positive spectrum of $H$. 

\begin{lemma}
\label{lem:low proj}  
Let $V \in L^\infty(\R) \cap C^{1}(\R)$ be real-valued, and assume that $\jap{x}^{6} V^{(\ell)}(x) \in L^1(\R)$ for all $0\leq \ell \leq 1$.
Let $\chi_0(\xi^2)$  be a smooth cutoff to $|\xi|\lesssim 1$, equal to $1$ near $\xi=0$ and set $w(x)=\langle x\rangle^4$. 
With $P_c$ the projection onto the continuous spectral subspace of~$H$, 
\begin{align}
 \Big\| w^{-1} \Bigl( e^{it \jtD} \chi_0(H) P_c \, g - c_0 \frac{e^{i\frac{\pi}{4}} e^{it}}{t^\frac12}  ( \varphi\otimes\varphi) g  \Bigr) \Big\|_{L^\infty_x} &\le \frac{C}{t^{\frac32}} \| w g\|_{L^1_x}, \qquad t \geq 1,   \label{eq:t32} \\
 \Bigl\| w^{-1} \partial_x \Bigl( e^{it\jtD} \chi_0(H) P_c g - c_0 \frac{e^{i\frac{\pi}{4}} e^{it}}{t^\frac12} ( \varphi\otimes\varphi) g \Bigr) \Bigr\|_{L^\infty_x} &\leq \frac{C}{t^{\frac32}} \| w g \|_{L^1_x}, \qquad t \geq 1.  \label{equ:local_decay_improved_subtr_off_with_derivative}
\end{align}
The real constant $c_0$ only depends on the scattering matrix $S(0)$ of the potential $V(x)$ at zero energy, cf.~\eqref{eq:SU2}, and is explicitly given by
\begin{equation} \label{equ:def_c0}
 c_0 = \frac{1}{(2\pi)^{\frac32}} \frac{T(0)^2}{1+R_-(0)}.
\end{equation}
More generally, let $\omega=\omega(\xi)$ be a function  bounded on the support of $\chi_0(\xi^2)$ with its derivatives up to order four. Then 
\begin{align}
 \Big\| w^{-1} \Bigl( e^{it \jtD} \omega(H) \chi_0(H) P_c \, g - c_0 \omega(0) \frac{e^{i\frac{\pi}{4}} e^{it}}{t^\frac12}  ( \varphi\otimes\varphi) g  \Bigr) \Big\|_{L^\infty_x} &\le \frac{C(\omega)}{t^{\frac32}} \| w g\|_{L^1_x}, \qquad t \geq 1,   \label{eq:t32 om}\\
 \Bigl\| w^{-1} \partial_x \Bigl( e^{it\jtD} \omega(H) \chi_0(H) P_c g - c_0 \omega(0) \frac{e^{i\frac{\pi}{4}} e^{it}}{t^\frac12} ( \varphi\otimes\varphi) g \Bigr) \Bigr\|_{L^\infty_x} &\leq \frac{C(\omega)}{t^{\frac32}} \| w g \|_{L^1_x}, \qquad t \geq 1.  \label{equ:local_decay_improved_subtr_off_with_derivative_om}
\end{align}
\end{lemma}
\begin{proof}
We first derive the local decay estimates~\eqref{eq:t32} and~\eqref{equ:local_decay_improved_subtr_off_with_derivative}. Afterwards we comment on the proofs of the generalized versions~\eqref{eq:t32 om} and~\eqref{equ:local_decay_improved_subtr_off_with_derivative_om}. Fix $g,h\in C_{comp}(\R)$. By Lemma~\ref{lem:stone}, and using that $T(-\xi)=\ol{T(\xi)}$, 
\begin{equation} \label{equ:representation_kg_prop_low_energy}
\begin{aligned}
\bigl( e^{it \jtD} \chi_0(H) P_c g \bigr)(x) &= \frac{1}{\pi } \int_0^\infty \int_{-\infty}^x e^{it\langle\xi\rangle} \Re\big[ T(\xi) f_+(x,\xi)f_-(y,\xi) \big] g(y) \, \ud y \, \chi_0(\xi^2) \, \ud \xi \\
& \quad +\frac{1}{\pi } \int_0^\infty \int_x^{\infty} e^{it\langle\xi\rangle} \Re \big[ T(\xi) f_-(x,\xi)f_+(y,\xi) \big] g(y) \, \ud y \, \chi_0(\xi^2) \, \ud \xi \\
&= \frac{1}{2\pi }  \int_{-\infty}^x \int_{-\infty}^\infty e^{it\langle\xi\rangle}  T(\xi) f_+(x,\xi)f_-(y,\xi) \, \chi_0(\xi^2) \, d\xi \, g(y) \, \ud y \\
& \quad +\frac{1}{2\pi }  \int_x^{\infty} \int_{-\infty}^\infty e^{it\langle\xi\rangle}  T(\xi) f_-(x,\xi) f_+(y,\xi) \, \chi_0(\xi^2) \, \ud \xi \, g(y) \, \ud y. 
\end{aligned}
\end{equation}
To isolate $0$ energy we rewrite these expressions in the form 
\begin{equation*}
 \begin{aligned}
  \bigl( e^{it \jtD} \chi_0(H) P_c g \bigr)(x) &= \frac{1}{2\pi }  \int_{-\infty}^x \int_{-\infty}^\infty e^{it\langle\xi\rangle} T(\xi) \chi_0(\xi^2)  \, \ud \xi \, f_+(x,0) f_-(y,0) g(y) \, \ud y \\
  &\quad + \frac{1}{2\pi } \int_x^\infty \int_{-\infty}^\infty e^{it\langle\xi\rangle}  T(\xi) \chi_0(\xi^2) \, \ud \xi \, f_-(x,0) f_+(y,0) g(y) \, \ud y \\
  &\quad + \frac{1}{2\pi } \int_{-\infty}^x \int_{-\infty}^\infty e^{it\langle\xi\rangle} T(\xi) \chi_0(\xi^2) \, F_>(x,y;\xi) \, \ud \xi \, g(y) \, \ud y \\
  &\quad + \frac{1}{2\pi } \int_x^\infty \int_{-\infty}^\infty e^{it\langle\xi\rangle}  T(\xi) \chi_0(\xi^2) \, F_<(x,y;\xi) \, \ud \xi \, g(y) \, \ud y,
 \end{aligned}
\end{equation*}
where
\EQ{\label{eq:Fdef}
F_{>}(x,y;\xi) &= f_+(x,\xi)f_-(y,\xi)  - f_+(x,0)f_-(y,0), \\
F_{<}(x,y;\xi) &=  f_-(x,\xi)f_+(y,\xi)  - f_-(x,0)f_+(y,0).
}
Taking the inner product with $h$, we obtain
\begin{equation}
 \begin{aligned}
  \langle h, e^{it \jtD} \chi_0(H) P_c\, g \rangle &= \frac{1}{2\pi }  \int_{\bbR^2} \int_{-\infty}^\infty e^{it\langle\xi\rangle} T(\xi) \chi_0(\xi^2)  \, \ud \xi \, f_+(x,0) f_-(y,0) g(y) \ol{h(x)} \, \one_{[x>y]} \, \, \ud y \, \ud x \\
  &\quad + \frac{1}{2\pi } \int_{\bbR^2} \int_{-\infty}^\infty e^{it\langle\xi\rangle}  T(\xi) \chi_0(\xi^2) \, \ud \xi \, f_-(x,0) f_+(y,0) g(y) \ol{h(x)} \, \one_{[x<y]} \, \, \ud y \, \ud x \\
  &\quad + \frac{1}{2\pi } \int_{\bbR^2} \int_{-\infty}^\infty e^{it\langle\xi\rangle} T(\xi) \chi_0(\xi^2) \, F_>(x,y;\xi) \, \ud \xi \, g(y)  \ol{h(x)} \, \one_{[x>y]} \, \, \ud y \, \ud x \\
  &\quad + \frac{1}{2\pi } \int_{\bbR^2} \int_{-\infty}^\infty e^{it\langle\xi\rangle}  T(\xi) \chi_0(\xi^2) \, F_<(x,y;\xi) \, \ud \xi \, g(y) \ol{h(x)} \, \one_{[x<y]} \, \, \ud y \, \ud x \\
  &\equiv A_{>}+A_{<}+B_{>}+B_{<}.
 \end{aligned}
\end{equation}
By \eqref{eq:TR}, 
\[
f_-(x,0) = \frac{T(0)}{1+R_-(0)} f_+(x,0) =: \kappa \vphi(x).
\]
Note that $\kappa\in\R$ and $1+R_-(0)\ne0$ due to  $|T(0)|^2+|R_-(0)|^2=1$ and $T(0)\ne0$. Then we have
\EQ{\nn
A_>+A_< &= \frac{\kappa}{2\pi} \langle h, \varphi\rangle \langle \varphi, g \rangle \int_{-\infty}^\infty  e^{it\langle\xi\rangle} T(\xi) \chi_0(\xi^2)\, \ud \xi.    
}
Setting $k=2$ in~\cite[Theorem~7.7.5]{Hor} yields
\begin{equation} \label{equ:stationary_phase_low_energy}
\int_{-\infty}^\infty  e^{it\langle\xi\rangle} T(\xi) \chi_0(\xi^2)\, \ud \xi = \frac{e^{i\pi/4}}{\sqrt{2\pi t}} e^{it} \,  T(0) +R(t),\qquad |R(t)|\le C_0\, t^{-\frac32},
\end{equation}
where $C_0$ depends on the  derivatives of $T(\xi) \chi(\xi^2)$ up to and including fourth order. Thus, 
\EQ{
\label{eq:A}
A_>+A_< = c_0  \langle h, ( \varphi\otimes\varphi) \, g \rangle \Big( \frac{e^{i\pi/4}}{\sqrt{t}} e^{it} + \calO_{L^\infty_t} (t^{-\frac32})\Big),\qquad c_0 = \frac{\kappa}{2\pi \sqrt{2\pi } }T(0),
}
where the constants in the $\calO_{L^\infty_t}(\cdot)$ term depend on $T(\xi)$. Applying~\cite[Theorem~7.7.5]{Hor} to $B_{>}$ with $k=2$ and using that $F_{>}(x,y;0)=0$ yields
\EQ{\label{eq:B2}
& \Big| B_{>} -  C_1   e^{it} t^{-\frac32} \int_{\R^2} \partial_\xi^2 \big( T(\xi) \chi_0(\xi^2) F_{>}(x,y;\xi)\big) \big|_{\xi=0}\, \one_{[x>y]} \, g(y) \ol{h(x)} \, \ud x \, \ud y \Big| \\
& \le C_2\, t^{-2} \int_{\R^2} \sup_{\xi\in\bbR} \sup_{\ell\le 4}\Big| \partial_\xi^\ell \big( T(\xi) \chi_0(\xi^2) F_{>}(x,y;\xi)\big)\Big |\, \one_{[x>y]}\, \bigl| g(y) \ol{h(x)} \bigr| \, \ud x \, \ud y
}
with some constants $C_1, C_2$.   By \eqref{eq:Fdef} and~\eqref{eq:TR*}, 
\[
\sup_{|\xi|\les 1} \sup_{\ell\le 4} \bigl| \partial_\xi^\ell F_{>}(x,y;\xi) \bigr| \le C\langle x\rangle^4 \langle y\rangle^4.
\]
The bound on $B_{<}$ is analogous. 
In summary, 
\EQ{\nn
\Big|\langle h, e^{it \jtD} \chi_0(H) P_c\, g \rangle  -  c_0  \langle h, (\varphi\otimes\varphi) \, g \rangle \frac{e^{i\pi/4}}{\sqrt{t}} e^{it} \Big| \le C t^{-\frac32} \| w h\|_{L^1_x} \| w g\|_{L^1_x},
}
which implies the desired local decay estimate~\eqref{eq:t32} given by
\EQ{\label{eq:klein lam}
\Big\|w^{-1} \big[ e^{it \jtD} \chi_0(H) P_c - c_0 \,  {e^{i\pi/4}} t^{-\frac12} e^{it}  (\varphi\otimes\varphi) \big] g\Big \|_{L^\infty_x} \le C t^{-\frac32} \| w g\|_{L^1_x}.
}

 Next, we turn to the proof of~\eqref{equ:local_decay_improved_subtr_off_with_derivative}. From the representation~\eqref{equ:representation_kg_prop_low_energy}, we obtain upon taking a derivative in $x$ that
 \begin{align*}
  \partial_x \bigl( e^{it\jtD} \chi(H) P_c g \bigr)(x) &=  \frac{1}{2\pi }  \int_{-\infty}^x \int_{-\infty}^\infty e^{it\langle\xi\rangle}  T(\xi) (\partial_x f_+)(x,\xi)f_-(y,\xi) \, \chi_0(\xi^2) \, d\xi \, g(y) \, \ud y \\
  &\quad +\frac{1}{2\pi }  \int_x^{\infty} \int_{-\infty}^\infty e^{it\langle\xi\rangle}  T(\xi) (\partial_x f_-)(x,\xi) f_+(y,\xi) \, \chi_0(\xi^2) \, \ud \xi \, g(y) \, \ud y.
 \end{align*}
 In order to isolate $0$ energy, we then write
 \begin{equation} \label{equ:representation_kg_prop_px_low_energy}
 \begin{aligned}
  \partial_x \bigl( e^{it\jtD} \chi(H) P_c g \bigr)(x) &= \frac{1}{2\pi} \int_{-\infty}^x \int_{-\infty}^{\infty} e^{it\jxi} T(\xi) \chi_0(\xi^2) \, \ud \xi \, (\partial_x f_+)(x,0) f_-(y,0) \, g(y) \, \ud y \\
  &\quad + \frac{1}{2\pi} \int_x^{\infty} \int_{-\infty}^{\infty} e^{it\jxi} T(\xi) \chi_0(\xi^2) \, \ud \xi \, (\partial_x f_-)(x,0) f_+(y,0) \, g(y) \, \ud y \\
  &\quad + \frac{1}{2\pi} \int_{-\infty}^x \int_{-\infty}^{\infty} e^{it\jxi} T(\xi) \chi_0(\xi^2) \, G_{>}(x,y;\xi) \, \ud \xi \, g(y) \, \ud y \\
  &\quad + \frac{1}{2\pi} \int_x^{\infty} \int_{-\infty}^{\infty} e^{it\jxi} T(\xi) \chi_0(\xi^2) \, G_{<}(x,y;\xi) \, \ud \xi \, g(y) \, \ud y \\
  &\equiv I + II + III + IV
 \end{aligned}
 \end{equation}
 with
 \begin{align*}
  G_>(x,y;\xi) &:= (\partial_x f_+)(x,\xi) f_-(y,\xi) - (\partial_x f_+)(x,0) f_-(y,0), \\
  G_<(x,y;\xi) &:= (\partial_x f_-)(x,\xi) f_+(y,\xi) - (\partial_x f_-)(x,0) f_+(y,0).
 \end{align*}
 Using that $f_-(x,0) = \kappa f_+(x,0)$ with $\kappa := \frac{T(0)}{1+R_-(0)}$ and therefore $(\partial_x f_-)(x,0) = \kappa (\partial_x f_+)(x,0)$, we obtain for the first two terms that
 \begin{equation*}
  \begin{aligned}
   I + II &= \frac{\kappa}{2\pi} \int_{-\infty}^x \int_{-\infty}^{\infty} e^{it\jxi} T(\xi) \chi_0(\xi^2) \, \ud \xi \, (\partial_x f_+)(x,0) f_+(y,0) \, g(y) \, \ud y \\
   &\quad + \frac{\kappa}{2\pi} \int_x^{\infty} \int_{-\infty}^{\infty} e^{it\jxi} T(\xi) \chi_0(\xi^2) \, \ud \xi \, (\partial_x f_+)(x,0) f_+(y,0) \, g(y) \, \ud y \\ 
   &= \frac{\kappa}{2\pi} \biggl( \int_{-\infty}^\infty e^{it\jxi} T(\xi) \chi_0(\xi^2) \, \ud \xi \biggr) \biggl( \int_{\bbR} f_+(y,0) g(y) \, \ud y \biggr) (\partial_x f_+)(x,0). 
  \end{aligned}
 \end{equation*}
 Using~\eqref{equ:stationary_phase_low_energy}, we find that
 \begin{equation*}
  \begin{aligned}
   I + II &= c_0 \frac{e^{i\frac{\pi}{4}} e^{it}}{t^{\frac12}} \langle \varphi, g \rangle (\partial_x \varphi)(x) + R(t) \langle \varphi, g \rangle (\partial_x \varphi)(x), \qquad |R(t)| \leq C_0 t^{-\frac32}, \qquad t \geq 1.
  \end{aligned}
 \end{equation*}
 Thus, taking the inner product of~\eqref{equ:representation_kg_prop_px_low_energy} with $h$, we have 
 \begin{align*}
  &\Bigl\langle h, \partial_x \Bigl( e^{it\jtD} \chi_0(H) P_c g - c_0 \frac{e^{i\frac{\pi}{4}} e^{it}}{t^\frac12} (\varphi \otimes \varphi) g \Bigr) \Bigr\rangle \\
  &= R(t) \langle h, \partial_x \varphi \rangle \langle \varphi, g \rangle  \\
  &\quad + \frac{1}{2\pi } \int_{\bbR^2} \int_{-\infty}^\infty e^{it\langle\xi\rangle} T(\xi) \chi_0(\xi^2) \, G_>(x,y;\xi) \, \ud \xi \, g(y)  \ol{h(x)} \, \one_{[x>y]} \, \, \ud y \, \ud x \\
  &\quad + \frac{1}{2\pi } \int_{\bbR^2} \int_{-\infty}^\infty e^{it\langle\xi\rangle}  T(\xi) \chi_0(\xi^2) \, G_<(x,y;\xi) \, \ud \xi \, g(y) \ol{h(x)} \, \one_{[x<y]} \, \, \ud y \, \ud x.
 \end{align*}
 Now using that $G_>(x,y;0) = G_<(x,y;0) = 0$, the fact that $\px \varphi \in L^\infty_x(\bbR)$, and that 
 \begin{equation*}
  \sup_{|\xi| \lesssim 1} \, \sup_{\ell \leq 4} \, \Bigl( \bigl| \partial_\xi^\ell G_>(x,y;\xi) \bigr| + \bigl| \partial_\xi^\ell G_<(x,y;\xi) \bigr| \Bigr) \leq C \jap{x}^4 \jap{y}^4,
 \end{equation*}
 we can conclude the proof of~\eqref{equ:local_decay_improved_subtr_off_with_derivative} by arguing as in the preceding proof of~\eqref{eq:t32}.

 Finally, regarding the proofs of the generalized versions~\eqref{eq:t32 om} and~\eqref{equ:local_decay_improved_subtr_off_with_derivative_om} involving the operator $\omega(H)$, note that the conditions on $\omega$ are such that the preceding arguments still apply.
\end{proof}

The weights in~Lemma~\ref{lem:low proj} are most likely not sharp. 
We remark that a bound as in Lemma~\ref{lem:low proj}  cannot hold for large energies~$\xi$. 
In fact, it is an immediate consequence of stationary phase that derivatives of $g$ are needed to bound the pointwise size of the evolution in~\eqref{eq:t32 om}. 
We will pursue this in more detail below, but first establish local $L^2$-decay for energies separated from~$0$. 
For the following lemma, the distinction between $V$ generic and non-generic is irrelevant.  
Moreover, we use the notation $\jap{\tilde{D}} := \sqrt{1+H}$ on the positive spectrum of $H$. 

\begin{lemma}\label{lem:high f}
 Let $H = -\partial_x^2 + V(x)$ with real-valued $V\in L^\infty(\R)\cap C^{3}(\R)$, and assume that $\jap{x}^6 V^{(\ell)}(x) \in L^1(\R)$ for all $0\le\ell\le 3$. Fix $\xi_0 > 0$. Let $\chi_0$ be a smooth bump function such that $\chi_0(\xi) = 1$ for $|\xi| \leq 1$ and $\chi_0(\xi) = 0$ for $|\xi| \geq 2$.
 Then
 \begin{equation}
  \bigl\| \jx^{-2} e^{it\jtD} \bigl( 1- \chi_0(H/\xi_0^2) \bigr) P_c g \bigr\|_{L^2_x} \le C \jt^{-2} \| \jx^{2} g \|_{L^2_x}
 \end{equation}
 with some constant $C > 0$ depending on $\xi_0$ and $V$.  The same estimate holds with $e^{it\jtD}$ replaced by $e^{it\jtD} \omega(H)$ where $|\partial_\xi^\ell \omega(\xi)|\le C(\xi_0)$ for all $|\xi|\ge\xi_0$ and $0\le\ell\le 5$. 
\end{lemma}
\begin{proof}
By \eqref{eq:Ekern} the distorted Fourier basis takes the form
 \EQ{\label{eq:e bas}
  e(x,\xi) := \frac{1}{\sqrt{2\pi}} \left\{ \begin{aligned}
                                            &T(\xi) f_+(x,\xi), \quad &\xi \geq 0, \\
                                            &T(-\xi) f_-(x, -\xi), \quad &\xi < 0.
                                           \end{aligned} \right.
}
 Thus,  the distorted Fourier transform of $g=P_c g$ and its inverse are given by 
 \begin{equation*}
  \widetilde{g}(\xi) = \int_{\bbR} \overline{e(x,\xi)} g(x) \, \ud x, \quad g(x) = \int_{\R} e(x,\xi)  \widetilde{g}(\xi)\, d\xi
 \end{equation*}
 and Plancherel's theorem reads $\|g\|_{L^2_x} = \|\tilde g\|_{L^2_\xi}$. 
 The Klein-Gordon evolution therefore takes the form 
 \begin{align*}
  \bigl( e^{it\jtD} \bigl( 1- \chi_0(H/\xi_0^2) \bigr) P_c g \bigr)(x) &= \int_{\bbR} e(x, \xi) e^{it\jxi} \bigl( 1 - \chi_0(\xi^2/\xi_0^2) \bigr) \widetilde{g}(\xi) \, \ud \xi. 
 \end{align*} 
 In view of the cutoff $1-\chi_0(\xi^2/\xi_0^2)$ we can treat the regions $\xi \geq \xi_0$ and $\xi \leq -\xi_0$ separately. We also introduce smooth cutoff functions $\theta_\pm(x)$ such that  $\theta_+(x) + \theta_-(x) = 1$ for all $x \in \bbR$ and such that $\theta_{\pm}(x) = 0$ for $\pm x < -1$.
 By symmetry, it suffices to consider the case $\xi \geq \xi_0$ which  contributes the following expression to the time evolution: 
 \begin{align*}
  &\int_0^\infty e(x, \xi) e^{it\jxi} \bigl( 1 - \chi_0(\xi^2/\xi_0^2) \bigr) \widetilde{g}(\xi) \, \ud \xi \\
  &= \theta_+(x) \int_0^\infty T(\xi) f_+(x,\xi) e^{it\jxi} \bigl( 1 - \chi_0(\xi^2/\xi_0^2) \bigr) \widetilde{g}(\xi) \, \ud \xi \\
  &\quad + \theta_-(x) \int_0^\infty T(\xi) f_+(x,\xi) e^{it\jxi} \bigl( 1 - \chi_0(\xi^2/\xi_0^2) \bigr) \widetilde{g}(\xi) \, \ud \xi \\
  &=: I_+ + I_-.
 \end{align*}
 We further rewrite the term $I_+$ as
 \begin{align*}
  I_+ = \theta_+(x) \int_0^\infty e^{ix\xi} e^{it\jxi} T(\xi) m_+(x,\xi) \bigl( 1 - \chi_0(\xi^2/\xi_0^2) \bigr) \widetilde{g}(\xi) \, \ud \xi 
 \end{align*}
 and using the identity 
 \begin{align*}
  T(\xi) f_+(x,\xi) = f_-(x,-\xi) + R_-(\xi) f_-(x,\xi) = e^{ix\xi} m_-(x,-\xi) + R_-(\xi) e^{-ix\xi} m_-(x,\xi),
 \end{align*}
 we can express the term $I_-$ as
 \begin{equation} \label{equ:Iminus_rewritten}
  \begin{aligned}
   I_- &= \theta_-(x) \int_0^\infty e^{ix\xi} e^{it\jxi} m_-(x,-\xi) \bigl( 1 - \chi_0(\xi^2/\xi_0^2) \bigr) \widetilde{g}(\xi) \, \ud \xi \\
   &\quad + \theta_-(x) \int_0^\infty e^{-ix\xi} e^{it\jxi} m_-(x,\xi) \bigl( 1 - \chi_0(\xi^2/\xi_0^2) \bigr) \widetilde{g}(\xi) \, \ud \xi.
  \end{aligned}
 \end{equation}
 Let us further consider the term $I_+$. Using that $\partial_\xi (e^{it\jxi}) = it \frac{\xi}{\jxi} e^{it\jxi}$ and integrating by parts once, we obtain
 \begin{align*}
  I_+ &= - \frac{1}{it} \theta_+(x) \int_0^\infty e^{it\jxi} \partial_\xi \biggl( \frac{\jap{\xi}}{\xi} e^{ix\xi} T(\xi) m_+(x,\xi) \bigl( 1 - \chi_0(\xi^2/\xi_0^2) \bigr) \widetilde{g}(\xi) \biggr) \, \ud \xi \\
  &= - \frac{ix}{it} \theta_+(x) \int_0^\infty e^{ix\xi} e^{it\jxi} \frac{\jap{\xi}}{\xi} T(\xi) m_+(x,\xi) \bigl( 1 - \chi_0(\xi^2/\xi_0^2) \bigr) \widetilde{g}(\xi) \, \ud \xi \\
  &\quad - \frac{1}{it} \theta_+(x) \int_0^\infty e^{ix\xi} e^{it\jxi} \partial_\xi \biggl( \frac{\jap{\xi}}{\xi} T(\xi) m_+(x,\xi) \bigl( 1 - \chi_0(\xi^2/\xi_0^2) \bigr) \biggr) \widetilde{g}(\xi) \, \ud \xi \\
  &\quad - \frac{1}{it} \theta_+(x) \int_0^\infty e^{ix\xi} e^{it\jxi} \frac{\jap{\xi}}{\xi} T(\xi) m_+(x,\xi) \bigl( 1 - \chi_0(\xi^2/\xi_0^2) \bigr) \partial_\xi \widetilde{g}(\xi) \, \ud \xi \\
  &=: I_+^{(a)} + I_+^{(b)} + I_+^{(c)}.
 \end{align*}
We view 
 \begin{align*}
 t \jx^{-1} I_+^{(a)} = - \frac{x}{\jx} \int_{0}^{\infty} e^{ix\xi} \biggl( \frac{\jap{\xi}}{\xi} T(\xi) \theta_+(x) m_+(x,\xi) \widetilde{\chi}_{\{\xi \geq \xi_0/2\}}(\xi) \bigl( 1 - \chi_0(\xi^2/\xi_0^2) \bigr) \biggr) e^{it\jxi} \widetilde{g}(\xi) \, \ud \xi 
 \end{align*}
 as a pseudo-differential operator on the line (after introducing another smooth cutoff $\widetilde{\chi}_{\{\xi \geq \xi_0/2\}}(\xi)$) with symbol 
 \begin{align*}
  a(x,\xi) :=  - \frac{x}{\jx} \frac{\jap{\xi}}{\xi} T(\xi) \theta_+(x) m_+(x,\xi) \widetilde{\chi}_{\{\xi \geq \xi_0/2\}}(\xi) \bigl( 1 - \chi_0(\xi^2/\xi_0^2) \bigr)
 \end{align*}
 By the Calderon-Vaillancourt theorem, see for example~\cite[Proposition 9.4]{MS}, we  infer that
 \begin{align*}
  \bigl\| \jx^{-1} I_+^{(a)} \bigr\|_{L^2_x} \lesssim_{\xi_0} \frac{1}{t} \bigl\| e^{it\jxi} \widetilde{g}(\xi) \bigr\|_{L^2_\xi} \simeq \frac{1}{t} \| \widetilde{g}(\xi) \|_{L^2_\xi} \lesssim \frac{1}{t} \|g\|_{L^2_x}.
 \end{align*}
 Lemmas~\ref{lem:fm} and~\ref{lem:T} imply that the hypotheses of that theorem hold, i.e.,  that the symbol $a$ satisfies 
 \EQ{\label{eq:CV hyp}
 |\partial^j_x a(x,\xi)| + |\partial^k_\xi a(x,\xi)|\le C\qquad \forall \; j,k=0,1,2,3
 }
 Note that it makes no difference in that lemma if we assume $x\ge0$ or $x\ge-10$, say, for the bounds on $m_+$. 
 The terms $I_+^{(b)}$ and $I_+^{(c)}$ can be handled analogously, together with the terms in the identity~\eqref{equ:Iminus_rewritten} for $I_-$ since we now are dealing with $m_-(x,\xi)$ on $x\les 1$. The $L^2$ estimate of $I_+^{(c)}$ requires the bound 
 \EQ{\label{eq:pg bd}
 \| \partial_\xi  \widetilde{g}(\xi) \|_{L^2(|\xi|\ge\xi_0)} \les \| \jap{x} g \|_{L^2_x}
 }
 which again follows from the Calderon-Vaillancourt theorem. Indeed, for $\xi\gtrsim \xi_0$,  by~\eqref{eq:e bas}, 
 \EQ{ \label{eq:again CV}
 \partial_\xi \widetilde{g}(\xi) &= \int_{\bbR}  \partial_\xi  \overline{e(x,\xi)} g(x) \, \ud x  \\
 & = \int_{\bbR}  \partial_\xi  \overline{e(x,\xi)} \theta_+(x) g(x) \, \ud x + \int_{\bbR}  \partial_\xi  \overline{e(x,\xi)} \theta_-(x) g(x) \, \ud x \\
 & = \frac{1}{\sqrt{2\pi}}  \int_{\bbR}  \partial_\xi ( T(-\xi) f_+(x,-\xi)) \theta_+(x) g(x) \, \ud x  \\
 &\quad + \frac{1}{\sqrt{2\pi}}  \int_{\bbR}  \partial_\xi (  f_-(x,\xi) + R_-(-\xi) f_-(x,-\xi)) \theta_-(x) g(x) \, \ud x  \\
 & = S_+(g)(\xi) + S_-(g)(\xi)
 }
 where the last line follows from \eqref{eq:TR*}. On the one hand, 
 \EQ{\nn
 \sqrt{2\pi}\, S_+(g)(\xi)&= \int_{\bbR}  \partial_\xi ( T(-\xi) f_+(x,-\xi)) \theta_+(x) g(x) \, \ud x \\
 &= \int_{\bbR}  \partial_\xi ( e^{-ix\xi}\; T(-\xi) m_+(x,-\xi)) \theta_+(x) g(x) \, \ud x\\
 & =   \int_{\bbR}  e^{-ix\xi} \big[-ix  T(-\xi) m_+(x,-\xi)+ \partial_\xi {(T(-\xi) m_+(x,-\xi)})\big] \theta_+(x) g(x) \, \ud x 
 }
 By Lemmas~\ref{lem:fm} and~\ref{lem:T}, the symbol 
 \[
  b_+(x,\xi):= \big[-ix  T(-\xi) m_+(x,-\xi)+ \partial_\xi {(T(-\xi) m_+(x,-\xi)})\big] \theta_+(x)  \jap{x}^{-1}
 \]
 satisfies the hypotheses of the Calderon-Vaillancourt theorem, cf~\eqref{eq:CV hyp}, and we obtain the desired $L^2$ bound from the term~$S_+(g)(\xi)$. 
 On the other hand, 
 \EQ{\nn 
 \sqrt{2\pi}  \, S_-(g)(\xi) & =   \int_{\bbR} \partial_\xi (  f_-(x,\xi) + R_-(-\xi) f_-(x,-\xi)) \theta_-(x) g(x)  \, \ud x\\
 &=  \int_{\bbR} \partial_\xi \big[  e^{-ix\xi} m_-(x,\xi) + R_-(-\xi) e^{ix\xi} m_-(x,-\xi) \big] \theta_-(x) g(x)  \, \ud x  \\
 & =  \int_{\bbR}  e^{-ix\xi} \big[ -ix m_-(x,\xi) + \partial_\xi m_-(x,\xi)\big] \theta_-(x) g(x)  \, \ud x \\
 & \quad + \int_{\bbR}  e^{ix\xi} \big[ ix R_-(-\xi)   m_-(x,-\xi) + \partial_\xi  ( R_-(-\xi) m_-(x,-\xi)) \big] \theta_-(x) g(x)  \, \ud x \\
 & = \int_{\bbR}  e^{-ix\xi} b_{--} (x,\xi) \jap{x} g(x)  \, \ud x + \int_{\bbR}  e^{ix\xi} b_{-+}(x,\xi) \jap{x} g(x)  \, \ud x
 }
The symbols $b_{--} (x,\xi)$ and $b_{-+}(x,\xi)$ also satisfy the hypotheses of the Calderon-Vaillancourt theorem as before. 
In conclusion, we may again apply Calderon-Vaillancourt theorem which proves \eqref{eq:pg bd}, at least for positive $\xi$. However, the contributions by negative~$\xi$ is analogous. 
In summary, we have only obtained $t^{-1}$-decay at the expense of one power of~$x$. 
Integrating by parts one more time gives the desired $t^{-2}$ estimate.  The same proof which implied~\eqref{eq:pg bd} also yields
 \[
\|\partial_\xi^2 \widetilde{g}(\xi)\|_{L^2(|\xi|\ge\xi_0)} \lesssim \|\jx^2 g(x)\|_{L^2_x}.
\]
The assumptions on $V(x)$ are compatible with the requirements in this proof: integrating by parts twice in~$\xi$ leads to $\partial_\xi^j m_{\pm}(x,\xi)$ with $0\le j\le 2$. For the Calderon-Vaillancourt theorem one then needs three derivatives in $x$ and $\xi$, but separately. For the final statement involving the operator $\omega(H)$, note that the conditions on $\omega$ are such that the Calderon-Vaillancourt theorem still applies. 
\end{proof}

We can now state a complete list of local $L^2$ decay estimates on the linear evolution $e^{it\jap{\tilde D}} P_c$ that will be needed in the nonlinear analysis. 

\begin{corollary} \label{cor:main}
Let $H = -\partial_x^2 + V(x)$ with real-valued $V\in L^\infty(\R)\cap C^{3}(\R)$, and assume that $\jap{x}^6 V^{(\ell)}(x) \in L^1(\R)$ for all $0\le\ell\le 3$.
Let $\psi(\xi)$ be a smooth function with $\psi(\xi) = 0$ for $|\xi| \leq \xi_0$ for some $\xi_0 > 0$ and such that $|\partial_\xi^\ell \psi(\xi)| \leq C(\psi)$ for $0 \leq \ell \leq 5$. Then for $\sigma>\frac92$ and all $t\in\R$, 
\begin{align}
 \bigl\| \jx^{-\sigma} e^{it\jtD} P_c g \bigr\|_{L^2_x} &\lesssim  {\jt^{-\frac12}} \| \jx^{\sigma} g \|_{L^2_x}, \label{equ:local_decay_est} \\
 \bigl\| \jx^{-\sigma} \sqrt{H} \jtD^{-1} e^{it\jtD} P_c g \bigr\|_{L^2_x} &\lesssim {\jt^{-\frac32}} \| \jx^{\sigma} g \|_{L^2_x}, \label{equ:local_decay_sqrtH_est} \\
 \Bigl\| \jx^{-\sigma} \frac{\jtD-1}{\jtD} e^{it\jtD} P_c g \Bigr\|_{L^2_x} &\lesssim {\jt^{-\frac32}} \| \jx^{\sigma} g \|_{L^2_x}, \label{equ:local_decay_jtDminus1_est}  \\
 \bigl\| \jx^{-\sigma} \psi(\wtilD) e^{it\jtD} P_c g \bigr\|_{L^2_x} &\lesssim {\jt^{-\frac32}} \| \jx^{\sigma} g \|_{L^2_x}. \label{equ:local_decay_away_zero_freq_est}
\end{align}
as well as 
\begin{align}
 \Bigl\| \jx^{-\sigma} \Bigl( e^{it\jtD} P_c g - c_0 \frac{e^{i\frac{\pi}{4}} e^{it}}{t^\hf} (\varphi \otimes \varphi) g \Bigr) \Bigr\|_{L^2_x} &\lesssim  {\jt^{-\thf}} \| \jx^{\sigma} g \|_{L^2_x}, \quad t \geq 1,  \label{equ:local_decay_improved_subtr_off} \\
 \Bigl\| \jx^{-\sigma} \px \Bigl( e^{it\jtD} \chi_0(H) P_c g - c_0 \frac{e^{i\frac{\pi}{4}} e^{it}}{t^\hf} (\varphi \otimes \varphi) g \Bigr) \Bigr\|_{L^2_x} &\lesssim  {\jt^{-\thf}} \| \jx^{\sigma} g \|_{L^2_x}, \quad t \geq 1, \label{equ:local_decay_improved_subtr_off_with_derivative_in_cor}
\end{align}
where $c_0$ defined in~\eqref{equ:def_c0} is an absolute constant that only depends on the the scattering matrix $S(0)$ of the potential $V(x)$ at zero energy.
Finally, we have the variants 
\begin{align}
 \bigl\| \jx^{-\sigma} \jtD^{-1} e^{it\jtD} P_c g \bigr\|_{L^2_x} &\lesssim {\jt^{-\frac12}} \| \jx^{\sigma} g \|_{L^2_x}, \label{equ:local_decay_tDinv_est} \\
 \bigl\| \jx^{-\sigma} \psi(\wtilD) \jtD^{-1} e^{it\jtD} P_c g \bigr\|_{L^2_x} &\lesssim {\jt^{-\frac32}} \| \jx^{\sigma} g \|_{L^2_x}. \label{equ:local_decay_away_zero_freq_tDinv_est} \\
 \Bigl\| \jx^{-\sigma} \Bigl( \jtD^{-1} e^{it\jtD} P_c g - c_0 \frac{e^{i\frac{\pi}{4}} e^{it}}{t^\hf} (\varphi \otimes \varphi) g \Bigr) \Bigr\|_{L^2_x} &\lesssim {\jt^{-\thf}} \| \jx^{\sigma} g \|_{L^2_x}, \quad t \geq 1, \label{equ:local_decay_improved_tDinv_subtr_off} \\
 \Bigl\| \jx^{-\sigma} \px \Bigl( \jtD^{-1} e^{it\jtD} \chi_0(H) P_c g - c_0 \frac{e^{i\frac{\pi}{4}} e^{it}}{t^\hf} (\varphi \otimes \varphi) g \Bigr) \Bigr\|_{L^2_x} &\lesssim {\jt^{-\thf}} \| \jx^{\sigma} g \|_{L^2_x}, \quad t \geq 1. \label{equ:local_decay_improved_tDinv_subtr_off_with_derivative}
\end{align}
\end{corollary} 
\begin{proof}
Lemmas~\ref{lem:high f} and \ref{lem:low proj} imply that \eqref{equ:local_decay_improved_subtr_off} holds with $\sigma>\frac92$. Lemma~\ref{lem:low proj} also implies that~\eqref{equ:local_decay_improved_subtr_off_with_derivative_in_cor} holds with $\sigma > \frac92$. Moreover, \eqref{equ:local_decay_away_zero_freq_est} with $\sigma=2$ is a direct consequence of Lemma~\ref{lem:high f}. 
For $0<t\le 1$, \eqref{equ:local_decay_est} follows from $L^2$ boundedness of the evolution, while for $t\ge1$ it follows from \eqref{equ:local_decay_improved_subtr_off}. For the latter, we use that $\| \jap{x}^{-\frac12-} (\varphi \otimes \varphi) f\|_2\lesssim \| \jap{x}^{\frac12+}f\|_2$. 

Applying the more general version of Lemma~\ref{lem:low proj} with $\omega(\xi)=\xi\jap{\xi}^{-1}$, respectively, $\omega(\xi)=\frac{\jap{\xi}-1}{\jap{\xi}}=\frac{\xi^2}{\jap{\xi}(1+\jap{\xi})}$, which both vanish at $\xi=0$, eliminates the projection $\varphi \otimes \varphi$ from~\eqref{equ:local_decay_improved_subtr_off}. By the same argument as before, invoking the more general version of Lemma~\ref{lem:high f}, we therefore obtain \eqref{equ:local_decay_sqrtH_est} and~\eqref{equ:local_decay_jtDminus1_est}. In the same way one derives the final estimates \eqref{equ:local_decay_tDinv_est}, \eqref{equ:local_decay_away_zero_freq_tDinv_est}, \eqref{equ:local_decay_improved_tDinv_subtr_off}, and~\eqref{equ:local_decay_improved_tDinv_subtr_off_with_derivative}.
\end{proof}

We expect the local decay estimates for the Klein-Gordon evolution $e^{it\jtD} P_c$ established in this paper to be of independent interest. 
In particular, the refined local decay estimates~\eqref{equ:local_decay_improved_subtr_off} and~\eqref{equ:local_decay_improved_subtr_off_with_derivative_in_cor} seem to not have appeared in the literature yet. 
Their proofs are inspired by the method of proof of Proposition~9 in the joint work~\cite{KS07} of the third author with Krieger, where pointwise decay estimates are established for a perturbed 3D wave evolution upon subtracting off a projection to a resonance function.
We refer to Komech-Kopylova~\cite{KomKop10}, Kopylova~\cite{Kop11}, and Egorova-Kopylova-Marchenko-Teschl~\cite{EgoKopMarTes16} for prior results on local decay estimates for one-dimensional Klein-Gordon equations with potential terms.

Next, we establish a pointwise bound on the evolution for all energies. 

\begin{lemma}
\label{lem:pw decay}
Let $H = -\partial_x^2 + V(x)$ with real-valued $V\in L^\infty(\R)\cap C^{3}(\R)$, and assume that $\jap{x}^6 V^{(\ell)}(x) \in L^1(\R)$ for all $0\le\ell\le 3$. Then 
 \begin{equation} \label{equ:dispersive_decay_propagator}
  \| e^{i t \jtD} P_c g \|_{L^\infty_x} \le \frac{C(\mu,V)}{t^{\hf}} \| \jtD^{\thf+\mu} g  \|_{L^1_x}
 \end{equation}
 for all $t>0$ and $\mu>0$. 
\end{lemma}
\begin{proof}
Throughout, $g\in\calS(\R)$. 
Let $\chi$ be a bump function supported on $\R\setminus\{0\}$ and fix any $\lambda\ge 1$. Using the distorted Fourier basis~\eqref{eq:e bas},  consider the evolution 
\EQ{\label{eq:Phitpm}
\big( e^{i t \jtD}  \chi(\tD/\lambda) P_c g\big)(x) &= \int_{\R} e^{it\jap{\xi}} e(x,\xi) \chi(\xi/\lambda) \tilde g(\xi)\, \ud\xi \\
& =  \frac{1}{\sqrt{2\pi}} \int_{0}^\infty  e^{it\jap{\xi}} T(\xi)f_+(x,\xi)  \chi(\xi/\lambda) \tilde g(\xi)\, \ud\xi \\
&\quad + \frac{1}{\sqrt{2\pi}} \int_{-\infty}^0 e^{it\jap{\xi}} T(-\xi) f_-(x,-\xi) \chi(\xi/\lambda) \tilde g(\xi)\, \ud\xi  \\
& =: (\Phi^+_\lambda(t) g)(x) + (\Phi^-_\lambda(t) g)(x)
}
Without loss of generality we  assume $x>0$. Then using $\theta_{\pm}$ from the proof of Lemma~\ref{lem:high f}, 
\EQ{\nn 
 (\Phi^+_\lambda(t) g)(x) &= \frac{1}{\sqrt{2\pi}} \int_{0}^\infty e^{it\jap{\xi}} T(\xi) e^{ix\xi} m_+(x,\xi)  \chi(\xi/\lambda) \tilde g(\xi)\, \ud\xi  \\
 &= \frac{1}{2\pi} \int_\R \int_{0}^\infty e^{it\jap{\xi}} |T(\xi)|^2 e^{i(x-y)\xi} m_+(x,\xi) m_+(y,-\xi) \theta_+(y) \chi(\xi/\lambda)  \, \ud\xi \, g(y)\, \ud y\\
 & \quad + \frac{1}{2\pi} \int_\R \int_{0}^\infty e^{it\jap{\xi}} T(\xi) e^{ix\xi} m_+(x,\xi)   \big(  e^{-iy\xi} m_-(y,\xi) \\
 &\qquad\quad + R_-(-\xi) e^{iy\xi} m_-(y,-\xi)\big) \theta_-(y) \chi(\xi/\lambda)  \, \ud\xi \, g(y)\, \ud y\\
 &\quad =:  \int_\R K_\lambda^{+-} (t,x,y)  g(y)\, \ud y  + \int_\R K_\lambda^{++} (t,x,y)  g(y)\, \ud y
}
with 
\EQ{\nn 
K_\lambda^{+-}(t,x,y) & =  \int_{0}^\infty  e^{it[\jap{\xi}+\xi(x-y)/t]}  \omega_\lambda(x,y,\xi) \, \ud\xi  \\
\omega_\lambda(x,y,\xi) &:= \frac{1}{2\pi}T( \xi)  m_+(x,\xi) \big[ T(-\xi)  m_+(y,-\xi) \theta_+(y)+m_-(y,\xi) \theta_-(y)  \big]\,  \chi(\xi/\lambda) \\
K_\lambda^{++}(t,x,y) & = \int_{0}^\infty  e^{it[\jap{\xi}+\xi(x+y)/t]}  \tilde \omega_\lambda(x,y,\xi)  \, \ud\xi \\
\tilde \omega_\lambda(x,y,\xi) &:= \frac{1}{2\pi} T( \xi) R_-(-\xi)  m_+(x,\xi)   m_-(y,-\xi) \theta_-(y)    \chi(\xi/\lambda)
}
There exists a constant $C_0>1$ so that $C_0^{-1}\lambda<|\xi|<C_0\lambda$ on the support of $\chi(\xi/\lambda)$. 
By Lemmas~\ref{lem:m symb} and~\ref{lem:T}, 
$
|\partial_\xi^\ell \omega_\lambda(x,y,\xi)| \le C \lambda^{-\ell}$ for $\ell=0,1,2$
uniformly in $x\ge0$, and $y,\xi\in\R$.  By the same lemmas the analogous bounds holds for $\tilde \omega_\lambda(x,y,\xi)$. 
We rescale to obtain 
\EQ{\nn 
K_\lambda^{+-}(t,x,y) & = {\lambda}\int_{0}^\infty   e^{i\lambda t [\lambda^{-1} \jap{\lambda \xi} + \xi(x-y)/t]}  \omega_\lambda(x,y,\lambda\xi) \, \ud\xi  
}
One has the bound $|K_\lambda^{+-}(t,x,y)|\le C\lambda$ uniformly in $x\ge0$, $y\in\R$, $t>0$, $\lambda\ge1$. 
If $t\ge\lambda$, then we claim the stronger bound 
\EQ{\label{eq:claim K}
|K_\lambda^{+-}(t,x,y)|\le C\lambda^{\frac32} t^{-\frac12}.
}
We write  
\EQ{\label{eq:Ks}
K_\lambda^{+-}(t,x,y) & = {\lambda}\int_{0}^\infty   e^{is\varphi^+_\lambda(\xi; t,x-y)}  \omega_\lambda(x,y,\lambda\xi) \, \ud\xi  
}
with $s:=\lambda^{-1}  t $ and 
 phase $\varphi^+_\lambda(\xi; t,x-y):=\lambda^2[\lambda^{-1} \jap{\lambda \xi} + \xi(x-y)/t]$.
Then 
\EQ{\label{eq:derphi} 
\partial_\xi \varphi^+_\lambda(\xi; t,x-y) &=  \lambda^2\big[\frac{\lambda\xi}{\jap{\lambda\xi}} + (x-y)/t\big] \\
\partial_\xi^2 \varphi^+_\lambda(\xi; t,x-y) &=  \frac{\lambda^3}{\jap{\lambda\xi}^3} \simeq 1\\
\partial_\xi^3 \varphi^+_\lambda(\xi; t,x-y) &=  -3\frac{\lambda^5\xi}{\jap{\lambda\xi}^5} \simeq 1
}
on the support $I_0:=[\xi_1,\xi_2]\subset (0,\infty)$ of $\chi$ (recall $\lambda\ge1$).  We distinguish the following two cases, for fixed $x,y,t,\lambda$ as above: 
\begin{itemize}
\item[(a)] $\min | \partial_\xi \varphi^+_\lambda(\xi; t,x-y)|\gtrsim s^{-\frac12}$ on $I_0$
\item[(b)] $\min | \partial_\xi \varphi^+_\lambda(\xi; t,x-y)|\ll s^{-\frac12}$ on $I_0$
\end{itemize}
In the first Case (a), we deduce from the second derivative in  \eqref{eq:derphi} that 
\[
| \partial_\xi \varphi^+_\lambda(\xi; t,x-y)|\gtrsim s^{-\frac12} + \min \bigl\{ |\xi-\xi_1|, |\xi-\xi_2| \bigr\} \quad \forall\; \xi \in I_0
\]
 Integrating by parts once in \eqref{eq:Ks} yields 
 \EQ{
 \nn
 |K_\lambda^{+-}(t,x,y)|&\le C\lambda s^{-1} \int_{I_0} \Big( \frac{|\partial_\xi^2 \varphi^+_\lambda(\xi; t,x-y)|}{(\partial_\xi \varphi^+_\lambda(\xi; t,x-y))^2} +  \frac{1}{|\partial_\xi \varphi^+_\lambda(\xi; t,x-y)|}\Big)\, d\xi \\
 &\le C\lambda s^{-\frac12} 
 }
as claimed by \eqref{eq:claim K}. On the other hand, in Case (b) suppose the minimum of $\min | \partial_\xi \varphi^+_\lambda(\xi; t,x-y)|$ is attained at $\xi_*\in I_0$. 
Then we infer from  the second derivative that 
\[
| \partial_\xi \varphi^+_\lambda(\xi; t,x-y)|\gtrsim |\xi-\xi_*| \text{\ \ on\ \ } \xi\in I_0, \; |\xi-\xi_*|\ge s^{-\frac12}
\]
Let $\psi$ be a smooth bump function which equals $1$ on $[-1,1]$. Then with $\calL:= \frac{1}{i\partial_\xi \varphi^+_\lambda}\partial_\xi$, 
We write  
\EQ{\nn
|K_\lambda^{+-}(t,x,y)| & \le  {\lambda}\Big| \int_{0}^\infty   e^{is\varphi^+_\lambda(\xi; t,x-y)}  \omega_\lambda(x,y,\lambda\xi) \psi((\xi-\xi_*)s^{\frac12})\, \ud\xi \Big|  \\
&\quad + {\lambda}s^{-2} \Big| \int_{0}^\infty   e^{is\varphi^+_\lambda(\xi; t,x-y)} (\calL^*)^2\big[ \omega_\lambda(x,y,\lambda\xi) \bigl( 1 - \psi((\xi-\xi_*)s^{\frac12}) \bigr) \big]\, \ud\xi \Big| \\
&\lesssim \lambda s^{-\frac12} + \lambda s^{-2}  \int_{I_0} \one_{[|\xi-\xi_*|\ge s^{-\frac12}]} \Big( |\xi-\xi_*|^{-4}+ |\xi-\xi_*|^{-2} s\Big)\, \ud\xi\\
&\lesssim  \lambda s^{-\frac12}
}
which establishes the claim \eqref{eq:claim K}. The analysis of $K_\lambda^{++}(t,x,y)$ is completely analogous, as is the evolution of the negative frequencies given by $\Phi^-_\lambda(t)$, see~\eqref{eq:Phitpm}. In summary, for all $\lambda\ge1$, $t>0$, 
\EQ{\label{eq:awayf0}
\|e^{i t \jtD}  \chi(\tD/\lambda) P_c g\|_{L^\infty_x} \le C(V,\chi) t^{-\frac12} \lambda^{\frac32} \|g\|_{L^1_x}
}
For small energies we proceed as in the proof of Lemma~\ref{lem:low proj}, and write with a cutoff $\chi_0$ around zero energies
\EQ{\label{eq:small xi} 
& \big( e^{it \langle \tilde D\rangle} \chi_0(H) P_c\, g\big)(x)   \\
& = \frac{1}{2\pi }  \int_{\R} \int_{-\infty}^\infty e^{it\langle\xi\rangle}  T(\xi) f_+(x,\xi)f_-(y,\xi) \, \chi_0(\xi^2)\, \ud\xi   \,  g(y) \one_{[x>y]} \,  \ud y \\
& \quad +\frac{1}{2\pi }  \int_{\R} \int_{-\infty}^\infty e^{it\langle\xi\rangle}  T(\xi) f_-(x,\xi)f_+(y,\xi)  \, \chi_0(\xi^2)\, d\xi  \, g(y) \one_{[x<y]}  \, \ud y \\
& =: \Psi_>(t) P_c g(x) + \Psi_<(t) P_c g(x)
}
We again restrict to $x>0$ without loss of generality, and write $\Psi_>(t) P_c g(x)$ in the form
\EQ{\nn
&\Psi_>(t) P_c g(x) = \int_\R K_>(t,x,y) g(y)\, \ud y \\
&K_>(t,x,y)  = \frac{1}{2\pi }   \one_{[x>0>y]} \int_{-\infty}^\infty e^{i[t\langle\xi\rangle+(x-y)\xi]}  T(\xi) m_+(x,\xi)m_-(y,\xi) \, \chi_0(\xi^2)\, \ud\xi  \\
&\quad + \frac{1}{2\pi }   \one_{[x>y>0]} \int_{-\infty}^\infty e^{i[t\langle\xi\rangle+x\xi]}  m_+(x,\xi)[e^{-iy\xi} m_+(y,-\xi)+e^{iy\xi} R_+(\xi)m_+(y,\xi)] \, \chi_0(\xi^2)\, \ud\xi  
}
By Lemmas~\ref{lem:m symb} and~\ref{lem:T} the non-oscillatory integrands possess two $\xi$ derivatives uniformly bounded on their supports. The preceding stationary phase analysis therefore applies to $K_>(t,x,y)$ by setting $\lambda=1$, in particular $s=t$ in this case, cf.~\eqref{eq:Ks}. As a result one obtains 
\EQ{\label{eq:dispt}
\| \Psi_>(t) P_c g\|_{L^\infty_x} \le Ct^{-\frac12} \|g\|_{L^1_x} \quad \forall t>0
}
and the same also holds for $\Psi_<(t) g$ by a similar argument. Performing a dyadic decomposition of energies, and adding up all contributions from~\eqref{eq:awayf0} and~\eqref{eq:dispt} yields
\EQ{\label{eq:LPdec}
\| e^{i t \jtD} P_c g \|_{L^\infty_x} &\le C(V) t^{-\hf} \Big( \|\chi_0(H)P_c g\|_{L^1_x} + \sum_{j=0}^\infty 2^{3j/2} \| \chi(\tD/2^{j})P_c g\|_{L^1_x} \Big) \\
&=  C(V) t^{-\hf} \Big( \|\chi_0(H) P_c g\|_{L^1_x} + \sum_{j=0}^\infty 2^{-j\mu} \| \psi_j(H) P_c \jtD^{\frac32+\mu}  g\|_{L^1_x} \Big)
}
with $\mu > 0$ arbitrary and 
\begin{equation*}
\psi_j(H) :=  2^{(\thf+\mu)j}  \jtD^{-\frac32-\mu} \chi(\tD/2^{j}), \quad j \geq 0.
\end{equation*}
Summing up~\eqref{eq:LPdec} will complete the proof provided we have the operator bounds
\EQ{
\label{eq:multop}
\|\chi_0(H) P_c g\|_{L^1_x} \les \|g\|_{L^1_x}, \quad \sup_{j\ge0} \| \psi_j(H)P_c g\|_{L^1_x} \les \|g\|_{L^1_x}
}
with constants only depending on $H$. The latter operator bounds are immediate consequences of the kernel bounds~\eqref{equ:wtilD_chi0_kernel_bound} and~\eqref{equ:jtD_chi_kernel_bound} with $N=2$ established in Lemma~\ref{lem:K bds}.
\end{proof}

Finally, we present a result on the asymptotics of the linear Klein-Gordon evolution $e^{it\jtD} P_c g$ for initial conditions supported away from zero energy.

\begin{lemma} \label{lem:asymptotics_KG}
Let $H = -\partial_x^2 + V(x)$ with real-valued $V\in L^\infty(\R)\cap C^{3}(\R)$, and assume that $\jap{x}^6 V^{(\ell)}(x) \in L^1(\R)$ for all $0\le\ell\le 3$. Let $\chi_0(\xi^2)$ be a smooth cutoff to $|\xi| \lesssim 1$, equal to $1$ near $\xi = 0$. Set $\chi_1(H) := 1 - \chi_0(H)$.
Then we have 
 \begin{equation} \label{eq:stern}
  \Bigl\| e^{it\jtD} \chi_1(H) P_c g - \frac{1}{t^\hf} e^{i\frac{\pi}{4}} e^{i\rho} \chi_1(\xi_0^2) \jap{\xi_0}^\thf \wtilg(\xi_0) \one_{(-1,1)}({\textstyle \frac{x}{t}}) \Bigr\|_{L^\infty_x} \leq \frac{C(V, \chi_1)}{t^{\frac23}} \|\jx g\|_{H^2_x}, \quad t \geq 1,
 \end{equation}
 where $\rho \equiv \rho(t,x) := \sqrt{t^2 - x^2}$ and $\frac{\xi_0}{\jap{\xi_0}} = - \frac{x}{t}$.
\end{lemma}
Before we turn to the proof of Lemma~\ref{lem:asymptotics_KG}, we record the useful relations
\begin{equation*}
 \frac{\xi_0}{\jap{\xi_0}} = - \frac{x}{t} \quad \Leftrightarrow \quad \xi_0 = - \frac{x}{\rho}, \qquad \jap{\xi_0} = \frac{t}{\rho}.
\end{equation*}
\begin{proof}[Proof of Lemma~\ref{lem:asymptotics_KG}]
We have 
\EQ{\label{eq:high xi}
\big( e^{i t \jtD}  \chi_1(H) P_c g\big)(x) 
& =  \frac{1}{\sqrt{2\pi}} \int_{0}^\infty  e^{it\jap{\xi}} T(\xi)f_+(x,\xi)  \chi_1(\xi^2) \tilde g(\xi)\, \ud\xi \\
&\quad + \frac{1}{\sqrt{2\pi}} \int_{-\infty}^0 e^{it\jap{\xi}} T(-\xi) f_-(x,-\xi) \chi_1(\xi^2) \tilde g(\xi)\, \ud\xi  \\
& =: (\calE^+(t) \tilde g)(x) + (\calE^-(t) \tilde g)(x)
}
By reflection symmetry it suffices to assume that $x\ge0$. Then with $\phi_\pm(\xi,u) := \jap{\xi} \pm  u\xi$, $u:=x/t$, 
\EQ{\nn
(\calE^+(t) \tilde g)(x) & =   \int_{0}^\infty  e^{it\phi_+(\xi,u)} T(\xi)m_+(x,\xi)  \chi_1(\xi^2) \tilde g(\xi)\, \frac{\ud \xi}{\sqrt{2\pi}}    \\
(\calE^-(t) \tilde g)(x)  & =  \int_{-\infty}^0 \Big( e^{it\phi_+(\xi,u)} m_+(x,\xi)+ e^{it\phi_-(\xi,u)} R_+(-\xi) m_+(x,-\xi))  \chi_1(\xi^2) \tilde g(\xi)\,\frac{\ud \xi}{\sqrt{2\pi}}  
}
If $x\ge t\ge1$, then 
\EQ{\label{eq:phase der}
|\partial_\xi \phi_\pm (\xi,u )| &= |\xi\jap{\xi}^{-1} \pm u |\ge 1 - |\xi|\jap{\xi}^{-1}\ge  \jap{\xi}^{-2}/2 \\
\partial_\xi^2 \phi_\pm (\xi,u ) &= \jap{\xi}^{-3}
}
We break up the integration in $\calE^+(t)$ by means of the smooth partition of unity $1 = \chi_1(\xi^2/t) + \chi_0(\xi^2/t)$ and integrate by parts in the latter integral. 
Using \eqref{eq:phase der} and the bounds on $m_+, T$ from above yields 
\begin{align}
|(\calE^+(t) \tilde g)(x)| &\les \int_{0}^\infty \chi_1(\xi^2/t) |\tilde g(\xi)|\, \ud \xi + t^{-1} \int_{0}^\infty  \frac{|\partial_\xi^2\phi_+(\xi,u)|}{ \partial_\xi \phi_\pm (\xi,u )^2}  \chi_0(\xi^2/t) \chi_1(\xi^2) |\tilde g(\xi)| \, \ud \xi \nn \\
&\qquad + t^{-1} \int_{0}^\infty  | \partial_\xi \phi_\pm (\xi,u )|^{-1}  \Big| \partial_\xi \big[T(\xi)m_+(x,\xi) \chi_0(\xi^2/t) \chi_1(\xi^2) \tilde g(\xi)\big] \Big| \, \ud \xi \nn \\
& \les t^{-\frac34} \bigl( \| \jap{\xi}^2 \tilde g(\xi)\|_{L^2_\xi} + \| \jap{\xi}^2 \partial_\xi \tilde g(\xi)\|_{L^2_\xi} \bigr) \label{eq:E+}
\end{align}
By analogous calculations one derives the same bound on $(\calE^-(t) \tilde g)(x)$. Now suppose $|x|< t$. The phases $\phi_\pm(\xi_0,u)$ have stationary points given by $\xi_0^{\pm}=\mp\jap{\xi_0^\pm} u$ or $\xi_0^\pm=\frac{\mp u}{\sqrt{1-u^2}}$. In either case one has $\phi_{\pm}(\xi_0^\pm, u)=\sqrt{1-u^2}$ which implies $t \phi_{\pm}(\xi_0^\pm, u)=\sqrt{t^2-x^2}=\rho$. 

We now claim that the bounds~\eqref{eq:phase der} continue to hold (up to multiplicative constants) for all $\xi\in\R\setminus I(\xi_0^\pm)$ where 
\[
I(\xi_0^\pm):=[\xi_0^\pm - \jap{\xi_0^\pm}/100, \xi_0^\pm + \jap{\xi_0^\pm}/100]
\]  In fact,
\EQ{\label{eq:p phi}
|\partial_\xi \phi_{\pm}(\xi, u) | &= |\partial_\xi \phi_{\pm}(\xi, u)- \partial_\xi \phi_{\pm}(\xi_0^\pm, u)| = | \xi\jap{\xi}^{-1} - \xi_0\jap{\xi_0}^{-1}| \\
& = \frac{|\xi^2-\xi_0^2|}{(\jap{\xi}\jap{\xi_0})^2 | \xi\jap{\xi}^{-1} + \xi_0\jap{\xi_0}^{-1}|}
}
where we dropped the $\pm$ superscript for simplicity. Without loss of generality, assume $\xi_0\ge0$. Then if $\xi\ge    \xi_0 + \jap{\xi_0}/100$, \eqref{eq:p phi} implies that
\[
|\partial_\xi \phi_{\pm}(\xi, u) |  \gtrsim \jap{\xi_0}^{-2}\gtrsim \jap{\xi}^{-2}
\]
while for $\xi\le    \xi_0 - \jap{\xi_0}/100$ one has
$
|\partial_\xi \phi_{\pm}(\xi, u) |  \gtrsim \jap{\xi}^{-2}
$.
Setting $$\omega_u^{\pm}(\xi):= \chi\big(C_0(\xi-\xi_0^\pm)\jap{\xi_0^\pm}^{-1}\big)$$ for some large constant $C_0$, and repeating the arguments leading to~\eqref{eq:E+} therefore yields 
\EQ{\label{eq:zwisch}
&\Big|  \big( e^{i t \jtD}  \chi_1(H) P_c g\big)(x) - \int_0^\infty  e^{it\phi_+(\xi,u)}  T(\xi) m_+(x,\xi) \chi_1(\xi^2) \omega_u^{+}(\xi)\tilde g(\xi)\,\frac{\ud \xi}{\sqrt{2\pi}} \\
&\qquad -\int_{-\infty}^0  e^{it\phi_+(\xi,u)} m_+(x,\xi) \chi_1(\xi^2) \omega_u^{+}(\xi)\tilde g(\xi)\,\frac{\ud \xi}{\sqrt{2\pi}} \\
&\qquad - \int_{-\infty}^0 e^{it\phi_-(\xi,u)} R_+(-\xi) m_+(x,-\xi)  \chi_1(\xi^2) \omega_u^{-}(\xi) \tilde g(\xi)\,\frac{\ud \xi}{\sqrt{2\pi}}  \Big| \\
& \les t^{-\frac34} \bigl( \| \jap{\xi}^2 \tilde g(\xi)\|_{L^2_\xi} + \| \jap{\xi}^2 \partial_\xi \tilde g(\xi)\|_{L^2_\xi} \bigr) 
}
which holds uniformly in $t\ge1$ and $x\ge0$. Note that $\chi_1(\xi^2) \omega_u^{+}(\xi)=0$ on $\xi\ge0$, $ \chi_1(\xi^2) \omega_u^{-}(\xi)=0$ on $\xi\le0$ if $C_0$ is large. Thus, only the second integral in~\eqref{eq:zwisch} contributes, and we denote it by $(\Psi(t)\tilde g)(x)$. To analyze it, we write (again suppressing superscripts $\pm$)
\EQ{\nn
\phi_+(\xi,u) - \phi_+(\xi_0,u) &= \jap{\xi} + u\xi - \jap{\xi_0} - u\xi_0 \\
&= \frac{(\xi-\xi_0)^2}{\jap{\xi_0} (1+\xi\xi_0+\jap{\xi}\jap{\xi_0})}=:\eta^2
}
The change of variables $\xi\mapsto\eta$ is a diffeomorphism on the support of $\omega_u^{+}(\xi)$ given by 
\[
\eta = \frac{\xi-\xi_0}{\sqrt{\jap{\xi_0} (1+\xi\xi_0+\jap{\xi}\jap{\xi_0})} },\quad \frac{d\eta}{d\xi} \simeq \jap{\xi_0}^{-\frac32}
\]
Therefore, using the standard Fourier transform, 
\begin{align}
(\Psi(t)\tilde g)(x) &= \frac{e^{i\rho} }{\sqrt{2\pi}}\int_{-\infty}^\infty  e^{it\eta^2}\; \ol{G}(\eta;t,x) \, \ud \eta \nn\\
 &= \frac{e^{i\rho}e^{i\frac{\pi}{4}}}{2\pi\sqrt{2t}}  \int_{-\infty}^\infty e^{-\frac{iy^2}{4t}}\; \ol{\widehat{G}(y;t,x)} \, \ud  y \label{eq:G} \\
\ol{G}(\eta;t,x) &= m_+(x,\xi) \chi_1(\xi^2) \omega_u^{+}(\xi)\tilde g(\xi) \frac{d\xi}{d\eta} \nn \\
\frac{1}{2\pi} \int_{-\infty}^\infty  \ol{\widehat{G}(y;t,x)} \, \ud  y &= \ol{G}(0;t,x) = m_+(x,\xi_0) \chi_1(\xi_0^2) \omega_u^{+}(\xi_0)\tilde g(\xi_0) \frac{d\xi}{d\eta}(\xi_0) \nn \\
&= \sqrt{2}\, m_+(x,\xi_0) \chi_1(\xi_0^2) \jap{\xi_0}^{\frac32}\tilde g(\xi_0) \label{eq:rt 2}
\end{align}
By \eqref{eq:m+-1}, $|m_+(x,\xi_0)-1|\les \jap{x}^{-5} $ with an implicit constant depending only on $V$. If $\chi_1(\xi_0^2)=1$, then $|\xi_0|\gtrsim 1$, and $|x/t|=|\xi_0|\jap{\xi_0}^{-1}\gtrsim 1$. Therefore,
\EQ{\nn
\eqref{eq:rt 2} = \sqrt{2}\,  \chi_1(\xi_0^2) \jap{\xi_0}^{\frac32}\tilde g(\xi_0) + \calO(t^{-5} \chi_1(\xi_0^2)\jap{\xi_0}^{\frac32}\tilde g(\xi_0))
}
and inserting this into \eqref{eq:G}, 
\EQ{\label{eq:schr2}
(\Psi(t)\tilde g)(x) &= \frac{e^{i\rho}e^{i\frac{\pi}{4}}}{\sqrt{t}} \chi_1(\xi_0^2) \jap{\xi_0}^{\frac32}\tilde g(\xi_0) + \calO(t^{-\frac{11}{2}} \chi_1(\xi_0^2)\jap{\xi_0}^{\frac32}\tilde g(\xi_0)) \\
&\qquad + \calO\Big( t^{-\frac12} \int \big| e^{-\frac{iy^2}{4t}} -1 \big| \, |{\widehat{G}(y;t,x)}  |\, \ud y\Big)
}
The integral in the last line is estimated as follows:
\EQ{\nn
 \int \big| e^{-\frac{iy^2}{4t}} -1 \big| \, |{\widehat{G}(y;t,x)}  |\, \ud y & \les t^{-\frac12}\int_{[|y|^2\le t]} |y|\,  |{\widehat{G}(y;t,x)}  |\, \ud y  + \int_{[|y|^2\ge t]} \,  |{\widehat{G}(y;t,x)}  |\, \ud y  \\
 &\les t^{-\frac14} \| y \, {\widehat{G}(y;t,x)}  \|_{L^2_y} \les t^{-\frac14} \|\partial_\eta {G(\eta;t,x)}  \|_{L^2_\eta}
}
By definition, 
\EQ{\nn 
\int \big|\partial_\eta {{G}(\eta;t,x)}\big|^2\, d\eta & = \int \Big|\frac{d\xi}{d\eta}\Big|\, \big|\partial_\xi\big[m_+(x,\xi) \chi_1(\xi^2) \omega_u^{+}(\xi)\tilde g(\xi) \frac{d\xi}{d\eta}\big] \big|^2\ud \xi\\
&\les  \int  \big|\partial_\xi\big[m_+(x,\xi) \chi_1(\xi^2) \omega_u^{+}(\xi)\tilde g(\xi) \frac{d\xi}{d\eta}\big] \big|^2 \jap{\xi}^{\frac32}\ud \xi
}
Now we note that by complex interpolation of the preceding bound with 
\[
 \int \big|  {{G}(\eta;t,x)}\big|^2\, d\eta \les \int \big|m_+(x,\xi) \chi_1(\xi^2) \omega_u^{+}(\xi)\tilde g(\xi) \frac{d\xi}{d\eta} \big|^2 \jap{\xi}^{-\frac32} \ud \xi
\]
we obtain that for all $\frac12<\beta\le1$, 
\EQ{\nn 
 \int \big| e^{-\frac{iy^2}{4t}} -1 \big| \, |{\widehat{G}(y;t,x)}  |\, \ud y  &\les t^{\frac14-\frac\beta2} \big\| |y|^{\beta} \, {\widehat{G}(y;t,x)}  \big\|_{L^2_y} \\
 &\les t^{\frac14-\frac\beta2} \| (-\partial_\eta^2)^{\frac{\beta}{2}} {G(\eta;t,x)}  \|_{L^2_\eta}  \\
 &\les t^{\frac14-\frac\beta2} \Big(\int  \big|(-\partial_\xi^2)^{\frac{\beta}{2}}\big[m_+(x,\xi) \chi_1(\xi^2) \omega_u^{+}(\xi)\tilde g(\xi) \frac{d\xi}{d\eta} \big]\big|^2 \jap{\xi}^{-\frac32+3\beta}\ud \xi\Big)^{\frac12}
}
On the one hand,
\EQ{\nn
 \Big(\int  \big|m_+(x,\xi) \chi_1(\xi^2) \omega_u^{+}(\xi)\tilde g(\xi) \frac{d\xi}{d\eta} \big|^2\ud \xi\Big)^{\frac12} &\les \Big(\int |\tilde g(\xi)  |^2 \jap{\xi}^3 \ud \xi\Big)^{\frac12} }
and, on the other hand, 
\EQ{\nn
\Big(\int  \big|(-\partial_\xi^2)^{\frac{1}{2}}\big[m_+(x,\xi) \chi_1(\xi^2) \omega_u^{+}(\xi)\tilde g(\xi) \frac{d\xi}{d\eta} \big]\big|^2\ud \xi\Big)^{\frac12}
&\les \Big(\int \big[ |\tilde g(\xi) |^2 \jap{\xi}+ |\partial_\xi \tilde g(\xi)  |^2 \jap{\xi}^3\big]\, \ud \xi\Big)^{\frac12}
}
In conclusion, we can bound with $\beta=\frac56$, 
\EQ{\label{eq:schr3}
  \int \big| e^{-\frac{iy^2}{4t}} -1 \big| \, |{\widehat{G}(y;t,x)}  |\, \ud y   & \les t^{-\frac16} \Big(\int \big[ |\tilde g(\xi) |^2 + |\partial_\xi \tilde g(\xi)  |^2 \big]\, \jap{\xi}^4 \ud \xi\Big)^{\frac12}
}
Combining \eqref{eq:E+}, \eqref{eq:zwisch}, \eqref{eq:schr2}, and \eqref{eq:schr3} yields 
\EQ{\label{eq:schr4}
&\Big|  \big( e^{i t \jtD}  \chi_1(H) P_c g\big)(x) -  \frac{e^{i\rho}e^{i\frac{\pi}{4}}}{\sqrt{t}} \chi_1(\xi_0^2) \jap{\xi_0}^{\frac32}\tilde g(\xi_0)\one_{[|x/t|<1]} \Big| \\
& \les t^{-\frac{11}{2}} \chi_1(\xi_0^2)\jap{\xi_0}^{\frac32}|\tilde g(\xi_0)|+  t^{-\frac23} \bigl( \| \jap{\xi}^2 \tilde g(\xi)\|_{L^2_\xi} + \| \jap{\xi}^2 \partial_\xi \tilde g(\xi)\|_{L^2_\xi} \bigr) \\
&\les  t^{-\frac23} \bigl( \| \jap{\xi}^2 \tilde g(\xi)\|_{L^2_\xi} + \| \jap{\xi}^2 \partial_\xi \tilde g(\xi)\|_{L^2_\xi} \bigr) 
}
which holds uniformly in $t\ge1$ and $x\ge0$. The $t^{-\frac{11}{2}}$-term on the second line is estimated by Sobolev, and $\xi_0=\xi_0^+=-\frac{ x}{\sqrt{t^2-x^2}}$. 
We may of course assume that $\tilde g(\xi)=0$ for $|\xi|\les 1$. It remains to prove that for all Schwartz functions $g$, 
\[
\| \jap{\xi}^2 \chi_1(\xi^2)  \tilde g(\xi)\|_{L^2_\xi} + \| \chi_1(\xi^2) \jap{\xi}^2 \partial_\xi \tilde g(\xi)\|_{L^2_\xi} \les \| \jap{\cdot} g\|_{H^2_x}
\]
This follows again by means of Calderon-Vaillancourt, see~\eqref{eq:again CV}. The estimate~\eqref{eq:pg bd} controls $\|\chi_1(\xi^2)\partial_\xi \tilde g(\xi)\|_{L^2_\xi}$ by $\|\jap{x} g(x)\|_{L^2_x}$. To 
incorporate the $\xi^2$ factor, we compute
 \EQ{\nn
 \xi^2 \widetilde{g}(\xi) 
 & = \int_{\bbR} \xi^2 \overline{e(x,\xi)}\, \theta_+(x) g(x) \, \ud x + \int_{\bbR}  \xi^2  \overline{e(x,\xi)}\, \theta_-(x) g(x) \, \ud x \\
 & = -\frac{1}{\sqrt{2\pi}}  \int_{\bbR}  \partial_x^2  (e^{-ix\xi})  T(-\xi) m_+(x,-\xi) \theta_+(x) g(x) \, \ud x  \\
 &\quad - \frac{1}{\sqrt{2\pi}}  \int_{\bbR}  (\partial_x^2  (e^{-ix\xi})  m_-(x,\xi) + R_-(-\xi) \partial_x^2  (e^{ix\xi}) m_-(x,-\xi)) \theta_-(x) g(x) \, \ud x  
 }
Integrating by parts in $x$, and multiplying by $\chi_1(\xi^2)$, we can then apply Calderon-Vaillancourt as in~\eqref{eq:again CV}.
\end{proof}

\section{Local Decay Bounds} \label{sec:local_decay_bounds}

The main goal of this section is to establish global existence and local decay bounds for the solution $v(t)$ to~\eqref{equ:thm1_nlkg}.
The key ingredient for the proof are the local $L^2_x$ decay estimates for the Klein-Gordon propagator $e^{it\jtD} P_c \equiv  e^{i t \sqrt{1+H}} P_c$ established in Corollary~\ref{cor:main}.

\begin{proposition}[Global existence and local decay bounds] \label{prop:local_decay_bounds}
 Assume that $V(x)$ and $\alpha(x)$ are as in the statement of Theorem~\ref{thm:thm1}, and let $\sigma = 5$.
 There exists a small absolute constant $\varepsilon_0 > 0$ so that for any initial datum $v_0$ with
 \begin{equation*}
  \varepsilon := \|\jx^\sigma v_0\|_{H^2_x} \leq \varepsilon_0,
 \end{equation*}
 there exists a global-in-time solution $v \in C(\bbR; H^2_x)$ to~\eqref{equ:thm1_nlkg} satisfying the uniform bounds
 \begin{equation} \label{equ:local_decay_bounds_for_v}
  \begin{aligned}
   \sup_{t\in\bbR} \, \Bigl\{ \jt^{-(0+)} \| v(t) \|_{H^2_x} &+ \jt^\hf \|\jx^{-\sigma} v(t)\|_{L^2_x} + \jt \|\jx^{-\sigma} (1-\chi_0(H)) v(t)\|_{L^2_x} \\
   &\quad \quad + \jt \|\jx^{-\sigma} \sqrt{H} v(t)\|_{L^2_x} + \jt \|\jx^{-\sigma} \pt (e^{-it} v(t))\|_{L^2_x} \Bigr\} \lesssim \varepsilon.
  \end{aligned}
 \end{equation}
\end{proposition}
\begin{proof}
By time reversal symmetry, it suffices to argue forward in time. Assuming that the absolute constant $0 < \varepsilon_0 \ll 1$ is sufficiently small, by a standard contraction mapping argument, we obtain a local solution $v \in C([0,T_0]; H^2_x)$ on a time interval $[0,T_0]$ for some $T_0 \geq 1$.
In order to then conclude that $v(t)$ exists globally in time, it is enough to show that the $H^2_x$ norm of $v(t)$ does not blow up in finite time.
We now establish global existence and the uniform bounds~\eqref{equ:local_decay_bounds_for_v} via a bootstrap argument. 
For any $0 < T \leq T_0$ we consider the bootstrap quantity
\begin{equation*}
 \begin{aligned}
  M(T) := \sup_{0 \leq t \leq T} \, \Bigl\{ \jt^{-(0+)} \| v(t) \|_{H^2_x} &+ \jt^\hf \|\jx^{-\sigma} v(t)\|_{L^2_x} + \jt \|\jx^{-\sigma} (1-\chi_0(H)) v(t)\|_{L^2_x} \\ 
   &\quad \quad + \jt \|\jx^{-\sigma} \sqrt{H} v(t)\|_{L^2_x} + \jt \|\jx^{-\sigma} \pt (e^{-it} v(t))\|_{L^2_x} \Bigr\}. 
 \end{aligned}
\end{equation*}
Since the absolute constant $0 < \varepsilon_0 \ll 1$ can be chosen sufficiently small, in what follows we may freely assume that $T \geq 1$, and that $M(T) \leq 1$ to simplify the bookkeeping for some of the nonlinear estimates. 
We also recall that under the assumptions of Theorem~\ref{thm:thm1}, we have that $v(t) = P_c v(t)$ for all $t \in [0,T_0]$.
Moreover, we stress that we will frequently use that by the weighted Sobolev estimate from Lemma~\ref{lem:weighted_sobolev},
\begin{equation} \label{equ:Linfty_v_local_decay_bound}
 \|\jx^{-\sigma} v(t)\|_{L^\infty_x} \lesssim \|\jx^{-\sigma} v(t)\|_{L^2_x} + \|\jx^{-\sigma} \sqrt{H} v(t)\|_{L^2_x} \leq \frac{M(T)}{\jt^\hf}, \quad 0 \leq t \leq T.
\end{equation}
In order to derive bounds on all components of the bootstrap quantity $M(T)$, we work with Duhamel's formula for the solution $v(t)$ given by
\begin{equation*}
 v(t) = e^{it\jtD} P_c v_0 + \frac{1}{2i} \int_0^t e^{i(t-s)\jtD} \jtD^{-1} P_c \bigl( \alpha(\cdot) u(s)^2 \bigr) \, \ud s.
\end{equation*}

\medskip 

\noindent \underline{{\it Growth bound for the $H^2_x$ norm of $v(t)$}:}
We begin with a growth bound for the $H^2_x$ norm of the solution $v(t)$. Using the equivalence of norms from Lemma~\ref{lem:equiv_norms} and the product estimate~\eqref{equ:prod_est2}, we obtain from Duhamel's formula for $v(t)$ for any $0 \leq t \leq T$ that
\begin{align*}
 \|v(t)\|_{H^2_x} &\lesssim \|\jtD^2 v(t)\|_{L^2_x} \lesssim \|\jtD^2 P_c v_0\|_{L^2_x} + \int_0^t \bigl\| \jtD P_c \bigl( \alpha(\cdot) u(s)^2 \bigr) \bigr\|_{L^2_x} \, \ud s \\
 &\lesssim \|v_0\|_{H^2_x} + \int_0^t \|\jx^{2\sigma} \alpha\|_{W^{1,\infty}_x} \bigl( \| \jx^{-\sigma} v(s) \|_{L^2_x} + \| \jx^{-\sigma} \sqrt{H} v(s) \|_{L^2_x} \bigr)^2 \, \ud s \\
 &\lesssim \|v_0\|_{H^2_x} + \int_0^t \frac{M(T)^2}{\js} \, \ud s \\
 &\lesssim \|v_0\|_{H^2_x} + \log(1+\jt) M(T)^2.
\end{align*}

\medskip 

\noindent \underline{{\it Local decay for $\pt (e^{-it} v(t))$}:}
Now we derive an improved local decay bound for the time derivative of the phase-filtered component $\partial_t \bigl( e^{-it} v(t) \bigr)$. To this end we compute that 
\begin{equation} 
 \begin{aligned}
  \partial_t \bigl( e^{-it} v(t) \bigr) = e^{-it} \biggl( (\jtD-1) e^{it\jtD} P_c v_0 &+ \frac{1}{2 i} \jtD^{-1} P_c \bigl( \alpha(\cdot) u(t)^2 \bigr) \\ 
  &+ \frac{1}{2} \int_0^t \frac{\jtD-1}{\jtD} e^{i(t-s) \jtD} P_c \bigl( \alpha(\cdot) u(s)^2 \bigr) \, \ud s \biggr).
 \end{aligned}
\end{equation}
By the local decay estimate~\eqref{equ:local_decay_jtDminus1_est} for the Klein-Gordon propagator, it then follows for $0 \leq t \leq T$ 
\begin{align*}
 &\bigl\| \jap{x}^{-\sigma} \partial_t \bigl( e^{-it} v(t) \bigr) \bigr\|_{L^2_x} \\
 &\quad \lesssim \frac{\|\jx^\sigma \jtD P_c v_0 \|_{L^2_x}}{\jt^\thf} + \|\jap{x}^{2 \sigma} \alpha\|_{L^2_x} \|\jap{x}^{-\sigma} v(t)\|_{L^\infty_x}^2 \\
 &\quad \quad \quad + \int_0^t \, \biggl\| \jap{x}^{-\sigma} \frac{\jtD-1}{\jtD} e^{i(t-s)\jtD} P_c \jap{x}^{-\sigma} \biggr\|_{L^2_x \to L^2_x} \bigl\| \jap{x}^{3 \sigma} \alpha \bigr\|_{L^2_x} \|\jap{x}^{-\sigma} v(s)\|_{L^\infty_x}^2 \, \ud s \\
 &\quad \lesssim \frac{\|\jx^\sigma v_0 \|_{H^1_x}}{\jt^\thf} + \frac{M(T)^2}{\jt} + \int_0^t \frac{1}{\jap{t-s}^{\thf}} \frac{M(T)^2}{\js} \, \ud s \\
 &\quad \lesssim \frac{1}{\jap{t}} \bigl( \|\jx^\sigma v_0 \|_{H^1_x} + M(T)^2 \bigr).
\end{align*}

\medskip 

\noindent \underline{{\it Local decay for $(1-\chi_0(H)) v(t)$}:} 
Next, we conclude an improved local decay bound for the high-energy part $(1-\chi_0(H)) v(t)$ of the solution $v(t)$. Observe that $(1-\chi_0(H)) v(t)$ is given in Duhamel form by
\begin{equation*}
 (1-\chi_0(H)) v(t) = (1-\chi_0(H)) e^{it\jtD} P_c v_0 + \frac{1}{2i} \int_0^t (1-\chi_0(H)) e^{i(t-s)\jtD} \jtD^{-1} P_c \bigl( \alpha(\cdot) u(s)^2 \bigr) \, \ud s.
\end{equation*}
Using the improved local decay estimates~\eqref{equ:local_decay_away_zero_freq_est} and~\eqref{equ:local_decay_away_zero_freq_tDinv_est} for the Klein-Gordon propagator away from zero energy, it is straightforward to obtain for $0 \leq t \leq T$ the desired bound
\begin{align*}
 &\| \jap{x}^{-\sigma} (1-\chi_0(H)) v(t) \|_{L^2_x} \\
 &\quad \lesssim \frac{\|\jx^\sigma (1-\chi_0(H)) P_c v_0 \|_{L^2_x}}{\jt^\thf} \\
 &\quad \quad \quad + \int_0^t \Bigl\| \jap{x}^{-\sigma} (1-\chi_0(H)) \jtD^{-1} e^{i(t-s)\jtD} P_c \jap{x}^{-\sigma} \Bigr\|_{L^2_x \to L^2_x} \bigr\| \jap{x}^\sigma \bigl( \alpha(\cdot) u(s)^2 \bigr) \bigr\|_{L^2_x} \, \ud s \\
 &\quad \lesssim \frac{\|\jx^\sigma v_0 \|_{L^2_x}}{\jap{t}^{\frac{3}{2}}} + \int_0^t \frac{1}{\jap{t-s}^{\frac{3}{2}}} \| \jap{x}^{3\sigma} \alpha \|_{L^2_x} \|\jap{x}^{-\sigma} v(s)\|_{L^\infty_x}^2 \, \ud s \\
 &\quad \lesssim \frac{\|\jx^\sigma v_0 \|_{L^2_x}}{\jap{t}^{\frac{3}{2}}} + \int_0^t \frac{1}{\jap{t-s}^{\frac{3}{2}}} \frac{M(T)^2}{\jap{s}} \, \ud s \\
 &\quad \lesssim \frac{1}{\jap{t}} \bigl( \|\jx^\sigma v_0 \|_{L^2_x} + M(T)^2 \bigr).
\end{align*}

\medskip 

\noindent \underline{{\it Local decay for $\sqrt{H} v(t)$:}} 
In a similar manner we obtain an improved local decay bound for $\sqrt{H} v(t)$. Note that $\sqrt{H} v(t)$ is given in Duhamel form by
\begin{equation*}
 \sqrt{H} v(t) = \sqrt{H} e^{it\jtD} P_c v_0 + \frac{1}{2i} \int_0^t \sqrt{H} \jtD^{-1} e^{i(t-s)\jtD} P_c \bigl( \alpha(\cdot) u(s)^2 \bigr) \, \ud s.
\end{equation*}
Using the improved local decay estimate~\eqref{equ:local_decay_sqrtH_est} for the Klein-Gordon propagator, we easily obtain for $0 \leq t \leq T$ the desired bound
\begin{align*}
 &\| \jap{x}^{-\sigma} \sqrt{H} v(t) \|_{L^2_x} \\
 &\quad \lesssim \frac{\|\jx^\sigma \jtD P_c v_0 \|_{L^2_x}}{\jt^\thf} + \int_0^t \Bigl\| \jap{x}^{-\sigma} \sqrt{H} \jtD^{-1} e^{i(t-s)\jtD} P_c \jap{x}^{-\sigma} \Bigr\|_{L^2_x \to L^2_x} \bigr\| \jap{x}^\sigma \bigl( \alpha(\cdot) u(s)^2 \bigr) \bigr\|_{L^2_x} \, \ud s \\
 &\quad \lesssim \frac{\|\jx^\sigma v_0 \|_{H^1_x}}{\jap{t}^{\frac{3}{2}}} + \int_0^t \frac{1}{\jap{t-s}^{\frac{3}{2}}} \| \jap{x}^{3\sigma} \alpha \|_{L^2_x} \|\jap{x}^{-\sigma} v(s)\|_{L^\infty_x}^2 \, \ud s \\
 &\quad \lesssim \frac{\|\jx^\sigma v_0 \|_{H^1_x}}{\jap{t}^{\frac{3}{2}}} + \int_0^t \frac{1}{\jap{t-s}^{\frac{3}{2}}} \frac{M(T)^2}{\jap{s}} \, \ud s \\
 &\quad \lesssim \frac{1}{\jap{t}} \bigl( \|\jx^\sigma v_0 \|_{H^1_x} + M(T)^2 \bigr).
\end{align*}

\medskip 

\noindent \underline{{\it Local decay for $v(t)$}:} Finally, the derivation of the local decay bound for $v(t)$ requires a much more careful argument. 
The first step is to determine the leading order behavior of the variable coefficient quadratic nonlinearity. 
To this end we introduce the function
\begin{equation} \label{equ:definition_w}
 w(t) := c_0 \frac{e^{i\frac{\pi}{4}} e^{it}}{t^\hf} \langle \varphi, v_0 \rangle \varphi + \frac{1}{2i} \int_0^{t-1} c_0 \frac{e^{i\frac{\pi}{4}} e^{i(t-s)}}{(t-s)^{\hf}} \bigl\langle \varphi, \alpha(\cdot) u(s)^2 \bigr\rangle \varphi \, \ud s, \quad t \geq 1,
\end{equation}
with $c_0$ defined in~\eqref{equ:def_c0}.
Then we may write
\begin{equation*}
 w(t,x) = a(t) \varphi(x), \quad t \geq 1,
\end{equation*}
with the time-dependent coefficient 
\begin{equation} \label{equ:definition_a}
 a(t) := c_0 \frac{e^{i\frac{\pi}{4}} e^{it}}{t^\hf} \langle \varphi, v_0 \rangle + \frac{1}{2i} \int_0^{t-1} c_0 \frac{e^{i\frac{\pi}{4}} e^{i(t-s)}}{(t-s)^{\hf}} \bigl\langle \varphi, \alpha(\cdot) u(s)^2 \bigr\rangle \, \ud s.
\end{equation}
The key property of the function $w(t)$ is that the difference $v(t) - w(t)$ has faster local decay in the sense that uniformly for all $1 \leq t \leq T$
\begin{equation} \label{equ:w_minus_v_local_decay_est} 
 \bigl\| \jx^{-\sigma} \bigl( v(t) - w(t) \bigr) \bigr\|_{L^2_x} \lesssim \frac{1}{\jt} \bigl( \|\jx^\sigma v_0\|_{L^2_x} + M(T)^2 \bigr).
\end{equation}
To prove~\eqref{equ:w_minus_v_local_decay_est} we write for any $1 \leq t \leq T$
\begin{align*}
 v(t) - w(t) &= \biggl( e^{it\jtD} P_c v_0 - c_0 \frac{e^{i\frac{\pi}{4}} e^{it}}{t^\hf} \langle \varphi, v_0 \rangle \varphi \biggr) \\
 &\quad + \frac{1}{2i} \int_0^{t-1} \biggl( e^{i(t-s)\jtD} \jtD^{-1} P_c \bigl( \alpha(\cdot) u(s)^2 \bigr) - c_0 \frac{e^{i\frac{\pi}{4}} e^{i(t-s)}}{(t-s)^\hf} \bigl\langle \varphi, \alpha(\cdot) u(s)^2 \bigr\rangle \varphi \biggr) \, \ud s \\
 &\quad + \frac{1}{2i} \int_{t-1}^t e^{i(t-s)\jtD} \jtD^{-1} P_c \bigl( \alpha(\cdot) u(s)^2 \bigr) \, \ud s.
\end{align*}
Then we use the improved local decay estimates~\eqref{equ:local_decay_improved_subtr_off} and~\eqref{equ:local_decay_improved_tDinv_subtr_off} to bound the first two terms on the right-hand side, while the standard local decay estimate~\eqref{equ:local_decay_est} suffices to estimate the third term on the right-hand side. Specifically, we obtain for $1 \leq t \leq T$ that
\begin{align*}
 \bigl\| \jx^{-\sigma} \bigl( v(t) - w(t) \bigr) \bigr\|_{L^2_x} &\lesssim \frac{\|\jx^\sigma v_0\|_{L^2_x}}{\jt^\thf} + \int_0^{t-1} \frac{1}{\jap{t-s}^\thf} \bigl\| \jx^\sigma \alpha(x) u(s)^2 \bigr\|_{L^2_x} \, \ud s \\
 &\quad \quad + \int_{t-1}^t \frac{1}{\jap{t-s}^\hf} \bigl\| \jx^\sigma \alpha(x) u(s)^2 \bigr\|_{L^2_x} \, \ud s \\
 &\lesssim \frac{\|\jx^\sigma v_0\|_{L^2_x}}{\jt^\thf} + \int_0^{t-1} \frac{1}{\jap{t-s}^\thf} \bigl\| \jx^{3 \sigma} \alpha \bigr\|_{L^2_x} \| \jx^{-\sigma} v(s) \|_{L^\infty_x}^2 \, \ud s \\
 &\quad \quad + \int_{t-1}^t \frac{1}{\jap{t-s}^\hf} \bigl\| \jx^{3 \sigma} \alpha \bigr\|_{L^2_x} \| \jx^{-\sigma} v(s) \|_{L^\infty_x}^2 \, \ud s \\
 &\lesssim \frac{\|\jx^\sigma v_0\|_{L^2_x}}{\jt^\thf} + \int_0^{t-1} \frac{1}{\jap{t-s}^\thf} \frac{M(T)^2}{\js} \, \ud s + \int_{t-1}^t \frac{1}{\jap{t-s}^\hf} \frac{M(T)^2}{\js} \, \ud s \\
 &\lesssim \frac{1}{\jt} \bigl( \|\jx^\sigma v_0\|_{L^2_x} + M(T)^2 \bigr). 
\end{align*}
This yields the desired faster local decay bound~\eqref{equ:w_minus_v_local_decay_est}, which suggests that the leading order behavior of the variable coefficient quadratic nonlinearity $\alpha(x) (v+\bar{v})^2$ is governed by $\alpha(x) (w + \bar{w})^2$. In order to further analyze the latter, we will need the following estimates related to $w(t)$ that hold uniformly for all $1 \leq t \leq T$,
\begin{align}
 \|w(t)\|_{L^\infty_x} &\lesssim \frac{\log(1+\jt)}{\jt^\hf} \bigl( \|\jx^{\sigma} v_0\|_{L^2_x} + M(T)^2 \bigr), \label{equ:w_local_decay_est}  \\
 |a(t)| &\lesssim \frac{\log(1+\jt)}{\jt^\hf} \bigl( \| \jx^\sigma v_0\|_{L^2_x} + M(T)^2 \bigr), \label{equ:a_decay_bound} \\
 |\pt ( e^{-it} a(t) )| &\lesssim \frac{1}{\jt} \bigl( \|\jx^\sigma v_0\|_{L^2_x} + M(T)^2 \bigr). \label{equ:a_pt_decay_bound}
\end{align}
These estimates follow directly from the definitions of $w(t)$ and $a(t)$. Indeed, from the definition of $w(t)$ we obtain uniformly for all $1 \leq t \leq T$ that
\begin{align*}
 \|w(t)\|_{L^\infty_x} &\lesssim \frac{1}{t^\hf} \bigl\| \langle \varphi, v_0 \rangle \varphi(x) \bigr\|_{L^\infty_x} + \int_0^{t-1} \frac{1}{(t-s)^\hf} \bigl\| \langle \varphi, \alpha(\cdot) u(s)^2 \rangle \varphi(x)\bigr\|_{L^\infty_x} \, \ud s \\
 &\lesssim \frac{\|v_0\|_{L^1_x}}{\jt^\hf} + \int_0^{t-1} \frac{1}{(t-s)^\hf} \|\jx^{2\sigma} \alpha\|_{L^1_x} \|\jx^{-\sigma} v(s)\|_{L^\infty_x}^2 \, \ud s \\
 &\lesssim \frac{\|\jx^\sigma v_0\|_{L^2_x}}{\jt^\hf} + \int_0^{t-1} \frac{1}{(t-s)^\hf} \frac{M(T)^2}{\js} \, \ud s \\
 &\lesssim \frac{\log(1+\jt)}{\jt^\hf} \bigl( \|\jx^{\sigma} v_0\|_{L^2_x} + M(T)^2 \bigr).
\end{align*}
This proves~\eqref{equ:w_local_decay_est}. Similarly, from the definition of $a(t)$ we infer uniformly for all $1 \leq t \leq T$ that
\begin{align*}
 |a(t)| &\lesssim \frac{| \langle \varphi, v_0 \rangle|}{\jt^\hf} + \int_0^{t-1} \frac{1}{(t-s)^\hf} \|\jx^{2\sigma} \alpha(x)\|_{L^1_x} \|\jx^{-\sigma} v(s)\|_{L^\infty_x}^2 \, \ud s \\
 &\lesssim \frac{\|v_0\|_{L^1_x}}{\jt^\hf} + \int_0^{t-1} \frac{1}{(t-s)^\hf} \frac{M(T)^2}{\js} \, \ud s \\
 &\lesssim \frac{\log(1+\jt)}{\jt^\hf} \bigl( \|\jx^\sigma v_0\|_{L^2_x} + M(T)^2 \bigr),
\end{align*}
which yields~\eqref{equ:a_decay_bound}. Finally, to prove~\eqref{equ:a_pt_decay_bound} we compute 
\begin{align*}
 \pt \bigl( e^{-it} a(t) \bigr) = &-\frac12 c_0 \frac{e^{i\frac{\pi}{4}}}{t^\thf} \langle \varphi, v_0 \rangle + \frac{1}{2i} c_0 e^{i\frac{\pi}{4}} e^{-i(t-1)} \langle \varphi, \alpha(\cdot) u(t-1)^2 \rangle \\
 &-\frac{1}{4i} \int_0^{t-1} c_0 \frac{e^{i\frac{\pi}{4}} e^{-is}}{(t-s)^\thf} \langle \varphi, \alpha(\cdot) u(s)^2 \rangle \, \ud s. 
\end{align*}
Hence, we obtain uniformly for all $1 \leq t \leq T$ 
\begin{align*}
 \bigl| \pt (e^{-it} a(t)) \bigr| &\lesssim \frac{|\langle \varphi, v_0 \rangle|}{t^\thf} + \bigl\| \jx^{2\sigma} \alpha \bigr\|_{L^1_x} \| \jx^{-\sigma} v(t-1) \|_{L^\infty_x}^2 \\
 &\quad + \int_0^{t-1} \frac{1}{(t-s)^\thf} \| \jx^{2\sigma} \alpha \|_{L^1_x} \bigl\| \jx^{-\sigma} v(s) \|_{L^\infty_x}^2 \, \ud s \\
 &\lesssim \frac{\|v_0\|_{L^1_x}}{t^\thf} + \frac{M(T)^2}{\jap{t-1}} + \int_0^{t-1} \frac{1}{(t-s)^\thf} \frac{M(T)^2}{\js} \, \ud s \\
 &\lesssim \frac{1}{\jt} \bigl( \|\jx^\sigma v_0\|_{L^2_x} + M(T)^2 \bigr).
\end{align*}

We are now prepared to prove the local decay bound for $v(t)$. 
For short times $0 \leq t \leq 1$, using the local decay estimates~\eqref{equ:local_decay_est} and~\eqref{equ:local_decay_tDinv_est}, we easily obtain  
\begin{equation*}
 \sup_{0 \leq t \leq 1} \| \jx^{-\sigma} v(t) \|_{L^2_x} \lesssim \|v_0\|_{L^2_x} + M(T)^2.
\end{equation*}
It therefore suffices to consider times $t \geq 1$. We begin by decomposing the Duhamel formula for $v(t)$ into 
\begin{equation} \label{equ:decomposition_v_for_local_decay}
 \begin{aligned}
  v(t) &= e^{it\jtD} P_c v_0 + \frac{1}{2i} \int_0^t e^{i(t-s)\jtD} \jtD^{-1} P_c \bigl( \alpha(\cdot) u(s)^2 \bigr) \, \ud s \\
  &= e^{it\jtD} P_c v_0 + \frac{1}{2i} \int_0^1 e^{i(t-s)\jtD} \jtD^{-1} P_c \bigl( \alpha(\cdot) u(s)^2 \bigr) \, \ud s \\
  &\quad + \frac{1}{2i} \int_1^t e^{i(t-s)\jtD} \jtD^{-1} P_c \Bigl( \alpha(\cdot) \bigl( (v(s) + \bar{v}(s))^2 - (w(s) + \bar{w}(s))^2 \bigr) \Bigr) \, \ud s \\
  &\quad + \frac{1}{2i} \int_1^t e^{i(t-s)\jtD} \jtD^{-1} P_c \bigl( \alpha(\cdot) (w(s) + \bar{w}(s))^2 \bigr) \, \ud s \\
  &\equiv I + II + III + IV.
 \end{aligned}
\end{equation}
Then the first two terms on the right-hand side of~\eqref{equ:decomposition_v_for_local_decay} can be easily estimated. Using the standard local decay estimates~\eqref{equ:local_decay_est} and~\eqref{equ:local_decay_tDinv_est}, we have for any $1 \leq t \leq T$ 
\begin{align*}
 \|\jx^{-\sigma} I\|_{L^2_x} &= \bigl\| \jx^{-\sigma} e^{it\jtD} P_c v_0 \bigr\|_{L^2_x} \lesssim \frac{\|\jx^\sigma v_0\|_{L^2_x}}{\jt^\hf}
\end{align*}
and 
\begin{align*}
 \|\jx^{-\sigma} II\|_{L^2_x} &\lesssim \int_0^1 \bigl\| \jx^{-\sigma} \jtD^{-1} e^{i(t-s)\jtD} P_c \jx^{-\sigma} \bigr\|_{L^2_x \to L^2_x} \bigl\| \jx^{\sigma} \bigl( \alpha(\cdot) u(s)^2 \bigr) \bigr\|_{L^2_x} \, \ud s \\
 &\lesssim \int_0^1 \frac{1}{\jap{t-s}^\hf} \|\jx^{3\sigma} \alpha\|_{L^2_x} \|\jx^{-\sigma} v(s)\|_{L^\infty_x}^2 \, \ud s \\
 &\lesssim \int_0^1 \frac{1}{\jap{t-s}^\hf} \|\jx^{3\sigma} \alpha\|_{L^2_x} M(T)^2 \, \ud s \\
 &\lesssim \frac{M(T)^2}{\jt^\hf}.
\end{align*}
In order to estimate the third term on the right-hand side of~\eqref{equ:decomposition_v_for_local_decay}, we combine the local decay estimate~\eqref{equ:local_decay_est} with the improved local decay bound~\eqref{equ:w_minus_v_local_decay_est} for the difference $v(t) - w(t)$ and the bound~\eqref{equ:w_local_decay_est} for $w(t)$ to obtain for any $1 \leq t \leq T$ that
\begin{align*}
 &\|\jx^{-\sigma} III\|_{L^2_x} \\
 &\quad \lesssim \int_1^t \bigl\| \jx^{-\sigma} \jtD^{-1} e^{i(t-s)\jtD} P_c \jx^{-\sigma} \bigr\|_{L^2_x \to L^2_x} \bigl\| \jx^\sigma \alpha(x) \bigl( (v(s) + \bar{v}(s))^2 - (w(s) + \bar{w}(s))^2 \bigr\|_{L^2_x} \, \ud s \\
 &\quad \lesssim \int_1^t \frac{1}{\jap{t-s}^\hf} \| \jx^{3\sigma} \alpha \|_{L^\infty_x} \| \jx^{-\sigma} (v(s)-w(s)) \|_{L^2_x} \bigl( \|\jx^{-\sigma} v(s)\|_{L^\infty_x} + \|\jx^{-\sigma} w(s)\|_{L^\infty_x} \bigr) \, \ud s \\
 &\quad \lesssim \int_1^t \frac{1}{\jap{t-s}^\hf} \frac{1}{\js} \bigl( \|\jx^\sigma v_0\|_{L^2_x} + M(T)^2 \bigr) \biggl( \frac{M(T)}{\js^\hf} + \frac{\log(1+\js)}{\js^\hf} \bigl( \|\jx^\sigma v_0\|_{L^2_x} + M(T)^2 \bigr) \biggr) \, \ud s \\
 &\quad \lesssim \frac{1}{\jt^\hf} \bigl( \|\jx^\sigma v_0\|_{L^2_x} + M(T)^2 \bigr)^2.
\end{align*}
The estimate for the fourth term on the right-hand side of~\eqref{equ:decomposition_v_for_local_decay} is the most delicate due to the slow time decay of $\alpha(x) ( w(s) + \bar{w}(s) )^2$. We decompose the term further so that we can exploit the time oscillations of $w(s)$. 
To this end we introduce a cutoff function $\psi \in C_c^\infty(\bbR)$ that is supported in small neighborhoods around $\xi = \pm \sqrt{3}$ and such that $\psi(\xi) = 1$ for say $|\xi - (\pm \sqrt{3})| \leq 10^{-2}$. Then we write 
\begin{equation} \label{equ:decomposition_termIV_for_local_decay}
\begin{aligned}
 IV &= \frac{1}{2i} \int_1^t e^{i(t-s)\jtD} \jtD^{-1} P_c \bigl( \alpha(\cdot) (w(s) + \bar{w}(s))^2 \bigr) \, \ud s \\
 &= \frac{1}{2i} \int_1^t \Bigl( e^{i(t-s)\jtD} \jtD^{-1} P_c \bigl( \alpha \varphi^2 \bigr) \Bigr) (a(s) + \bar{a}(s))^2 \, \ud s \\
 &= \frac{1}{2i} \int_1^t \Bigl( \psi(\tD) e^{i(t-s)\jtD} \jtD^{-1} P_c \bigl( \alpha \varphi^2 \bigr) \Bigr) (a(s) + \bar{a}(s))^2 \, \ud s \\
 &\quad + \frac{1}{2i} \int_1^t \Bigl( \bigl( 1 - \psi(\tD) \bigr) e^{i(t-s)\jtD} \jtD^{-1} P_c \bigl( \alpha \varphi^2 \bigr) \Bigr) (a(s) + \bar{a}(s))^2 \, \ud s \\
 &\equiv IV_{(a)} + IV_{(b)}.
\end{aligned}
\end{equation}
The first term on the right-hand side of~\eqref{equ:decomposition_termIV_for_local_decay} can be easily estimated using the improved local decay estimate~\eqref{equ:local_decay_away_zero_freq_tDinv_est} for the Klein-Gordon propagator away from zero energy and the bound~\eqref{equ:a_decay_bound}. Uniformly for all $1 \leq t \leq T$ we obtain that
\begin{align*}
 \|\jx^{-\sigma} IV_{(a)}\|_{L^2_x} &\lesssim \int_1^t \Bigl\| \jx^{-\sigma} \Bigl( \psi(\tD) e^{i(t-s)\jtD} \jtD^{-1} P_c \bigl( \alpha \varphi^2 \bigr) \Bigr) \Bigr\|_{L^2_x} |a(s)|^2 \, \ud s \\
 &\lesssim \int_1^t \frac{1}{\jap{t-s}^\thf} \bigl\| \jx^\sigma \alpha \varphi^2 \bigr\|_{L^2_x} \frac{\bigl( \log(1+\js) \bigr)^2}{\js} \bigl( \|\jx^\sigma v_0\|_{L^2_x} + M(T)^2 \bigr)^2 \, \ud s \\
 &\lesssim \frac{1}{\jt^{1-}} \bigl( \|\jx^\sigma v_0\|_{L^2_x} + M(T)^2 \bigr)^2.
\end{align*}
To estimate the second term on the right-hand side of~\eqref{equ:decomposition_termIV_for_local_decay}, we further expand it as 
\begin{equation}
 \begin{aligned}
  IV_{(b)} &= \frac{1}{2i} \int_1^t \Bigl( \bigl( 1 - \psi(\tD) \bigr) e^{i(t-s)\jtD} \jtD^{-1} P_c \bigl( \alpha \varphi^2 \bigr) \Bigr) e^{2is} \bigl( e^{-is} a(s) \bigr)^2 \, \ud s \\ 
  &\quad + \frac{1}{i} \int_1^t \Bigl( \bigl( 1 - \psi(\tD) \bigr) e^{i(t-s)\jtD} \jtD^{-1} P_c \bigl( \alpha \varphi^2 \bigr) \Bigr) \bigl| e^{-is} a(s) \bigr|^2 \, \ud s \\ 
  &\quad + \frac{1}{2i} \int_1^t \Bigl( \bigl( 1 - \psi(\tD) \bigr) e^{i(t-s)\jtD} \jtD^{-1} P_c \bigl( \alpha \varphi^2 \bigr) \Bigr) e^{-2is} \bigl( \overline{ e^{-is} a(s) } \bigr)^2 \, \ud s \\
  &\equiv IV_{(b)}^{(1)} + IV_{(b)}^{(2)} + IV_{(b)}^{(3)}.
 \end{aligned}
\end{equation}
Then we integrate by parts in time $s$. Note that for the first term $IV_{(b)}^{(1)}$ this could potentially be problematic, because on the distorted Fourier side the phase of $e^{is(2-\jap{\xi})}$ vanishes at frequencies $\xi = \pm \sqrt{3}$. However, owing to the cutoff $(1-\psi(\xi))$ the integrand is zero in a neighborhood of $\xi = \pm \sqrt{3}$. We only provide the details for the treatment of the term $IV_{(b)}^{(1)}$, the other terms being easier. We find that uniformly for all $1 \leq t \leq T$ one has that
\begin{align*}
 IV_{(b)}^{(1)} &= -\frac{1}{2} \int_1^t e^{it\jtD} \Bigl( \frac{\ud}{\ud s} \Bigl( e^{is(2-\jtD)} \Bigr) (2-\jtD)^{-1} \jtD^{-1} \bigl( 1 - \psi(\tD) \bigr) P_c \bigl( \alpha \varphi^2 \bigr) \Bigr) \bigl( e^{-is} a(s) \bigr)^2 \, \ud s \\
 &= -\frac12 \Bigl( (2-\jtD)^{-1} \jtD^{-1} \bigl( 1 - \psi(\tD) \bigr) P_c \bigl( \alpha \varphi^2 \bigr) \Bigr) a(t)^2 \\
 &\quad + \frac12 e^{i(t-1)\jtD} \Bigl( (2-\jtD)^{-1} \jtD^{-1} \bigl( 1 - \psi(\tD) \bigr) P_c \bigl( \alpha \varphi^2 \bigr) \Bigr) a(1)^2 \\
 &\quad + \int_1^t \Bigl( e^{i(t-s) \jtD} (2-\jtD)^{-1} \jtD^{-1} \bigl( 1 - \psi(\tD) \bigr) P_c \bigl( \alpha \varphi^2 \bigr) \Bigr) \bigl( e^{-is} a(s) \bigr) \ps \bigl( e^{-is} a(s) \bigr) \, \ud s.
\end{align*}
Then using the standard local decay estimate~\eqref{equ:local_decay_est}, the bounds~\eqref{equ:a_decay_bound} and~\eqref{equ:a_pt_decay_bound} for the coefficient~$a(t)$, and Lemma~\ref{lem:bound_sqrt3_vanishes}, we finally obtain that
\begin{align*}
 &\| \jx^{-\sigma} IV_{(b)}^{(1)} \|_{L^2_x} \\
 &\lesssim \Bigl\| \jx^{-\sigma} \Bigl( (2-\jtD)^{-1} \jtD^{-1} \bigl( 1 - \psi(\tD) \bigr) P_c \bigl( \alpha \varphi^2 \bigr) \Bigr) \Bigr\|_{L^2_x} |a(t)|^2 \\
 &\quad + \Bigl\| \jx^{-\sigma} e^{i(t-1)\jtD} \Bigl( (2-\jtD)^{-1} \jtD^{-1} \bigl( 1 - \psi(\tD) \bigr) P_c \bigl( \alpha \varphi^2 \bigr) \Bigr) \Bigr\|_{L^2_x} |a(1)|^2 \\
 &\quad + \int_1^t \Bigl\| \jx^{-\sigma} \Bigl( e^{i(t-s) \jtD} (2-\jtD)^{-1} \jtD^{-1} \bigl( 1 - \psi(\tD) \bigr) P_c \bigl( \alpha \varphi^2 \bigr) \Bigr) \Bigr\|_{L^2_x} \bigl| e^{-is} a(s) \bigr| \bigl| \ps \bigl( e^{-is} a(s) \bigr) \bigr| \, \ud s \\
 &\lesssim \Bigl\| \Bigl( (2-\jtD)^{-1} \jtD^{-1} \bigl( 1 - \psi(\tD) \bigr) P_c \bigl( \alpha \varphi^2 \bigr) \Bigr) \Bigr\|_{L^2_x} \frac{\bigl( \log(1+\jt) \bigr)^2}{\jt} \bigl( \|\jx^\sigma v_0\|_{L^2_x} + M(T)^2 \bigr)^2 \\
 &\quad + \frac{1}{\jt^\hf} \Bigl\| \jx^\sigma \Bigl( (2-\jtD)^{-1} \jtD^{-1} \bigl( 1 - \psi(\tD) \bigr) P_c \bigl( \alpha \varphi^2 \bigr) \Bigr) \Bigr\|_{L^2_x} \bigl( \|\jx^\sigma v_0\|_{L^2_x} + M(T)^2 \bigr)^2 \\
 &\quad + \int_1^t \frac{1}{\jap{t-s}^\hf} \, \Bigl\| \jx^\sigma (2-\jtD)^{-1} \jtD^{-1} \bigl( 1 - \psi(\tD) \bigr) P_c \bigl( \alpha \varphi^2 \bigr) \Bigr) \Bigr\|_{L^2_x} \times \\ 
 &\quad \quad \qquad \qquad \qquad \qquad \qquad \qquad \qquad \qquad \qquad \qquad \times \frac{\log(1+\js)}{\js^\thf} \bigl( \|\jx^\sigma v_0\|_{L^2_x} + M(T)^2 \bigr)^2 \, \ud s \\
 &\lesssim \frac{1}{\jt^\hf} \bigl\| \jx^{\sigma+3} \alpha \varphi^2 \bigr\|_{L^2_x} \bigl( \|\jx^\sigma v_0\|_{L^2_x} + M(T)^2 \bigr)^2.
\end{align*}

Putting all of the preceding estimates together (and using that we may freely assume that $M(T) \leq 1$), we arrive at the estimate
\begin{equation*}
 M(T) \lesssim \|\jx^\sigma v_0\|_{H^2_x} + M(T)^2.
\end{equation*}
The assertion of Proposition~\ref{prop:local_decay_bounds} now follows by a standard continuity argument.
\end{proof}

A key step in the derivation of the local decay bounds for the solution $v(t)$ in Proposition~\ref{prop:local_decay_bounds} was to isolate the leading order behavior of the variable coefficient quadratic nonlinearity $\alpha(x) (v+\bar{v})^2$. It is determined by $\alpha(x) (w + \bar{w})^2$, where $w(t)$ is defined in~\eqref{equ:definition_w} 
and where we write $w(t,x) = a(t) \varphi(x)$, $t \geq 1$,
with the time-dependent coefficient $a(t)$ defined in~\eqref{equ:definition_a}.
From the local decay bounds established in Proposition~\ref{prop:local_decay_bounds}, we infer two improved local decay bounds for the difference $\chi_0(H) v(t) - w(t)$ for $t \geq 1$. These will be needed later in the proof of Theorem~\ref{thm:thm1}.
\begin{corollary} \label{cor:local_decay_improved_low_energy_v_minus_w}
 Under the assumptions of Proposition~\ref{prop:local_decay_bounds}, we have  
 \begin{align}
  \bigl\| \jx^{-\sigma} \bigl( \chi_0(H) v(t) - w(t) \bigr) \bigr\|_{L^2_x} &\lesssim \frac{\varepsilon}{\jt}, \quad t \geq 1, \label{equ:local_decay_bound_low_energy_v_minus_w} \\
  \bigl\| \jx^{-\sigma} \partial_x \bigl( \chi_0(H) v(t) - w(t) \bigr) \bigr\|_{L^2_x} &\lesssim \frac{\varepsilon}{\jt}, \quad t \geq 1. \label{equ:local_decay_bound_low_energy_px_v_minus_w}
 \end{align}
\end{corollary}
\begin{proof}
 We present the details for the derivation of the second asserted local decay bound~\eqref{equ:local_decay_bound_low_energy_px_v_minus_w}. We proceed similarly as in the proof of the preceding Proposition~\ref{prop:local_decay_bounds}.
 By Duhamel's formula for the solution $v(t)$ and the definition~\eqref{equ:definition_w} of $w(t)$, we have for $t \geq 1$ 
 \begin{align*}
  &\partial_x \bigl( \chi_0(H) v(t) - w(t) \bigr) \\
  &\quad = \partial_x \biggl( e^{it\jtD} \chi_0(H) P_c v_0 - c_0 \frac{e^{i\frac{\pi}{4}} e^{it}}{t^\hf} \langle \varphi, v_0 \rangle \varphi \biggr) \\
  &\quad \quad + \frac{1}{2i} \int_0^{t-1} \partial_x \biggl( e^{i(t-s)\jtD} \chi_0(H) \jtD^{-1} P_c \bigl( \alpha(\cdot) u(s)^2 \bigr) - c_0 \frac{e^{i\frac{\pi}{4}} e^{i(t-s)}}{(t-s)^{\hf}} \bigl\langle \varphi, \alpha(\cdot) u(s)^2 \bigr\rangle \varphi \biggr) \, \ud s \\
  &\quad \quad + \frac{1}{2i} \int_{t-1}^t \partial_x \biggl( e^{i(t-s)\jtD} \chi_0(H) \jtD^{-1} P_c \bigl( \alpha(\cdot) u(s)^2 \bigr) \biggr) \, \ud s.
 \end{align*}
 Using the local decay estimates~\eqref{equ:local_decay_improved_subtr_off_with_derivative_in_cor} and~\eqref{equ:local_decay_improved_tDinv_subtr_off_with_derivative} for the Klein-Gordon propagator along with the local decay bounds for $v(t)$ from Proposition~\ref{prop:local_decay_bounds} and the Sobolev estimate~\eqref{equ:Linfty_v_local_decay_bound}, we obtain for $t \geq 1$ 
 \begin{align*}
  &\bigl\|\jx^{-\sigma} \partial_x \bigl( \chi_0(H) v(t) - w(t) \bigr) \bigr\|_{L^2_x} \\
  &\quad \lesssim \frac{\|\jx^\sigma v_0\|_{L^2_x}}{\jt^\thf} + \int_0^{t-1} \frac{1}{\jap{t-s}^\thf} \bigl\| \jx^\sigma \alpha(x) u(s)^2 \bigr\|_{L^2_x} \, \ud s + \int_{t-1}^t \bigl\| \alpha(x) u(s)^2 \bigr\|_{L^2_x} \, \ud s \\
  &\quad \lesssim \frac{\|\jx^\sigma v_0\|_{L^2_x}}{\jt^\thf} + \int_0^{t-1} \frac{1}{\jap{t-s}^\thf} \|\jx^{3 \sigma} \alpha\|_{L^2_x} \|\jx^{-\sigma} v(s)\|_{L^\infty_x}^2 \, \ud s \\
  &\quad \quad + \int_{t-1}^t  \|\jx^{2 \sigma} \alpha\|_{L^2_x} \|\jx^{-\sigma} v(s)\|_{L^\infty_x}^2  \, \ud s \\
  &\quad \lesssim \frac{\varepsilon}{\jt^\thf} + \int_0^{t-1} \frac{1}{\jap{t-s}^\thf} \frac{\varepsilon^2}{\js} \, \ud s + \int_{t-1}^t \frac{\varepsilon^2}{\js}  \, \ud s \\
  &\quad \lesssim \frac{\varepsilon}{\jt},
 \end{align*}
 as desired. 
 The proof of the first asserted local decay bound~\eqref{equ:local_decay_bound_low_energy_v_minus_w} proceeds similarly, using the local decay estimates \eqref{equ:local_decay_away_zero_freq_est}, \eqref{equ:local_decay_improved_subtr_off}, \eqref{equ:local_decay_away_zero_freq_tDinv_est}, and \eqref{equ:local_decay_improved_tDinv_subtr_off}.
\end{proof}

As a further corollary of the local decay bounds for $v(t)$ from Proposition~\ref{prop:local_decay_bounds}, we deduce the asymptotics of the coefficient function $a(t)$. 
\begin{corollary}(Asymptotics of $a(t)$) \label{cor:asymptotics_a}
 Under the assumptions of Proposition~\ref{prop:local_decay_bounds}, the coefficient
 \begin{equation*}
  a(t) := c_0 \frac{e^{i\frac{\pi}{4}} e^{it}}{t^\hf} \langle \varphi, v_0 \rangle + \frac{1}{2i} \int_0^{t-1} c_0 \frac{e^{i\frac{\pi}{4}} e^{i(t-s)}}{(t-s)^{\hf}} \bigl\langle \varphi, \alpha(\cdot) u(s)^2 \bigr\rangle \, \ud s, \quad t \geq 1,
 \end{equation*}
 has the asymptotics
 \begin{equation} \label{equ:asymptotics_a}
  a(t) = c_0 \frac{e^{i \frac{\pi}{4}} e^{it}}{t^\hf} a_0 + \calO_{L^\infty_t} \Bigl( \frac{\varepsilon^2}{t} \Bigr), \quad t \geq 1,
 \end{equation}
 where 
 \begin{equation} \label{equ:amplitude_a0}
  \begin{aligned}
  a_0 = \langle \varphi, v_0 \rangle &+ \frac{1}{2} \langle \varphi, \alpha(\cdot) v_0^2 \rangle - \langle \varphi, \alpha(\cdot) |v_0|^2 \rangle - \frac{1}{6} \langle \varphi, \alpha(\cdot) \overline{v}_0^2 \rangle \\
  &+ \int_0^\infty e^{is} \langle \varphi, \alpha(\cdot) \partial_s \bigl( e^{-is} v(s) \bigr) \bigl( e^{-is} v(s) \bigr) \rangle \, \ud s \\
  &- \int_0^\infty e^{-is} \langle \varphi, \alpha(\cdot) \partial_s \bigl( (e^{-is} v(s)) (e^{is} \bar{v}(s) ) \bigr) \rangle \, \ud s \\
  &- \frac{1}{3} \int_0^\infty e^{-3is} \langle \varphi, \alpha(\cdot) \partial_s \bigl( e^{is} \bar{v}(s) \bigr) \bigl( e^{is} \bar{v}(s) \bigr) \rangle \, \ud s.
  \end{aligned}
 \end{equation}
\end{corollary}
\begin{proof}
 We proceed similarly to the proof of Proposition~3.2 in \cite{LLS2}. 
 We begin by writing
 \begin{align*}
  a(t) = c_0 \frac{e^{i \frac{\pi}{4}} e^{it}}{t^\hf} \biggl( \langle \varphi, v_0 \rangle + \frac{1}{2i} \int_0^{t-1} \frac{t^\hf}{(t-s)^\hf} e^{-is} \langle \varphi, \alpha(\cdot) u(s)^2 \rangle \, \ud s \biggr).
 \end{align*}
 Then the main work goes into peeling off the leading order behavior of the second term in the parentheses. To this end we insert the decomposition of $u(s)$ into its ``phase-filtered components''
 \begin{equation*}
  u(s) = e^{is} ( e^{-is} v(s) ) + e^{-is} ( e^{is} \bar{v}(s) )
 \end{equation*}
 to find that
 \begin{align*}
  \frac{1}{2i} \int_0^{t-1} \frac{t^\hf}{(t-s)^\hf} e^{-is} \langle \varphi, \alpha(\cdot) u(s)^2 \rangle \, \ud s &= \frac{1}{2i} \int_0^{t-1} \frac{t^\hf}{(t-s)^\hf} e^{is} \langle \varphi, \alpha(\cdot) \bigl( e^{-is} v(s) \bigr)^2 \rangle \, \ud s \\
  &\quad + \frac{2}{2i} \int_0^{t-1} \frac{t^\hf}{(t-s)^\hf} e^{-is} \langle \varphi, \alpha(\cdot) \bigl| e^{-is} v(s) \bigr|^2 \rangle \, \ud s \\
  &\quad + \frac{1}{2i} \int_0^{t-1} \frac{t^\hf}{(t-s)^\hf} e^{-3is} \langle \varphi, \alpha(\cdot) \bigl( e^{is} \bar{v}(s) \bigr)^2 \rangle \, \ud s \\
  &\equiv \frac{1}{2i} \bigl( I + II + III \bigr).
 \end{align*}
 Now we describe in detail how to peel off the leading order behavior of the term $I$, noting that the other terms can be treated analogously. We first exploit the oscillations and integrate by parts in time $s$ to find that 
 \begin{equation*}
  \begin{aligned}
   I &= \int_0^{t-1} \frac{t^\hf}{(t-s)^\hf} e^{is} \langle \varphi, \alpha(\cdot) \bigl( e^{-is} v(s) \bigr)^2 \rangle \, \ud s \\
   &= -i t^\hf e^{i(t-1)} \langle \varphi, \alpha(\cdot) \bigl( e^{-i (t-1)} v(t-1) \bigr)^2 \rangle \\
   &\quad + i \langle \varphi, \alpha(\cdot) v(0)^2 \rangle \\
   &\quad + \frac{i}{2} \int_0^{t-1} \frac{t^\hf}{(t-s)^\thf} e^{is} \langle \varphi, \alpha(\cdot) \bigl( e^{-is} v(s) \bigr)^2 \rangle \, \ud s \\
   &\quad + 2i \int_0^{t-1} \frac{t^\hf}{(t-s)^\hf} e^{is} \langle \varphi, \alpha(\cdot) \partial_s \bigl( e^{-is} v(s) \bigr) \bigl( e^{-is} v(s) \bigr) \rangle \, \ud s \\
   &\equiv I_{(a)} + I_{(b)} + I_{(c)} + I_{(d)}.
  \end{aligned}
 \end{equation*}
 Clearly, the term $I_{(b)}$ contributes to the leading order behavior of $I$. We now show that the terms $I_{(a)}$ and $I_{(c)}$ decay as $t \to \infty$, and we extract the leading order contribution from the term $I_{(d)}$. Using the local decay bounds for $v(t)$ from Proposition~\ref{prop:local_decay_bounds}, we obtain that 
 \begin{equation*}
  |I_{(a)}| \lesssim t^\hf \| \jx^{2\sigma} \alpha(x) \|_{L^\infty_x} \|\jx^{-\sigma} v(t-1)\|_{L^2_x}^2 \lesssim \frac{\varepsilon^2}{\jt^\hf}
 \end{equation*}
 and 
 \begin{align*}
  |I_{(c)}| &\lesssim t^\hf \int_0^{t-1} \frac{1}{(t-s)^\thf} \|\jx^{2\sigma} \alpha(x)\|_{L^\infty_x} \|\jx^{-\sigma} v(s)\|_{L^2_x}^2 \, \ud s \lesssim t^\hf \int_0^{t-1} \frac{1}{(t-s)^{\thf}} \frac{\varepsilon^2}{\js} \, \ud s \lesssim \frac{\varepsilon^2}{\jt^\hf}.
 \end{align*}
 Then we rewrite the last term $I_{(d)}$ as 
 \begin{align*}
  I_{(d)} &= 2i \int_0^{t-1} \frac{t^\hf}{(t-s)^\hf} e^{is} \langle \varphi, \alpha(\cdot) \partial_s \bigl( e^{-is} v(s) \bigr) \bigl( e^{-is} v(s) \bigr) \rangle \, \ud s \\
  &= 2i \int_0^\infty e^{is} \langle \varphi, \alpha(\cdot) \partial_s \bigl( e^{-is} v(s) \bigr) \bigl( e^{-is} v(s) \bigr) \rangle \, \ud s \\
  &\quad - 2i \int_{\frac{t}{2}}^\infty e^{is} \langle \varphi, \alpha(\cdot) \partial_s \bigl( e^{-is} v(s) \bigr) \bigl( e^{-is} v(s) \bigr) \rangle \, \ud s \\
  &\quad + 2i \int_{\frac{t}{2}}^{t-1} \frac{t^\hf}{(t-s)^\hf} e^{is} \langle \varphi, \alpha(\cdot) \partial_s \bigl( e^{-is} v(s) \bigr) \bigl( e^{-is} v(s) \bigr) \rangle \, \ud s \\
  &\quad + 2i \int_0^{\frac{t}{2}} \biggl( \frac{t^\hf}{(t-s)^\hf} - 1 \biggr) e^{is} \langle \varphi, \alpha(\cdot) \partial_s \bigl( e^{-is} v(s) \bigr) \bigl( e^{-is} v(s) \bigr) \rangle \, \ud s \\
  &\equiv I_{(d)}^{(1)} + I_{(d)}^{(2)} + I_{(d)}^{(3)} + I_{(d)}^{(4)}.
 \end{align*}
 Using the local decay bounds from Proposition~\ref{prop:local_decay_bounds}, in particular that $\|\jx^{-\sigma} \pt ( e^{-it} v(t) ) \|_{L^2_x}$ has faster decay, it is easy to see that the improper integral $I_{(d)}^{(1)}$ converges and contributes to the leading order behavior of the term $I$, while the other terms $I_{(d)}^{(2)}$, $I_{(d)}^{(3)}$, and $I_{(d)}^{(4)}$ are of the order $\calO_{L^\infty_t}( \varepsilon^2 \jt^{-\frac12} )$. Indeed, we find that
 \begin{equation*}
  |I_{(d)}^{(1)}| \lesssim \int_0^\infty \| \jx^{2\sigma} \alpha \|_{L^\infty_x} \| \jx^{-\sigma} \partial_s ( e^{-is} v(s) ) \|_{L^2_x} \| \jx^{-\sigma} v(s) \|_{L^2_x} \, \ud s \lesssim \int_0^\infty \frac{\varepsilon^2}{\js^\thf} \, \ud s \lesssim \varepsilon^2
 \end{equation*}
 and 
 \begin{align*}
  |I_{(d)}^{(2)}| &\lesssim \int_{\frac{t}{2}}^\infty \|\jx^{2\sigma} \alpha\|_{L^\infty_x} \|\jx^{-\sigma} \partial_s (e^{-is} v(s))\|_{L^2_x} \| \jx^{-\sigma} v(s) \|_{L^2_x} \, \ud s \lesssim \int_{\frac{t}{2}}^{\infty} \frac{\varepsilon^2}{\js^\thf} \, \ud s \lesssim \frac{\varepsilon^2}{\jt^\hf}.
 \end{align*}
 Analogously, we obtain that
 \begin{align*}
  |I_{(d)}^{(3)}| &\lesssim \int_{\frac{t}{2}}^{t-1} \frac{t^\hf}{(t-s)^\hf} \frac{\varepsilon^2}{\js^\thf} \, \ud s \lesssim \frac{\varepsilon^2}{\jt^\hf}, \\
  |I_{(d)}^{(4)}| &\lesssim \int_0^{\frac{t}{2}} \frac{s}{(t-s)^\hf (t^\hf + (t-s)^\hf)} \frac{\varepsilon^2}{\js^\thf} \, \ud s \lesssim \frac{\varepsilon^2}{\jt^\hf}.
 \end{align*}
 Thus, the leading order behavior of the term $I$ is given by
 \begin{equation*}
  I = i \langle \varphi, \alpha(\cdot) v_0^2 \rangle + 2i \int_0^\infty e^{is} \langle \varphi, \alpha(\cdot) \partial_s \bigl( e^{-is} v(s) \bigr) \bigl( e^{-is} v(s) \bigr) \rangle \, \ud s + \calO_{L^\infty_t} \biggl( \frac{\varepsilon^2}{\jt^\hf} \biggr). 
 \end{equation*}
 Similarly, we compute that the leading order behaviors of the terms $II$ and $III$ are given by 
 \begin{align*}
  II &= -2i \langle \varphi, \alpha(\cdot) |v_0|^2 \rangle - 2i \int_0^\infty e^{-is} \langle \varphi, \alpha(\cdot) \partial_s \bigl( (e^{-is} v(s)) (e^{+is} \bar{v}(s) ) \bigr) \rangle \, \ud s + \calO_{L^\infty_t} \biggl( \frac{\varepsilon^2}{\jt^\hf} \biggr), \\
  III &= -\frac{i}{3} \langle \varphi, \alpha(\cdot) \bar{v}_0^2 \rangle - \frac{2i}{3} \int_0^\infty e^{-3is} \langle \varphi, \alpha(\cdot) \partial_s \bigl( e^{+is} \bar{v}(s) \bigr) \bigl( e^{+is} \bar{v}(s) \bigr) \rangle \, \ud s + \calO_{L^\infty_t} \biggl( \frac{\varepsilon^2}{\jt^\hf} \biggr).
 \end{align*}
 Putting things together, we conclude that the asymptotic behavior of the coefficient $a(t)$ is given by~\eqref{equ:asymptotics_a}. This concludes the proof.
\end{proof}

\section{Proof of Theorem~\ref{thm:thm1}} \label{sec:proof_theorem}

Now we are in the position to provide the proof of Theorem~\ref{thm:thm1}. We first consider the non-resonant case. Afterwards we turn to the treatment of the more delicate resonant case. In the course of the proof we will frequently invoke the local decay bounds established in Proposition~\ref{prop:local_decay_bounds} and Corollary~\ref{cor:local_decay_improved_low_energy_v_minus_w}. Throughout we let $\sigma = 5$.

\medskip 

\noindent \underline{{\it Non-Resonant Case}:}
 We begin with the proof of the decay estimate~\eqref{equ:thm1_Linfty_decay_nonresonant}. 
 By time-reversal symmetry, it suffices to consider positive times $t > 0$.
 For short times $0 < t \leq 1$ we just use the Sobolev estimate from Lemma~\ref{lem:weighted_sobolev} together with the local decay bounds~\eqref{equ:local_decay_bounds_for_v}. From Duhamel's formula we obtain that
 \begin{align*}
  \sup_{0 \leq t \leq 1} \, \|v(t)\|_{L^\infty_x} &\lesssim \sup_{0 \leq t \leq 1} \, \bigl( \|v(t)\|_{L^2_x} + \|\sqrt{H} v(t)\|_{L^2_x} \bigr) \\
  &\lesssim \|P_c v_0\|_{L^2_x} + \|\sqrt{H} P_c v_0\|_{L^2_x} + \int_0^1 \| \alpha(x) v(s)^2 \|_{L^2_x} \, \ud s \\
  &\lesssim \|v_0\|_{H^1_x} + \int_0^1 \|\jx^{2\sigma} \alpha\|_{L^2_x} \|\jx^{-\sigma} v(s)\|_{L^\infty_x}^2 \, \ud s \\
  &\lesssim \varepsilon.
 \end{align*} 
 Then for times $t \geq 1$, we first peel off the leading order behavior of the variable coefficient quadratic nonlinearity in Duhamel's formula for $v(t)$ by inserting the function $w(t)$ defined in~\eqref{equ:definition_w} as well as the asymptotics for the coefficient $a(t)$ from Corollary~\ref{cor:asymptotics_a}. Exploiting the non-resonance assumption $\wtcalF[\alpha \varphi^2](\pm \sqrt{3}) = 0$, we may then integrate by parts in time in the leading order term in Duhamel's formula to recast it into a more favorable form. Subsequently, we apply the dispersive decay estimate~\eqref{equ:dispersive_decay_propagator} for the Klein-Gordon propagator to infer the decay estimate~\eqref{equ:thm1_Linfty_decay_nonresonant}. 
 
More specifically, we begin by writing for any $t \geq 1$,
\begin{equation} \label{equ:decomposition_v_for_Linfty_decay_nonresonant}
 \begin{aligned}
  v(t) &= e^{it\jtD} P_c v_0 + \frac{1}{2i} \int_0^t e^{i(t-s)\jtD} \jtD^{-1} P_c \bigl( \alpha(\cdot) u(s)^2 \bigr) \, \ud s \\
  &= e^{it\jtD} P_c v_0 + \frac{1}{2i} \int_0^1 e^{i(t-s)\jtD} \jtD^{-1} P_c \bigl( \alpha(\cdot) u(s)^2 \bigr) \, \ud s \\
  &\quad + \frac{1}{2i} \int_1^t e^{i(t-s)\jtD} \jtD^{-1} P_c \Bigl( \alpha(\cdot) \bigl( (v(s) + \bar{v}(s))^2 - (w(s) + \bar{w}(s))^2 \bigr) \Bigr) \, \ud s \\
  &\quad + \frac{1}{2i} \int_1^t \Bigl( e^{i(t-s)\jtD} \jtD^{-1} P_c \bigl( \alpha \varphi^2 \bigr) \Bigr) (a(s) + \bar{a}(s))^2 \, \ud s \\
  &\equiv I + II+ III + IV.
 \end{aligned}
\end{equation}
The $L^\infty_x$ bounds for the terms $I$ and $II$ in~\eqref{equ:decomposition_v_for_Linfty_decay_nonresonant} are straightforward and we omit the details.
We now consider the term $III$, afterwards we estimate the delicate term $IV$.
In the case of the term~$III$ in~\eqref{equ:decomposition_v_for_Linfty_decay_nonresonant}, we have by the dispersive decay estimate~\eqref{equ:dispersive_decay_propagator} (with $\mu =\frac12$) for all $t \geq 1$ that
\begin{align*}
 \|III\|_{L^\infty_x} &\lesssim \int_1^t \, \Bigl\| e^{i(t-s)\jtD} \jtD^{-1} P_c \Bigl( \alpha(\cdot) \bigl( (v(s) + \bar{v}(s))^2 - (w(s) + \bar{w}(s))^2 \bigr) \Bigr) \Bigr\|_{L^\infty_x} \, \ud s \\
 &\lesssim \int_1^t \frac{1}{(t-s)^\hf} \Bigl\| \jtD P_c \Bigl( \alpha(\cdot) \bigl( (v(s) + \bar{v}(s))^2 - (w(s) + \bar{w}(s))^2 \bigr) \Bigr) \Bigr\|_{L^1_x} \, \ud s. 
\end{align*}
We now show that
\begin{equation*}
 \Bigl\| \jtD P_c \Bigl( \alpha(\cdot) \bigl( (v(s) + \bar{v}(s))^2 - (w(s) + \bar{w}(s))^2 \bigr) \Bigr) \Bigr\|_{L^1_x} \lesssim \frac{\varepsilon^2}{\js^\thf}, \qquad s \geq 1,
\end{equation*}
which immediately implies the desired decay estimate 
\begin{equation*}
 \|III\|_{L^\infty_x} \lesssim \int_1^t \frac{1}{(t-s)^\hf} \frac{\varepsilon^2}{\js^\thf} \, \ud s \lesssim \frac{\varepsilon^2}{t^\hf}, \quad t \geq 1.
\end{equation*}
We may ignore the complex conjugates to simplify the notation and now prove that 
\begin{equation}
 \Bigl\| \jtD P_c \Bigl( \alpha(\cdot) \bigl( v(s)^2 - w(s)^2 \bigr) \Bigr) \Bigr\|_{L^1_x} \lesssim \frac{\varepsilon^2}{\js^\thf}, \qquad s \geq 1.
\end{equation}
To this end we first further decompose $v(s)$ into a low-energy and a high-energy part 
\begin{equation*}
 v(s) = \chi_0(H) v(s) + \bigl( 1 - \chi_0(H) \bigr) v(s),
\end{equation*}
where we recall that $\chi_0$ denotes a smooth bump function supported on $|\xi| \lesssim 1$ with $\chi_0(\xi) = 1$ near $\xi = 0$. We obtain 
\begin{equation} \label{equ:decomposition_nonresonant_necessary_sobolev}
 \begin{aligned}
  \jtD P_c \Bigl( \alpha(\cdot) \bigl( v(s)^2 - w(s)^2 \bigr) \Bigr) &= \jtD P_c \Bigl( \alpha(\cdot) \bigl( ( \chi_0(H) v(s) )^2 - w(s)^2 \bigr) \Bigr) \\
  &\quad + \jtD P_c \Bigl( \alpha(\cdot) \bigl(\chi_0(H) v(s)\bigr) \bigl( (1-\chi_0(H)) v(s) \bigr) \Bigr) \\
  &\quad + \jtD P_c \Bigl( \alpha(\cdot) \bigl( (1-\chi_0(H)) v(s) \bigr) v(s) \Bigr) \\
  &\equiv III_{(a)} + III_{(b)} + III_{(c)}.
 \end{aligned}
\end{equation}
By exploiting the faster local decay of the high-energy component $(1-\chi_0(H)) v$ of the solution as well as the faster local decay of $\sqrt{H} v$ established in Proposition~\ref{prop:local_decay_bounds}, the product estimate~\eqref{equ:prod_est1} already yields the desired bound for the last two terms on the right-hand side of~\eqref{equ:decomposition_nonresonant_necessary_sobolev}, 
\begin{equation*}
 \begin{aligned}
  \| III_{(b)} \|_{L^1_x} + \| III_{(c)} \|_{L^1_x} &\lesssim \|\jx^{1+2\sigma} \alpha\|_{L^\infty_x} \bigl( \|\jx^{-\sigma} v(s)\|_{L^2_x} + \|\jx^{-\sigma} \sqrt{H} v(s)\|_{L^2_x} \bigr) \times \\
  &\qquad \qquad \qquad \qquad \times \bigl( \|\jx^{-\sigma} (1-\chi_0(H)) v(s)\|_{L^2_x} + \|\jx^{-\sigma} \sqrt{H} v(s)\|_{L^2_x} \bigr) \\
  &\lesssim \frac{\varepsilon^2}{\js^\thf}.
 \end{aligned}
\end{equation*}
It remains to estimate the more subtle first term on the right-hand side of~\eqref{equ:decomposition_nonresonant_necessary_sobolev}. By H\"older's inequality, the equivalence of norms from Lemma~\ref{lem:equiv_norms}, and the usual product rule for the derivative, we have 
\begin{equation*}
 \begin{aligned}
  \| III_{(a)} \|_{L^1_x} &\lesssim \Bigl\| \jx^\sigma \jtD P_c \Bigl( \alpha(\cdot) \bigl( ( \chi_0(H) v(s) )^2 - w(s)^2 \bigr) \Bigr) \Bigr\|_{L^2_x} \\
  &\lesssim \bigl\| \jx^\sigma \alpha(x) \bigl( ( \chi_0(H) v(s) )^2 - w(s)^2 \bigr) \bigr\|_{H^1_x} \\
  &\lesssim \| \jx^{1+2\sigma} \alpha \|_{W^{1,\infty}_x} \bigl\| \jx^{-\sigma} \bigl( \chi_0(H) v(s) - w(s) \bigr) \bigr\|_{L^2_x} \bigl\| \jx^{-\sigma} \bigl( \chi_0(H) v(s) + w(s) \bigr) \bigr\|_{L^\infty_x} \\
  &\quad + \| \jx^{1+2\sigma} \alpha \|_{L^\infty_x} \bigl\| \jx^{-\sigma} \px \bigl( \chi_0(H) v(s) - w(s) \bigr) \bigr\|_{L^2_x} \bigl\| \jx^{-\sigma} \bigl( \chi_0(H) v(s) + w(s) \bigr) \bigr\|_{L^\infty_x} \\
  &\quad + \| \jx^{1+2\sigma} \alpha \|_{L^\infty_x} \bigl\| \jx^{-\sigma} \bigl( \chi_0(H) v(s) - w(s) \bigr) \bigr\|_{L^2_x} \bigl\| \jx^{-\sigma} \px \bigl( \chi_0(H) v(s) + w(s) \bigr) \bigr\|_{L^\infty_x}.
 \end{aligned}
\end{equation*}
The kernel bounds~\eqref{equ:wtilD_chi0_kernel_bound} and the local decay bound for $v(s)$ imply 
\begin{align*}
 \bigl\| \jx^{-\sigma} \bigl( \chi_0(H) v(s) \bigr) \bigr\|_{L^\infty_x} + \bigl\| \jx^{-\sigma} \px \bigl( \chi_0(H) v(s) \bigr) \bigr\|_{L^\infty_x} \lesssim \| \jx^{-\sigma} v(s) \|_{L^2_x} \lesssim \frac{\varepsilon}{\js^\hf},
\end{align*}
while the asymptotics for the coefficient function $a(s)$ from Corollary~\ref{cor:asymptotics_a} give
\begin{equation*}
 \| \jx^{-\sigma} w(s) \|_{L^\infty_x} + \| \jx^{-\sigma} \px w(s) \|_{L^\infty_x} \lesssim |a(s)| \bigl( \|\varphi\|_{L^\infty_x} + \|\px \varphi\|_{L^\infty_x} \bigr) \lesssim \frac{\varepsilon}{\js^\hf}.
\end{equation*}
Combining the preceding estimates with the faster local decay for $( \chi_0(H) v(s) - w(s) )$ and for $\px ( \chi_0(H) v(s) - w(s) )$ from Corollary~\ref{cor:local_decay_improved_low_energy_v_minus_w} given by
\begin{equation*}
 \bigl\| \jx^{-\sigma} \bigl( \chi_0(H) v(s) - w(s) \bigr) \bigr\|_{L^2_x} + \bigl\| \jx^{-\sigma} \px \bigl( \chi_0(H) v(s) - w(s) \bigr) \bigr\|_{L^2_x} \lesssim \frac{\varepsilon}{\js},
\end{equation*}
we arrive at the desired bound $\| III_{(a)} \|_{L^1_x} \lesssim \varepsilon^{-2} \js^{-\thf}$ for $s \geq 1$.

Finally, we consider the delicate term $IV$ in the decomposition~\eqref{equ:decomposition_v_for_Linfty_decay_nonresonant} of Duhamel's formula for~$v(t)$. We further decompose it by inserting the asymptotics for the coefficient $a(t)$ from Corollary~\ref{cor:asymptotics_a} and find that
\begin{equation} \label{equ:decomposition_v_for_Linfty_decay_nonresonant_delicate_term}
 \begin{aligned}
  IV &= c_0^2 \frac{a_0^2}{2} \int_1^t \Bigl( e^{i(t-s)\jtD} \jtD^{-1} P_c \bigl( \alpha \varphi^2 \bigr) \Bigr) \frac{e^{2is}}{s} \, \ud s \\
  &\quad + c_0^2 \frac{|a_0|^2}{i} \int_1^t \Bigl( e^{i(t-s)\jtD} \jtD^{-1} P_c \bigl( \alpha \varphi^2 \bigr) \Bigr) \frac{1}{s} \, \ud s \\
  &\quad - c_0^2 \frac{\bar{a}_0^2}{2} \int_1^t \Bigl( e^{i(t-s)\jtD} \jtD^{-1} P_c \bigl( \alpha \varphi^2 \bigr) \Bigr) \frac{e^{-2is}}{s} \, \ud s \\
  &\quad + \frac{1}{2i}  \int_1^t \Bigl( e^{i(t-s)\jtD} \jtD^{-1} P_c \bigl( \alpha \varphi^2 \bigr) \Bigr) \calO_{L^\infty_s} \Bigl( \frac{\varepsilon^2}{s^\thf}\Bigr) \, \ud s \\
  &\equiv IV_{(a)} + IV_{(b)} + IV_{(c)} + IV_{(d)}.
 \end{aligned}
\end{equation}
We observe that the term $IV_{(a)}$ can be thought of to determine the leading order behavior of $v(t)$, because on the distorted Fourier side in the integrand of $IV_{(a)}$ the phase of $e^{is(2-\jxi)}$  vanishes when $2-\jap{\xi} = 0$, i.e. for $\xi = \pm \sqrt{3}$. In contrast, the integrands in the terms $IV_{(b)}$ and $IV_{(c)}$ have better oscillatory behavior in time (at all frequencies) and the term $IV_{(d)}$ has better decay in $s$ of the integrand anyway. However, thanks to the non-resonance assumption $\widetilde{\calF}[\alpha \varphi^2](\pm \sqrt{3}) = 0$, we can still integrate by parts in time~$s$ in the delicate term $IV_{(a)}$ and cast it into a better form. We find that
\begin{align*}
 IV_{(a)} &= c_0^2 \frac{a_0^2}{2i} (2-\jtD)^{-1} \jtD^{-1} P_c \bigl( \alpha \varphi^2 \bigr) \frac{e^{2it}}{t} - c_0^2 \frac{a_0^2}{2i} \Bigl( e^{i(t-1)\jtD} (2-\jtD)^{-1} \jtD^{-1} P_c \bigl( \alpha \varphi^2 \bigr) \Bigr) e^{2i} \\
 &\quad + c_0^2 \frac{a_0^2}{2i} \int_1^t \Bigl( e^{i(t-s)\jtD} (2-\jtD)^{-1} \jtD^{-1} P_c \bigl( \alpha \varphi^2 \bigr) \Bigr) \frac{1}{s^2} \, \ud s.
\end{align*}
At this point we can infer the desired decay estimate for $IV_{(a)}$. For times $1 \leq t \leq 2$ we just use the Sobolev estimate from Lemma~\ref{lem:weighted_sobolev}, while we invoke the dispersive decay estimate~\eqref{equ:dispersive_decay_propagator} (with $\mu=\frac12$) and Lemma~\ref{lem:bound_sqrt3_vanishes} to obtain uniformly for all times $t \geq 2$ that
\begin{align*}
 \|IV_{(a)}\|_{L^\infty_x} &\lesssim \frac{|a_0|^2}{t} \bigl\| (2-\jtD)^{-1} \jtD^{-1} P_c \bigl( \alpha \varphi^2 \bigr) \bigr\|_{L^\infty_x} \\
 &\quad + \frac{|a_0|^2}{(t-1)^\hf} \bigl\| (2-\jtD)^{-1} \jtD P_c \bigl( \alpha \varphi^2 \bigr) \bigr\|_{L^1_x} \\
 &\quad + |a_0|^2 \int_1^t \frac{1}{(t-s)^\hf}  \bigl\| (2-\jtD)^{-1} \jtD P_c \bigl( \alpha \varphi^2 \bigr) \bigr\|_{L^1_x} \frac{1}{s^2} \, \ud s \\
 &\lesssim \bigl\| \jx^{\sigma+3} \alpha \varphi^2 \bigr\|_{L^2_x} \frac{|a_0|^2}{t^\hf}.
\end{align*}
The terms $IV_{(b)}$ and $IV_{(c)}$ can be estimated analogously after integrating by parts in time $s$, and the term $IV_{(d)}$ can be bounded directly.
This finishes the proof of the decay estimate~\eqref{equ:thm1_Linfty_decay_nonresonant} in the non-resonant case.

In order to specify the asymptotic behavior of the solution $v(t)$, we define the scattering data
\begin{equation}
 \begin{aligned}
  v_\infty &:= P_c v_0 + \frac{1}{2i} \int_0^1 e^{-i s \jtD} \jtD^{-1} P_c \bigl( \alpha(\cdot) u(s)^2 \bigr) \, \ud s \\
  &\quad + \frac{1}{2i} \int_1^\infty e^{-is\jtD} \jtD^{-1} P_c \Bigl( \alpha(\cdot) \bigl( (v(s) + \bar{v}(s))^2 - (w(s) + \bar{w}(s))^2 \bigr) \Bigr) \, \ud s \\
  &\quad + \frac{1}{2i} \int_1^\infty \Bigl( e^{-i s \jtD} \jtD^{-1} P_c \bigl( \alpha \varphi^2 \bigr) \Bigr) (a(s) + \bar{a}(s))^2 \, \ud s.
 \end{aligned}
\end{equation}
Then by mimicking the preceding arguments, it follows that $v_\infty \in H^2_x$ and that $v(t)$ scatters in $H^2_x$ to a free Klein-Gordon wave in the sense that 
\begin{equation*}
 \bigl\| v(t) - e^{it\jtD} v_\infty \bigr\|_{H^2_x} \lesssim \frac{\varepsilon^2}{\jt^\hf}, \quad t \geq 1.
\end{equation*}
This concludes the treatment of the non-resonant case.

\medskip 

\noindent \underline{{\it Resonant Case}:}
We begin with the proof of the decay estimate~\eqref{equ:thm1_Linfty_decay_resonant}. Again, it suffices to consider positive times $t > 0$. For times $0 < t \leq 1$ we just use the Sobolev estimate from Lemma~\ref{lem:weighted_sobolev} together with the local decay bounds~\eqref{equ:local_decay_bounds_for_v}, as in the preceding treatment of the non-resonant case.
Then it remains to consider times $t \geq 1$.
To this end we combine the dispersive decay estimate~\eqref{equ:dispersive_decay_propagator} for the Klein-Gordon propagator (with $\mu = \frac12$) and the product estimate~\eqref{equ:prod_est1} with the local decay bounds~\eqref{equ:local_decay_bounds_for_v} for $v(t)$, to infer from Duhamel's formula for $v(t)$ that uniformly for all $t \geq 1$,
\begin{align*}
 &\|v(t)\|_{L^\infty_x} \\
 &\lesssim \bigl\| e^{it\jtD} P_c v_0 \bigr\|_{L^\infty_x} + \int_0^t \bigl\| e^{i(t-s)\jtD} \jtD^{-1} P_c \bigl( \alpha(\cdot) u(s)^2 \bigr) \bigr\|_{L^\infty_x} \, \ud s \\
 &\lesssim \frac{\| \jtD^2 P_c v_0 \|_{L^1_x}}{t^\hf} + \int_0^t \frac{1}{(t-s)^\hf} \bigl\| \jtD P_c \bigl( \alpha(\cdot) u(s)^2 \bigr) \bigr\|_{L^1_x} \, \ud s \\
 &\lesssim \frac{\| \jx^\sigma \jtD^2 P_c v_0 \|_{L^2_x}}{t^\hf} + \int_0^t \frac{1}{(t-s)^\hf} \|\jx^{1+2\sigma}\alpha\|_{W^{1,\infty}_x} \bigl( \|\jx^{-\sigma} v(s)\|_{L^2_x} + \|\jx^{-\sigma} \sqrt{H} v(s)\|_{L^2_x} \bigr)^2 \, \ud s \\
 &\lesssim \frac{\| \jx^\sigma v_0 \|_{H^2_x}}{t^\hf} + \int_0^t \frac{1}{(t-s)^\hf} \frac{\varepsilon^2}{\js} \, \ud s \\
 &\lesssim \frac{\log(1+\jt)}{t^\hf} \varepsilon.
\end{align*}
This proves the decay estimate~\eqref{equ:thm1_Linfty_decay_resonant}.

Next, we consider the asymptotic behavior of the solution $v(t)$. 
To this end, we first decompose Duhamel's formula for $v(t)$ as in~\eqref{equ:decomposition_v_for_Linfty_decay_nonresonant} and \eqref{equ:decomposition_v_for_Linfty_decay_nonresonant_delicate_term} to write for times $t \geq 1$,
\begin{equation} \label{equ:decomposition_v_for_asymptotics}
 \begin{aligned}
  v(t) &= e^{it\jtD} P_c v_0 + \frac{1}{2i} \int_0^1 e^{i(t-s)\jtD} \jtD^{-1} P_c \bigl( \alpha(\cdot) u(s)^2 \bigr) \, \ud s \\
  &\quad + \frac{1}{2i} \int_1^t e^{i(t-s)\jtD} \jtD^{-1} P_c \Bigl( \alpha(\cdot) \bigl( (v(s) + \bar{v}(s))^2 - (w(s) + \bar{w}(s))^2 \bigr) \Bigr) \, \ud s \\
  &\quad + c_0^2 \frac{a_0^2}{2} \int_1^t \Bigl( e^{i(t-s)\jtD} \jtD^{-1} P_c \bigl( \alpha \varphi^2 \bigr) \Bigr) \frac{e^{2is}}{s} \, \ud s \\
  &\quad + c_0^2 \frac{|a_0|^2}{i} \int_1^t \Bigl( e^{i(t-s)\jtD} \jtD^{-1} P_c \bigl( \alpha \varphi^2 \bigr) \Bigr) \frac{1}{s} \, \ud s \\
  &\quad - c_0^2 \frac{\bar{a}_0^2}{2} \int_1^t \Bigl( e^{i(t-s)\jtD} \jtD^{-1} P_c\bigl( \alpha \varphi^2 \bigr) \Bigr) \frac{e^{-2is}}{s} \, \ud s \\
  &\quad + \frac{1}{2i}  \int_1^t \Bigl( e^{i(t-s)\jtD} \jtD^{-1} P_c \bigl( \alpha \varphi^2 \bigr) \Bigr) \calO_{L^\infty_s}\Bigl( \frac{\varepsilon^2}{s^\thf}\Bigr) \, \ud s.
 \end{aligned}
\end{equation}
In what follows we show that the modified scattering behavior of the nonlinear solution $v(t)$ is caused by the fourth term on the right-hand side of~\eqref{equ:decomposition_v_for_asymptotics}, which we denote by
\begin{equation*}
 v_{mod}(t) := c_0^2 \frac{a_0^2}{2} \int_1^t \Bigl( e^{i(t-s)\jtD} \jtD^{-1} P_c \bigl( \alpha \varphi^2 \bigr) \Bigr) \frac{e^{2is}}{s} \, \ud s.
\end{equation*} 
We group all other terms in Duhamel's formula~\eqref{equ:decomposition_v_for_asymptotics} for $v(t)$ into 
\begin{equation*}
 v_{free}(t) := v(t) - v_{mod}(t).
\end{equation*}
Proceeding as in the proof of the decay estimate~\eqref{equ:thm1_Linfty_decay_nonresonant} for the non-resonant case, we obtain the asserted decay estimate~\eqref{equ:thm1_Linfty_decay_vfree} for $v_{free}(t)$ given by
\begin{equation}
 \|v_{free}(t)\|_{L^\infty_x} \lesssim \frac{\varepsilon}{\jt^\hf}, \qquad t \geq 1.
\end{equation}
Moreover, upon defining the scattering data 
\begin{align*}
 v_\infty &:= P_c v_0 + \frac{1}{2i} \int_0^1 e^{-i s \jtD} \jtD^{-1} P_c \bigl( \alpha(\cdot) u(s)^2 \bigr) \, \ud s \\
  &\quad + \frac{1}{2i} \int_1^\infty e^{-is \jtD} \jtD^{-1} P_c \Bigl( \alpha(\cdot) \bigl( (v(s) + \bar{v}(s))^2 - (w(s) + \bar{w}(s))^2 \bigr) \Bigr) \, \ud s \\
  &\quad + c_0^2 \frac{|a_0|^2}{i} \int_1^\infty \Bigl( e^{-i s \jtD} \jtD^{-1} P_c \bigl( \alpha \varphi^2 \bigr) \Bigr) \frac{1}{s} \, \ud s \\
  &\quad - c_0^2 \frac{\bar{a}_0^2}{2} \int_1^\infty \Bigl( e^{-is \jtD} \jtD^{-1} P_c\bigl( \alpha \varphi^2 \bigr) \Bigr) \frac{e^{-2is}}{s} \, \ud s \\
  &\quad + \frac{1}{2i}  \int_1^\infty \Bigl( e^{-i s \jtD} \jtD^{-1} P_c \bigl( \alpha \varphi^2 \bigr) \Bigr) \calO_{L^\infty_s}\Bigl( \frac{\varepsilon^2}{s^\thf}\Bigr) \, \ud s,
\end{align*}
we find by proceding as in the non-resonant case that $v_\infty \in H^2_x$ and that $v_{free}(t)$ scatters in $H^2_x$ to a free Klein-Gordon wave in the sense that 
\begin{equation*}
 \bigl\| v_{free}(t) - e^{it\jtD} v_\infty \bigr\|_{H^2_x} \lesssim \frac{\varepsilon^2}{\jt^\hf}, \qquad t \geq 1.
\end{equation*}

\medskip 

Finally, we analyze the asymptotic behavior of $v_{mod}(t)$ for $t \gg 1$. Here we follow relatively closely the corresponding derivation in the proof of Theorem~1.1 in~\cite{LLS2}.
In what follows we use the short-hand notation
\begin{equation*}
 Y := \alpha \varphi^2.
\end{equation*}
Let $\psi \in C^\infty(\bbR)$ be a smooth bump function such that $\psi(\xi) = 1$ in a small neighborhood around $\xi = 0$ and such that 
\begin{equation} \label{equ:definition_bump_psi}
 \psi(\xi) = 0 \quad \text{for } |\xi| \geq \tilde{\delta}
\end{equation}
for some small $\tilde{\delta} \equiv \tilde{\delta}(\delta) > 0$, whose size will be specified further below.
Then we decompose the distorted Fourier transform $\wtilY(\xi)$ of $Y$ into 
\begin{equation*}
 \wtilY(\xi) = \wtilY_+(\xi) + \wtilY_-(\xi) + \wtilY_{nr}(\xi)
\end{equation*}
with 
\begin{equation*}
 \wtilY_\pm(\xi) := \psi(\xi \mp \sqrt{3}) \wtilY(\xi).
\end{equation*}
Correspondingly, we define for times $t \geq 1$,
\begin{align}
 v_{mod, \pm}(t) &:= c_0^2 \frac{a_0^2}{2} \int_1^t \Bigl( e^{i(t-s)\jtD} \jtD^{-1} P_c Y_\pm \Bigr) \frac{e^{2is}}{s} \, \ud s,  \label{equ:resonant_def_v_modpm} \\
 v_{mod, nr}(t) &:= c_0^2 \frac{a_0^2}{2} \int_1^t \Bigl( e^{i(t-s)\jtD} \jtD^{-1} P_c Y_{nr} \Bigr) \frac{e^{2is}}{s} \, \ud s.
\end{align}

\medskip 

\noindent {\it Decay of $v_{mod, nr}(t)$}:
Since $\wtilY_{nr}(\pm \sqrt{3}) = 0$ by construction, we can integrate by parts in time $s$ in the Duhamel integral for $v_{mod, nr}(t)$. Then using the standard dispersive decay estimate for the Klein-Gordon propagator from Lemma~\ref{lem:pw decay}, we obtain uniformly for all $t \geq 1$ that
\begin{align*}
 \|v_{mod,nr}(t)\|_{L^\infty_x} &\lesssim \frac{\varepsilon^2}{t^\hf}.
\end{align*}

\medskip 

\noindent {\it Decay of $v_{mod, \pm}(t)$ away from small conic neighborhoods of $x = \pm \frac{\sqrt{3}}{2} t$}:
It suffices to consider $v_{mod,+}(t)$, the treatment of $v_{mod,-}(t)$ being analogous. Assume that $x \geq 0$. Using the distorted Fourier transform and noting that $\wtilY_+(\xi)$ is supported on $(0,\infty)$, we write 
\begin{equation} \label{equ:representation_vmodplus_away_conic}
 \begin{aligned}
  v_{mod,+}(t,x) 
  &= c_0^2 \frac{a_0^2}{2 \sqrt{2\pi}} \int_1^t \int_\bbR T(\xi) m_+(x,\xi) e^{i( x\xi + (t-s)\jxi)} \jxi^{-1} \wtilY_+(\xi) \, \ud \xi \, \frac{e^{2is}}{s} \, \ud s.
 \end{aligned}
\end{equation}
The phase 
\begin{equation*}
 \phi(s,\xi; t,x) := x\xi + (t-s)\jxi
\end{equation*}
satisfies 
\begin{align*}
 \partial_\xi \phi(s, \xi; t, x) = x + (t-s) \frac{\xi}{\jap{\xi}}, \qquad \partial_\xi^2 \phi(s, \xi; t,x) = \frac{t-s}{\jap{\xi}^3}. 
\end{align*}
For any given $0 < \delta \ll 1$, we may choose the constant $\tilde{\delta} \equiv \tilde{\delta}(\delta) > 0$ in the definition~\eqref{equ:definition_bump_psi} of the cut-off funtion $\psi$ above so small such that
\begin{align*}
 \Bigl| \frac{\xi}{\jap{\xi}} - \Bigl( \pm \frac{\sqrt{3}}{2} \Bigr) \Bigr| \leq \frac{\delta}{2} \quad \text{ whenever } \quad \wtilY_+(\xi) \neq 0.
\end{align*}
Moreover, we have $|\partial_\xi^2 \phi(s,\xi; t,x)| \simeq (t-s)$ on the support of $\wtilY(\xi)$.
We distinguish two cases. 
Suppose $x \geq \bigl( \frac{\sqrt{3}}{2} + \delta \bigr) t$. Then on the support of $\wtilY_+(\xi)$ the phase satisfies 
\begin{align*}
 |\partial_\xi \phi| &\geq |x| - (t-s) \frac{|\xi|}{\jap{\xi}} \geq \Bigl( \frac{\sqrt{3}}{2} + \delta \Bigr) t - \Bigl( \frac{\sqrt{3}}{2} + \frac{\delta}{2} \Bigr) (t-s) \geq \frac{\delta}{2} t.
\end{align*}
Correspondingly, integrating by parts in $\xi$ and using Lemma~\ref{lem:m symb} as well as Lemma~\ref{lem:T}, we find 
\begin{align*}
 \bigl| v_{mod, +}(t,x) \bigr| \lesssim_{\delta, V} \int_1^t \frac{1}{t} \frac{\varepsilon^2}{s} \, \ud s \lesssim \frac{\varepsilon^2}{t^{1-}}.
\end{align*}
Now suppose $0 \leq x \leq \bigl( \frac{\sqrt{3}}{2} - \delta \bigr) t$. We divide the time integration interval into two subintervals $$[1, t] = [1, t_1] \cup [t_1, t],$$ where
\begin{align*}
 t_1 :=  \frac{\delta}{2 (\sqrt{3} + \delta)} t. 
\end{align*}
On the support of $\wtilY_+(\xi)$ the phase satisfies for $1 \leq s \leq t_1$ that
\begin{align*}
 |\partial_\xi \phi| &\geq t \frac{|\xi|}{\jap{\xi}} - |x| - s \frac{|\xi|}{\jap{\xi}} \geq t \Bigl( \frac{\sqrt{3}}{2} - \frac{\delta}{2} \Bigr) - t \Bigl( \frac{\sqrt{3}}{2} - \delta \Bigr) - \frac{\delta}{2 (\sqrt{3} + \delta)} t \Bigl( \frac{\sqrt{3}}{2} + \frac{\delta}{2} \Bigr) = \frac{\delta}{4} t.
\end{align*}
Integration by parts in $\xi$ therefore pays off for $1 \leq s \leq t_1$. Instead for times $s \geq t_1$ we can just use the usual $(t-s)^{-\frac{1}{2}}$ dispersive decay of the retarded Klein-Gordon propagator $e^{i(t-s)\jtD}$ from Lemma~\ref{lem:pw decay} and crudely bound $\frac{1}{s} \leq \frac{1}{t_1} \lesssim_\delta \frac{1}{t}$. Hence, in the case $0 \leq x \leq \bigl( \frac{\sqrt{3}}{2} - \delta \bigr) t$, we obtain that 
\begin{equation} 
 \begin{aligned}
  |v_{mod,+}(t,x)| &\lesssim_{\delta, V} \int_1^{t_1} \frac{1}{t} \frac{\varepsilon^2}{s} \, \ud s + \int_{t_1}^t \frac{1}{(t-s)^{\frac{1}{2}}}  \frac{\varepsilon^2}{t} \, \ud s \lesssim \frac{\varepsilon^2}{t^{1-}} + \frac{\varepsilon^2}{t^\hf} \lesssim \frac{\varepsilon^2}{t^\hf}.
 \end{aligned}
\end{equation}
If instead $x < 0$, we start from the representation~\eqref{equ:representation_vmodplus_away_conic} for $v_{mod, +}(t,x)$ and first express $T(\xi) f_+(x,\xi)$ in terms of $f_{-}(x,\cdot)$ using~\eqref{eq:TR}. Then we may proceed as above.

This concludes the derivation of the decay estimate~\eqref{equ:thm1_resonant_decay_off_rays} for $v_{mod}(t)$ away from small conic neighborhoods of the rays $x = \pm \frac{\sqrt{3}}{2} t$, as asserted in the statement of Theorem~\ref{thm:thm1}.

\medskip 

\noindent {{\it Asymptotics of $v_{mod, \pm}(t,x)$ along the rays $x = \pm \frac{\sqrt{3}}{2} t$}:}
We consider $v_{mod,-}(t,x)$ in detail, noting that the treatment of $v_{mod,+}(t,x)$ proceeds analogously. First, we may restrict the time integration in the definition of $v_{mod,-}(t,x)$ to times $1 \leq s \leq t-1$ at the expense of picking up a remainder term of order $\calO_{L^\infty_t}\bigl( \varepsilon^2 t^{-1} \bigr)$. Moreover, by Lemma~\ref{lem:asymptotics_KG} on the asymptotics of the Klein-Gordon propagator (and observing that $t^{-\hf} \jap{\xi_0}^\thf \jap{\xi_0}^{-1} = \rho^{-\hf}$ for $\xi_0$ as in the statement of Lemma~\ref{lem:asymptotics_KG}), we have for $1 \leq s \leq t-1$ that 
\begin{equation} \label{equ:resonant_asymptotics_retarded_free_along_ray}
 \begin{aligned}
  &\Bigl( e^{i(t-s)\jtD} \jtD^{-1} P_c Y_{-} \Bigr)\Bigl( \pm \frac{\sqrt{3}}{2} t \Bigr) \\
  &\qquad = \frac{e^{i \frac{\pi}{4}} e^{i\rho(t-s,\pm \frac{\sqrt{3}}{2} t )}}{\rho(t-s, \pm \frac{\sqrt{3}}{2} t )^{\frac{1}{2}}}  \wtilY_-\biggl( - \frac{\pm \frac{\sqrt{3}}{2} t}{\rho(t-s,\pm \frac{\sqrt{3}}{2} t )} \biggr) \one_{(-1,1)}\biggl( \frac{x}{t-s} \biggr) + \frac{1}{(t-s)^{\frac{2}{3}}} \calO \bigl( \| \jx Y_{-} \|_{H^2_x} \bigr),
 \end{aligned}
\end{equation}
where 
\begin{align*}
 \rho(t-s, {\textstyle \pm \frac{\sqrt{3}}{2} t}) &= \bigl( (t-s)^2 - {\textstyle \frac{3}{4}} t^2 \bigr)^{\frac{1}{2}} = {\textstyle \frac{t}{2}} \bigl( 1 - 8 {\textstyle \frac{s}{t}} + 4  {\textstyle \frac{s^2}{t^2}} \bigr)^{\frac{1}{2}}. 
\end{align*}
Inserting the asymptotics~\eqref{equ:resonant_asymptotics_retarded_free_along_ray} into~\eqref{equ:resonant_def_v_modpm} gives
\begin{equation*}
 v_{mod, -}\Bigl(t, \pm \frac{\sqrt{3}}{2} t \Bigr) = c_0^2 \frac{a_0^2}{2} \int_1^{t-\frac{\sqrt{3}}{2} t} \frac{e^{i \frac{\pi}{4}} e^{i\rho(t-s,\pm \frac{\sqrt{3}}{2} t )}}{\rho(t-s, \pm \frac{\sqrt{3}}{2} t )^{\frac{1}{2}}}  \wtilY_-\biggl( - \frac{\pm \frac{\sqrt{3}}{2} t}{\rho(t-s,\pm \frac{\sqrt{3}}{2} t )} \biggr) \frac{e^{2is}}{s} \, \ud s + \calO_{L^\infty_t} \Bigl( \frac{\varepsilon^2}{t^{\frac{2}{3}-}} \Bigr).
\end{equation*}
Since $\wtilY_-(\xi) = 0$ for $\xi > 0$, we have along the ray $x = -\frac{\sqrt{3}}{2} t$ that
\begin{equation*}
 v_{mod, -}\Bigl(t, - \frac{\sqrt{3}}{2} t \Bigr) = \calO_{L^\infty_t} \Bigl( \frac{\varepsilon^2}{t^{\frac{2}{3}-}} \Bigr).
\end{equation*}
Moreover, due to the sharp localization of the frequency support of $\wtilY_-(\xi)$ around $\xi = -\sqrt{3}$, for $t \gg 1$ the time integration in the last identity for $v_{mod,-}(t, \frac{\sqrt{3}}{2} t)$ is in fact only over an interval $1 \leq s \leq ct$ for some small constant $0 < c \ll 1$. Thus, along the ray $x = \frac{\sqrt{3}}{2} t$ one has that
\begin{equation} \label{equ:resonant_v_mod_minus_along_plus_ray}
 v_{mod,-}\Bigl(t, \frac{\sqrt{3}}{2} t \Bigr) = c_0^2 \frac{a_0^2}{2} \int_1^{ct} \frac{e^{i \frac{\pi}{4}} e^{i\rho(t-s, \frac{\sqrt{3}}{2} t )}}{\rho(t-s, \frac{\sqrt{3}}{2} t )^{\frac{1}{2}}}  \wtilY_-\biggl( - \frac{\frac{\sqrt{3}}{2} t}{\rho(t-s, \frac{\sqrt{3}}{2} t )} \biggr) \frac{e^{2is}}{s} \, \ud s + \calO_{L^\infty_t} \Bigl( \frac{\varepsilon^2}{t^{\frac{2}{3}-}} \Bigr).
\end{equation}
In view of the approximate identities
\begin{align*}
 - \frac{ \frac{\sqrt{3}}{2} t }{ \rho(t-s,  \frac{\sqrt{3}}{2} t) } &= -\sqrt{3} + \calO \Bigl( \frac{s}{t} \Bigr), \qquad \frac{1}{\rho(t-s, \frac{\sqrt{3}}{2} t)^{\frac{1}{2}}} = \frac{\sqrt{2}}{t^{\frac{1}{2}}} + \calO \Bigl( \frac{s}{t^{\frac{3}{2}}} \Bigr), 
\end{align*}
it follows that 
\begin{equation} \label{equ:resonant_v_mod_minus_along_plus_ray_almost_done}
 v_{mod,-}\Bigl(t, \frac{\sqrt{3}}{2} t \Bigr) = c_0^2 \frac{a_0^2}{\sqrt{2}} e^{i\frac{\pi}{4}} \wtilY(-\sqrt{3}) \frac{1}{t^{\frac{1}{2}}} \int_1^{ct} e^{i(\rho(t-s,\frac{\sqrt{3}}{2} t) + 2s)} \frac{1}{s} \, \ud s + \calO_{L^\infty_t} \Bigl( \frac{\varepsilon^2}{t^\hf} \Bigr).
\end{equation}
Now we observe that the phase 
\begin{equation*}
 \phi(s;t) := \rho \Bigl(t-s,\frac{\sqrt{3}}{2} t\Bigr) + 2s
\end{equation*}
is stationary at $s=0$ and that its Taylor expansion about $s=0$ is of the form
\begin{equation*}
 \phi(s;t) = \frac{t}{2} + \calO \Bigl( \frac{s^2}{t} \Bigr).
\end{equation*}
Thus, for times $1 \leq s \ll t^\hf$ the phase $\phi(s;t)$ is essentially constant and the integrand in~\eqref{equ:resonant_v_mod_minus_along_plus_ray_almost_done} is effectively monotone, which leads to the buildup of a $\log(t)$ factor. In order to arrive at a sharp formula for the asymptotics, we split the time integration interval into the two subintervals $1 \leq s \leq 10^{-3} t^{\frac{1}{2}}$ and $10^{-3} t^{\frac{1}{2}} \leq s \leq c t$. 
For the interval $1 \leq s \leq 10^{-3} t^\hf$ we compute that
\begin{align*}
 \int_1^{10^{-3} t^{\frac{1}{2}}} e^{i \phi(s;t)} \frac{1}{s} \, \ud s &= e^{i \frac{t}{2}} \int_1^{10^{-3} t^{\frac{1}{2}}} \frac{1}{s} \, \ud s + \int_1^{10^{-3} t^{\frac{1}{2}}} \calO \Bigl( \frac{s}{t} \Bigr) \, \ud s = \frac{e^{i \frac{t}{2}}}{2} \log(t) + \calO(1).
\end{align*}
Instead, on the interval $10^{-3} t^{\frac{1}{2}} \leq s \leq c t$ we integrate by parts. Since $\partial_s \phi(s;t) = \calO \bigl( \frac{s}{t} \bigr)$ and $\partial_s^2 \phi(s;t) = \calO \bigl( \frac{1}{t} \bigr)$ on that time integration interval, we find 
\begin{equation*}
 \biggl| \int_{10^{-3} t^{\frac{1}{2}}}^{c t} e^{i \phi(s;t)} \frac{1}{s} \, \ud s \biggr| \lesssim \int_{10^{-3} t^{\frac{1}{2}}}^{ct} \frac{t}{s^3} \, \ud s + \biggl| \frac{t}{s^2} \Bigr|_{s=10^{-3} t^\hf}^{s=ct} \biggr| \lesssim 1.
\end{equation*}
Hence, we obtain the asymptotics
\begin{equation*}
 v_{mod,-}\Bigl(t, \frac{\sqrt{3}}{2} t \Bigr) = c_0^2 \frac{a_0^2}{\sqrt{8}} e^{i\frac{\pi}{4}} e^{i \frac{t}{2}} \widetilde{\calF}[\alpha \varphi^2](-\sqrt{3}) \frac{\log(t)}{t^{\frac{1}{2}}} + \calO_{L^\infty_t} \Bigl( \frac{\varepsilon^2}{t^\hf} \Bigr), \quad t \gg 1.
\end{equation*}
This finishes the proof of Theorem~\ref{thm:thm1}.

\begin{remark} \label{rem:profile_log_divergence}
In the resonant case when $a_0 \neq 0$ the distorted Fourier transform of the profile $g(t) := e^{-it\jtD} v(t)$ of the solution $v(t)$ to \eqref{equ:thm1_nlkg} diverges logarithmically at the frequencies $\xi = \pm \sqrt{3}$, specifically we have 
\begin{equation*}
 \begin{aligned}
  \widetilde{g}(t, \pm \sqrt{3}) = c_0^2 \frac{a_0^2}{4} \widetilde{\calF}[ \alpha \varphi^2 ](\pm \sqrt{3}) \log(t) + \calO(\varepsilon), \quad t \gg 1.
 \end{aligned}
\end{equation*}
To see this, recall that $v_{mod}(t)$ is the fourth term on the right-hand side of the decomposition~\eqref{equ:decomposition_v_for_asymptotics} of the Duhamel formula for $v(t)$. The distorted Fourier transform of the profile $g_{mod}(t) := e^{-it\jtD} v_{mod}(t)$ of $v_{mod}(t)$ is given by
\begin{equation*}
 \begin{aligned}
  \widetilde{g}_{mod}(t,\xi) = c_0^2 \frac{a_0^2}{2} \jxi^{-1} \widetilde{\calF}[ \alpha \varphi^2 ](\xi) \int_1^t e^{is(2-\jxi)} \frac{1}{s} \, \ud s.
 \end{aligned}
\end{equation*}
Since $2-\jxi = 0$ for $\xi = \pm \sqrt{3}$, we correspondingly obtain that $\widetilde{g}_{mod}(t,\xi)$ diverges logarithmically at the frequencies $\xi = \pm \sqrt{3}$, 
\begin{equation*}
 \begin{aligned}
  \widetilde{g}_{mod}(t, \pm \sqrt{3}) = c_0^2 \frac{a_0^2}{4} \widetilde{\calF}[ \alpha \varphi^2 ](\xi) \int_1^t \frac{1}{s} \, \ud s = c_0^2 \frac{a_0^2}{4} \widetilde{\calF}[ \alpha \varphi^2 ](\xi) \log(t).
 \end{aligned}
\end{equation*}
The contributions of all other terms on the right-hand side of the decomposition~\eqref{equ:decomposition_v_for_asymptotics} to the distorted Fourier transform $\tilde{g}(t,\xi)$ of the profile are uniformly bounded in time (at all frequencies), which follows readily using the local decay bounds for $v(t)$ from Proposition~\ref{prop:local_decay_bounds}. For the contributions of the fifth and sixth terms one additionally has to exploit the oscillations. 
\end{remark}

\bibliographystyle{amsplain}
\bibliography{references}

\end{document}